\documentclass[12pt,reqno]{amsart}

\usepackage[margin=1in]{geometry}
\usepackage{amsmath}
\usepackage{amssymb}
\usepackage{amsthm}
\usepackage{dsfont}
\usepackage{graphicx} 
\usepackage{microtype}

\usepackage{seqsplit}

\usepackage[cal=euler]{mathalpha}

\usepackage{tikz}

\usepackage[dvipsnames]{xcolor}
\usepackage[backref=page]{hyperref}
\hypersetup{
    colorlinks=true,
    linkcolor=cyan!80!black,
    citecolor=MidnightBlue,
    urlcolor=magenta,
}

\usepackage{tikz}
\usepackage{natbib}

\usepackage{apptools}
\AtAppendix{\counterwithin{lemma}{section}}

\newcommand{\bbE}{\mathbb{E}}
\newcommand{\bbC}{\mathbb{C}}
\newcommand{\bbN}{\mathbb{N}}
\newcommand{\bbP}{\mathbb{P}}
\newcommand{\bbR}{\mathbb{R}}
\newcommand{\bbS}{\mathbb{S}}

\newcommand{\bX}{\mathbf{X}}
\newcommand{\bY}{\mathbf{Y}}

\newcommand{\bV}{\mathbf{V}}
\newcommand{\bW}{\mathbf{W}}

\newcommand{\be}{\mathbf{e}}

\newcommand{\bg}{\mathbf{g}}
\newcommand{\bx}{\mathbf{x}}
\newcommand{\by}{\mathbf{y}}
\newcommand{\bv}{\mathbf{v}}
\newcommand{\bw}{\mathbf{w}}
\newcommand{\bu}{\mathbf{u}}
\newcommand{\bs}{\mathbf{s}}
\newcommand{\bt}{\mathbf{t}}
\newcommand{\bff}{\mathbf{f}}

\newcommand{\bone}{\mathbf{1}}
\newcommand{\bzero}{\mathbf{0}}

\newcommand{\cA}{\mathcal{A}}
\newcommand{\cH}{\mathcal{H}}

\newcommand{\cN}{\mathcal{N}}
\newcommand{\cP}{\mathcal{P}}
\newcommand{\cF}{\mathcal{F}}

\newcommand{\cS}{\mathcal{S}}
\newcommand{\cT}{\mathcal{T}}
\newcommand{\cM}{\mathcal{M}}

\newcommand{\cG}{\mathcal{G}}

\newcommand{\fS}{\mathfrak{S}}

\newcommand{\Ber}{\mathsf{Ber}}
\newcommand{\Binomial}{\mathsf{Binomial}}
\newcommand{\Multinomial}{\mathsf{Multinomial}}

\newcommand{\BL}{\mathrm{BL}}

\newcommand{\ind}{\mathds{1}}
\newcommand{\Mat}{\mathrm{Mat}}
\newcommand{\Lip}{\mathrm{Lip}}
\renewcommand{\sc}{\mathsf{sc}}

\newcommand{\Gauss}{\mathsf{G}}

\newcommand{\op}{\mathrm{op}}
\newcommand{\mubar}{\bar{\mu}}

\newcommand{\sign}{\mathrm{sign}}

\DeclareMathOperator{\Var}{Var}
\DeclareMathOperator{\Tr}{Tr}

\DeclareMathOperator{\Cov}{Cov}
\newcommand{\perms}{\mathsf{Perm}}

\newcommand{\dist}{\mathrm{dist}}
\newcommand{\spn}{\mathrm{span}}

\newcommand{\la}{\langle}
\newcommand{\ra}{\rangle}

\newcommand{\GOE}{\mathsf{GOE}}
\newcommand{\GOTE}{\mathsf{GOTE}}

\renewcommand{\vec}{\mathrm{vec}}
\newcommand{\vech}{\mathrm{vech}}

\newcommand{\convas}{\xrightarrow{\text{a.s.}}}
\newcommand{\convp}{\xrightarrow{p}}
\newcommand{\convd}{\xrightarrow{d}}

\theoremstyle{plain}
\newtheorem{theorem}{Theorem}[section]
\newtheorem{corollary}{Corollary}[section]
\newtheorem{proposition}{Proposition}[section]
\newtheorem{lemma}{Lemma}[section]

\theoremstyle{definition}

\newtheorem{definition}{Definition}[section]

\theoremstyle{remark}
\newtheorem{remark}{Remark}[section]

\let\tilde\widetilde

\title[Spectra of Contractions of Tensor-GOE]{Spectra of contractions of the Gaussian Orthogonal \\ Tensor Ensemble}
\author[S. S. Mukherjee and H. Talukdar]{Soumendu Sundar Mukherjee and Himasish Talukdar}
\address{
    Statistics and Mathematics Unit \\
    Indian Statistical Institute \\
    203 B.T. Road, Kolkata 700108 \\
    West Bengal, India.
}
\email{ssmukherjee@isical.ac.in, talukdar.himasish@gmail.com}

\begin{document}

\begin{abstract}
In this article, we study the spectra of matrix-valued contractions of the Gaussian Orthogonal Tensor Ensemble (GOTE). Let $\cG$ denote a random tensor of order $r$ and dimension $n$ drawn from the density
\[
    f(\cG) \propto \exp\bigg(-\frac{1}{2r}\|\cG\|^2_{\mathrm{F}}\bigg).
\]
For $\bw \in \bbS^{n - 1}$, the unit sphere in $\bbR^n$, we consider the matrix-valued contraction $\cG \cdot \bw^{\otimes (r - 2)}$ when both $r$ and $n$ go to infinity such that $r / n \to c \in [0, \infty]$. We obtain semicircle bulk-limits in all regimes, generalising the works of \cite{goulart2022random, au2023spectral, bonnin2024universality} in the fixed-$r$ setting. Under a delocalisation condition on the contraction direction, we prove universality of the limiting spectral measure.

We further study the edge-spectrum. We obtain a Baik--Ben Arous--P\'{e}ch\'{e} phase transition for the largest and the smallest eigenvalues for $r \ge 4$, generalising a result of \cite{mukherjee2024spectra} in the context of adjacency matrices of random hypergraphs. For $r = 3$, the extreme eigenvalues stick to the edges of the support of the semicircle law, while for $r \ge 4$, two outlier eigenvalues emerge. Furthermore, for $r = 3$, the fluctuations of the extreme eigenvalues are given by the eigenvalues of a stochastic Airy operator, whereas for $r \ge 4$, the fluctuations are Gaussian. We show that for $r \ge 4$, the extreme eigenvectors of $\cG \cdot \bw^{\otimes (r - 2)}$ contain non-trivial information about the contraction direction $\bw$. In fact, in each of the regimes $1 \ll r \ll n$ and $r \gg n$, one may identify two explicit (data-dependent) vectors, one of which tends to become perfectly aligned with $\bw$. A universality result for edge eigenvalues is established under certain delocalisation conditions on $\bw$ in the regime $r = o(\frac{n}{\log n})$.

In the fixed-$r$ regime, we derive the limiting directional spectral measure, where, remarkably, the Kesten--McKay law makes an appearance. We additionally obtain Gaussian limit theorems for linear spectral statistics for fixed $r$.  

Finally, we report some results, in the case $r = 4$, on mixed contractions $\cG \cdot \bu \otimes \bv$ for $\bu, \bv \in \bbS^{n - 1}$. While the total variation distance between the joint distributions of the entries of $\cG \cdot \bu \otimes \bv$ and $\cG \cdot \bu \otimes \bu$ goes to $0$ when $\|\bu - \bv\| = o(n^{-1})$, the bulk and the extreme eigenvalues of these two matrices share the same limit profile as long as $\|\bu - \bv\| = o(1)$. Furthermore, it turns out that there are no outlier eigenvalues in the spectrum of $\cG \cdot \bu \otimes \bv$ when the contraction directions $\bu$ and $\bv$ are asymptotically orthogonal.
\end{abstract}


\maketitle
\thispagestyle{empty}


\section{Introduction}
A (real) tensor of \emph{order} $r$ and dimension $n$ is an element of $(\bbR^n)^{\otimes r}$ and can be represented conveniently as a multi-dimensional array of real numbers with respect to a chosen basis of $\bbR^n$. In this article, we always work with such a concrete representation as a multi-dimensional array and by an abuse of terminology call such arrays tensors.
A matrix is thus an order-$2$ tensor. A tensor $\cT$ is called \emph{symmetric} (also called super-symmetric by some authors) if for any permutation $\sigma$ of $[r] := \{1, \ldots, r\}$, one has
\[
    \cT_{i_1, \ldots, i_r} = \cT_{i_{\sigma(1)}, \ldots, i_{\sigma(r)}}.
\]
Let $\fS_{r, n}$ denote the set of all symmetric tensors of order $r$ and dimension $n$. Equipped with the (Frobenius) inner product
\[
    \langle \cS, \cT\rangle_{\mathrm{F}} := \sum_{(i_1, \ldots, i_r) \in [n]^r} \cS_{i_1, \ldots, i_r} \cT_{i_1, \ldots, i_r}, \quad \cS, \cT \in \fS_{r, n},
\]
it becomes a (finite-dimensional) Hilbert space. The Frobenius norm of a tensor $\cT \in \fS_{r, n}$ is $\|\cT\|_{\mathrm{F}} := \sqrt{\langle \cT, \cT\rangle_{\mathrm{F}}}$.

A symmetric random tensor $\cG \in \fS_{r, n}$ is said to belong to the \emph{Gaussian Orthogonal Tensor Ensemble (GOTE)} if, as a tensor-valued random variable, it has the following density with respect to the natural Lebesgue measure on $\fS_{r, n}$:
\[
    f(\cG) = \frac{1}{Z_{n, r}} \exp\bigg(-\frac{1}{2r}\|\cG\|^2_{\mathrm{F}}\bigg),
\]
where $Z_{n, r} := \int \exp\big(-\frac{1}{2r}\|\cG\|^2_{\mathrm{F}}\big)$.
We shall write that $\cG \sim \GOTE(r, n)$. 

If $\cG \sim \GOTE(r, n)$, then it may be shown that
$\cG_{i_1, \ldots, i_r} \sim \cN(0, \sigma_{i_1, \ldots, i_r}^2)$, where
\[
    \sigma_{i_1, \ldots, i_r}^2 = \frac{r}{\#\perms(i_1, \ldots, i_r)},
\]
with $\perms(i_1, \ldots, i_r)$ denoting the set of all permutations of the index vector $(i_1, \ldots, i_r)$, viewed as a multiset. Notice that $\GOTE(2, n)$ is the familiar Gaussian Orthogonal Ensemble (GOE) from random matrix theory.

In this article, we are interested in the spectra of matrix-valued contractions of GOTE tensors. For $\bw^{(1)}, \ldots, \bw^{(r - 2)} \in 
\bbS^{n - 1} := \{\bu \in \bbR^n : \|\bu\|_2 = 1\}$, consider the matrix
\begin{equation}\label{eq:mixed_contraction}
    M_n = \cG \cdot \bw^{(1)} \otimes \cdots \otimes \bw^{(r -2)},
\end{equation}
where
\[
    M_{ij} \equiv M_{n, ij} = \sum_{i_3, \cdots, i_r} \cG_{iji_3\cdots i_r}  w^{(1)}_{i_3} \cdots w^{(r - 2)}_{i_r}.
\]
When the \emph{contraction directions} $\bw^{(1)}, \ldots, \bw^{(r - 2)}$ are all the same, say $\bw^{(1)} = \cdots = \bw^{(r - 2)} = \bw$, we write
\begin{equation}\label{eq:pure_contraction}
    M_n = \cG \cdot \bw^{\otimes (r - 2)}.
\end{equation}
We refer to \eqref{eq:pure_contraction} as a \emph{pure} contraction as opposed to \eqref{eq:mixed_contraction} which we call a \emph{mixed} contraction. Figure~\ref{fig:pure-vx-mixed} shows the spectra of some examples of pure and mixed contractions.
\begin{figure}[!t]
    \centering
    \begin{tabular}{cc}
    \includegraphics[width=0.45\linewidth]{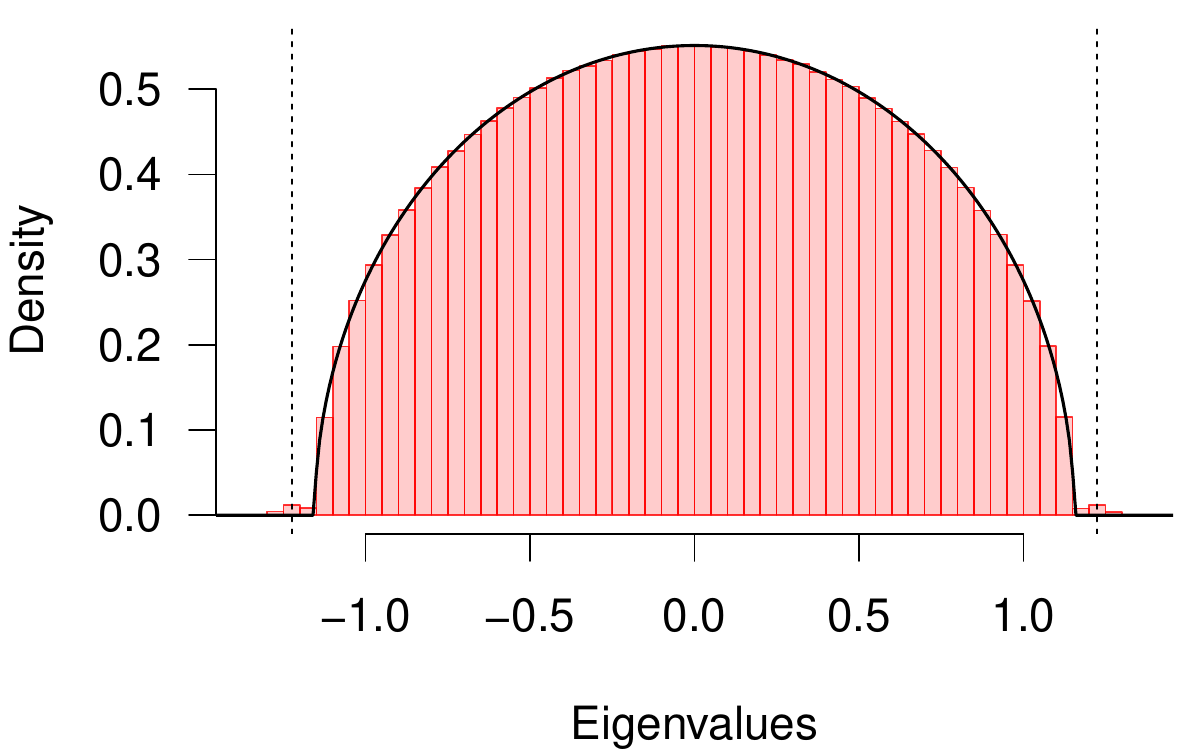} &
    \includegraphics[width=0.45\linewidth]{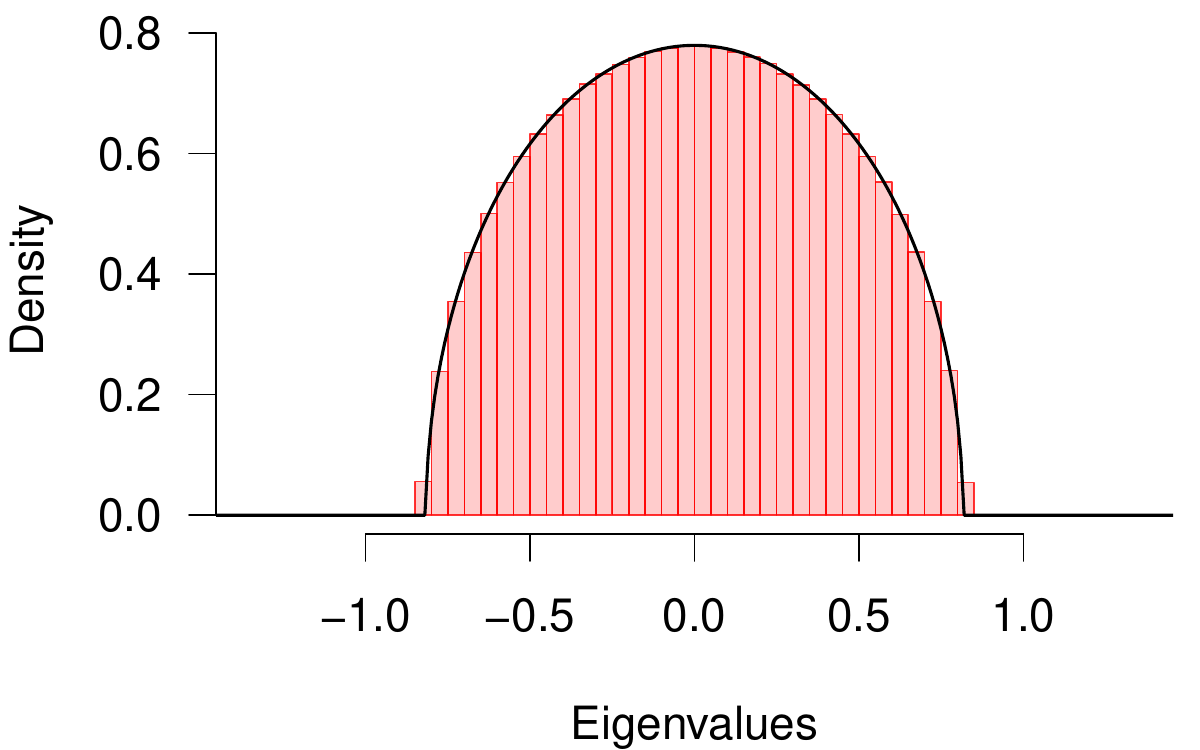} 
    \end{tabular}
    \caption{Histograms of the eigenvalues of $n^{-1/2}\cG \cdot \be_1 \otimes \be_1$ (left) and $n^{-1/2}\cG \cdot~\be_1 \otimes \be_2$ (right), where $n = 2500$, $\cG \sim \GOTE(4, n)$ (based on $100$ replications). The solid (blue) curves depict the corresponding semicircle densities. Note that there are two outlier eigenvalues in the spectrum of $n^{-1/2}\cG \cdot \be_1 \otimes \be_1$, at the locations $\pm \frac{1}{\sqrt{3}}\left(\sqrt{2} + \frac{1}{\sqrt{2}}\right)$, marked by the dotted vertical lines. There are none in the spectrum of $n^{-1/2}\cG \cdot \be_1 \otimes \be_2$. These observations are proved in Theorems~\ref{thm:edge-limits} and \ref{thm:mixed_edge}, respectively.}
    \label{fig:pure-vx-mixed}
\end{figure}

Besides their inherent mathematical appeal, tensor contractions show up naturally in various contexts. For instance, they appear when analysing the maximum likelihood estimator in tensor principal component analysis (see, e.g., \cite{goulart2022random}).

Pure contractions also arise as the Hessian of the Hamiltonian of an $r$-spin spherical spin-glass (see, e.g., \cite{auffinger2013random}). Let $\cT \in \fS_{r, n}$ denote a tensor consisting of i.i.d. standard Gaussian elements. The Hamiltonian of an $r$-spin spherical spin-glass with disorder tensor $\cT$ at spin configuration $\bw \in \bbS^{n - 1}$ is given by 
\[
    H_{n, r}(\bw) = \sqrt{n} \sum_{i_1, \ldots, i_r = 1}^n w_{i_1} \cdots w_{i_r} \cT_{i_1, \ldots, i_r}.
\]
Its Hessian is proportional to the pure contraction of $\cT$ in the direction $\bw$: 
\[
    \nabla^2_{\bw} H_{n, r}(\bw) = r(r - 1)\sqrt{n} \,\, \cT \cdot \bw^{\otimes (r - 2)}.
\]
An important part of the analysis in \cite{auffinger2013random} for studying the critical values of such spin-glasses via the Kac-Rice formula was an understanding of the Hessian at critical points. We shall expound on this point later in Remark~\ref{rem:comp-with-spherical-spin-glass}, but we note that in the analysis in \cite{auffinger2013random}, the fact that the disorder tensor has i.i.d. standard Gaussian elements was crucial, because it makes the distribution of the Hessian independent of the contraction direction so that one immediately obtains the conditional distribution of the Hessian at critical points. We are, however, interested in contractions of GOTE tensors whose different variance profile makes a marked difference. In particular, the contractions retain non-trivial information about $\bw$.

In another direction, adjacency matrices of hypergraphs, which are widely used in hypergraph algorithms, are precisely contractions of the corresponding adjacency tensors along the direction $\bone / \sqrt{n}$. Another example comes from taking $\bw = \be_1$, in which case the entries of the pure contraction \eqref{eq:pure_contraction} are given by
\[
   M_{ij} = \sum_{i_3, \ldots,  i_r} \cG_{i j i_3 \cdots i_r} \delta_{i_3 1} \cdots \delta_{i_r 1} = \cG_{ij1 \cdots 1}.
\]
This is a (Gaussian) Wigner-type matrix with the following variance profile:
\begin{align*}
    \Var(M_{ij}) &= \frac{1}{r - 1} \ind(1 \notin \{i, j\}) + \ind(1 \in \{i, j\}), \\
    \Var(M_{ii}) &= \frac{2}{r - 1} \ind(i \ne 1) + r \ind(i = 1).
\end{align*}
Such Wigner-type matrices have been studied in great detail in recent years (see, e.g., \cite{ajanki2016universality}).

However, in general, the entries of $\cG \cdot \bw^{\otimes (r - 2)}$ are highly correlated, rendering their analysis inaccessible via existing results on correlated Wigner matrices (such as the bulk results in \cite{pastur2011eigenvalue,  chakrabarty2013limiting, chakrabarty2016random, gotze2015limit, che2017universality,  erdHos2019random, catalano2024random} or the edge results in \cite{alt2020correlated, adhikari2019edge, reker2022operator, banerjee2024edge}). For more details on this, we refer the reader to the discussions in the introductions of \cite{au2023spectral, mukherjee2024spectra, banerjee2024edge}.

\subsection{Preliminaries on random matrices} 
Let $A_n$ be an $n \times n$ Hermitian matrix with ordered eigenvalues $\lambda_1 \ge \cdots \ge \lambda_n$. The probability measure 
\[
    \mu_{A_n} := \frac{1}{n} \sum_{i = 1}^n \delta_{\lambda_i}
\]
is called the \emph{Empirical Spectral Distribution} (ESD) of $A_n$. If entries of $A_n$ are random variables defined on a common probability space $(\Omega, \cA, \bbP)$ then $\mu_{A_n}$ is a random probability measure. In that case, there is another probability measure associated to the eigenvalues, namely the \textit{Expected Empirical Spectral Distribution} (EESD) of $A_n$, which is defined via its action on bounded measurable test functions $f$ as follows:
\[
    \int f \, d\mubar_{A_n} = \bbE \int f \, d\mu_{A_n},
\]
where $\bbE$ denotes expectation with respect to $\bbP$. In random matrix theory, one is typically interested in an ensemble $(A_n)_{n \ge 1}$ of such matrices of growing dimension $n$. If the weak limit, say $\mu_{\infty}$, of the sequence $(\mubar_{A_n})_{n \ge 1}$, exists, then it is referred to as the Limiting Spectral Distribution (LSD). Often one is able to show that the random measure $\mu_{A_n}$ also converges weakly (in probability or in almost sure sense) to $\mu_{\infty}$. For a comprehensive introductory account of the theory of random matrices, we refer the reader to \cite{anderson2010introduction}.

The preeminent model of random matrices is perhaps the Wigner matrix. For us a Wigner matrix $W_n$ will be a Hermitian random matrix whose upper triangular entries $W_{n, i, j}$ are i.i.d. zero mean unit variance random variables and the diagonal entries $W_{n, i, i}$ are i.i.d. zero mean random variables with finite variance. Moreover, the diagonal and the off-diagonal entries are mutually independent. If the entries are jointly Gaussian with the diagonal entries having variance $2$, then resulting ensemble of matrices is called the Gaussian Orthogonal Ensemble (GOE). In this article, we will denote the centered Gaussian distribution with variance $\sigma^2$ by $\nu_{\Gauss, \sigma^2}$.

We also define the \emph{semicircle distribution with variance $\sigma^2$}, henceforth denoted by $\nu_{\sc, \sigma^2}$, as the probability distribution on $\bbR$ with density
\[
    f(x) := \begin{cases}
    \frac{1}{2\pi \sigma^2} \sqrt{4 \sigma^2 - x^2} & \text{if} \, |x| \le 2 \sigma, \\
    0  & \text{otherwise}. 
\end{cases}
\]
E. Wigner proved in his famous paper \cite{wigner1958distribution} that the EESD of $n^{-1/2} W_n$ converges weakly to the standard semicircle distribution $\nu_{\sc, 1}$.

\subsection{Related works on tensor contractions.}
In \cite{gurau2020generalization}, R. Gurau considered the GOTE and showed that the limit of an appropriate generalization of the Stieltjes transform can be described by a generalised Wigner law, whose even moments are the Fuss-Catalan numbers. There has been a flurry of activity surrounding random tensors in the past few years \citep{goulart2022random, au2023spectral, bonnin2024universality,seddik2024random}. Among these, the most relevant to our setting are \cite{goulart2022random, au2023spectral, bonnin2024universality}. \cite{goulart2022random} considered pure contractions of a variant of the GOTE with $r$ fixed and showed that for any contraction direction $\bw \in \bbS^{n - 1}$, one gets a semicircle law as the LSD of the contracted matrix. \cite{au2023spectral} obtained the same result for general non-Gaussian entries. In fact, they also established joint convergence of a family of mixed contractions in the sense of free probability. The above mentioned semicircle LSD result also follows from the work of \cite{bonnin2024universality} who studied the more general tensor-valued contractions of the GOTE, obtaining the same distribution as \cite{gurau2020generalization}.
In the recent work \cite{mukherjee2024spectra}, spectral properties of adjacency matrices of random Erd\H{o}s-R\'{e}nyi hypergraphs were studied. These matrices are pure contractions of the underlying adjacency tensors along the direction $\bw = \frac{\bone}{\sqrt{n}}$. Finally, we note that LSDs of certain symmetric block matrices composed of contractions of a hyper-rectangular spiked Gaussian random tensor of a fixed order have been obtained in \cite{seddik2024random}, in course of their analysis of asymmetric spiked tensor models.

\subsection{Our contributions.}
Contractions of the GOTE form a remarkable ensemble of random matrices. Here one encounters many of the main protagonists of random matrix theory: the semicircle law, the Tracy-Widom law, Baik--Ben Arous--P\'{e}ch\'{e} (BBP) transition, the stochastic Airy operator, and the Kesten--McKay law. We now summarise the main contributions of this article.

\subsubsection{Pure contractions} First consider pure contractions of GOTE as defined in \eqref{eq:pure_contraction}. 
\begin{itemize}\setlength{\itemsep}{10pt}
    \item \textbf{Limiting spectral measure.} In the earlier works \citep{goulart2022random, au2023spectral, bonnin2024universality}, the LSD for pure contractions were obtained for fixed $r$. In Theorem~\ref{thm:lsd_pure} of this paper, we obtain the LSD of \eqref{eq:pure_contraction} for both fixed and growing $r$, which turns out to be a scaled semicircle law. Our proof leverages a remarkable exact distributional representation for GOTE contractions (as a random low-rank perturbation of a scaled GOE matrix) and is much simpler than the existing approaches. In Theorem~\ref{thm:univ_LSD}, we also prove universality of the LSD under a delocalisation condition on the contraction direction $\bw$. We note that in \cite{au2023spectral}, universality was proved via the moment method for fixed $r$ only (albeit without any delocalisation condition).
    
    \item \textbf{Limiting directional spectral measure.} For fixed $r$, we also obtain, in Theorem~\ref{thm:spectral_measure}, the limiting directional spectral measure along a direction $\bx \in \bbS^{n - 1}$, the limit law being parameterised by the limiting overlap $\lim_{n \to \infty} \langle \bx, \bw \rangle$ of $\bx$ with the contraction direction $\bw$. The limit measure turns out to be a mixture of two point masses, the semicircle law and the Kesten--Mckay law (which famously appears as the LSD of random regular graphs). 
    
    \item \textbf{A BBP phase transition and fluctuations.} The above-mentioned distributional representation also allows us to study the spectral edge. In Theorems~\ref{thm:edge-limits} and \ref{thm:edge-fluctuations} we derive, respectively, the first and second order asymptotic behaviours of the edge eigenvalues of the pure contraction \eqref{eq:pure_contraction}. 
    
    Interestingly, a BBP phase transition occurs at $r = 3$. There are no outlier eigenvalues in the spectrum for $r \le 3$, but two outlier eigenvalues emerge for $r \ge 4$.
    
    We also observe a rich variety of fluctuations. Recall that the edge eigenvalues of GOE ($r = 2$) have Tracy-Widom fluctuations, after a $n^{2/3}$ scaling. We see that for $r = 3$, one still has fluctuation at scale $n^{-2/3}$ but the limit law is the law of the largest eigenvalue of an appropriate stochastic Airy operator \cite{bloemendal2016limits}. The stochastic Airy operator appears in the analysis of $\beta$-ensembles. In these models, at the soft edge, the corresponding tridiagonal operators converge to the stochastic Airy operator after appropriate scaling. Contractions of order-3 GOTE tensors thus furnish another natural ensemble exhibiting such fluctuations.
    
    For $r \ge 4$, the two outlier eigenvalues exhibit jointly Gaussian fluctuations. If $r$ grows with $n$, the fluctuation stays Gaussian after appropriate scaling.
    
    The above-mentioned results are new and would be quite difficult to obtain without the distributional representation mentioned above. In Theorem~\ref{thm:univ_edge}, we establish universality of the first order edge results for $r \ll \frac{n}{\log n}$ under some delocalisation assumptions on $\bw$.
    
    \item \textbf{Overlap of edge eigenvectors with the contraction direction.} Theorem~\ref{thm:overlap-with-w} describes the behaviour of the projections of the two extreme edge eigenvectors along the contraction direction $\bw$. Notably, in the regimes $1 \ll r \ll n$ and $r \gg n$, $\bw$ asymptotically aligns with either the two extreme edge eigenvectors or an explicit linear combination thereof.
    
    \item \textbf{Fluctuations of linear spectral statistics.} Finally, in Theorem~\ref{thm:lss}, we describe the fluctuations of the linear spectral statistics of GOTE contractions for fixed $r$.
\end{itemize} 

\subsubsection{Mixed contractions for \texorpdfstring{$r = 4$}{}} We also obtain several interesting results on mixed contractions for the case $r=4$, i.e. for the matrix $n^{-1/2} \cG \cdot \bu \otimes \bv$, where $\cG \sim \GOTE(4, n)$. The proofs are again based on an exact distributional representation (see Lemma~\ref{lem:cov-representation-mixed-4}).  
\begin{itemize}\setlength{\itemsep}{10pt}
    \item \textbf{Limiting spectral measure.} In Theorem~\ref{thm:mixed_bulk} we obtain the LSD which was earlier obtained in \cite{au2023spectral} via the moment method.
    \item \textbf{Conditions for contiguity and the closeness of extreme eigenvalues.} In Theorem~\ref{thm:total_variation}, we give an upper bound on the total variation distance between the distributions of the mixed contraction $n^{-1/2}\cG \cdot \bu \otimes \bv$ and the pure contraction $n^{-1/2} \cG \cdot \bu \otimes \bu$ in terms of $\|\bu - \bv\|$. In Theorem~\ref{thm:mixed_vs_pure_top_evalue}, we compare their largest eigenvalues. In turns out that for contiguity, we need $\|\bu - \bv\| = o(1)$, whereas $\|\bu - \bv\| = o(n^{-1})$ is enough for closeness of the extreme eigenvalues.
    
    \item \textbf{Orthogonality of $\bu$ and $\bv$ implies no outlier eigenvalues.} Finally, in Theorem~\ref{thm:mixed_edge} we obtain the limits of the two extreme eigenvalues and show that these are outside the spectral bulk if and only if the contraction directions $\bu$ and $\bv$ are not orthogonal.
\end{itemize}

\subsection{Notational conventions}
For functions $f, g : \bbN \to \bbR$, we write (i) $f(n) = O(g(n))$, if there exist positive constants $n_0$ and $C$ such that $|f(n)| \le C |g(n)|$ for all $n \ge n_0$; (ii) $f(n) = o(g(n))$ or $f(n) \ll g(n)$ if $\lim_{n \to \infty} \frac{f(n)}{g(n)} = 0$ (we also write $f(n) \gg g(n)$ if $g(n) \ll f(n)$); (iii) $f(n) = \Theta(g(n))$ or $f(n) \asymp g(n)$ if $f(n) = O(g(n))$ and $g(n) = O(f(n))$; (iv) $f(n) \sim g(n)$ if $\lim_{n \to \infty} \frac{f(n)}{g(n)} = 1$; (v) $f(n) \gg g(n)$ if $\lim_{n \to \infty} \frac{g(n)}{f(n)} = 0$.

For a sequence of random variables $\{X_n\}_{n \ge 1}$, we write $X_n = O_\bbP(1)$ if for any $\epsilon > 0$, there exists $K_{\epsilon} > 0$ such that $\sup_n \bbP(|X_n| > K_{\epsilon}) \le \epsilon$. For two sequence of random variables $\{X_n\}_{n \ge 1}$ and $\{Y_n\}_{n \ge 1}$ we write $X_n = O_\bbP(Y_n)$ to mean $X_n = Z_n Y_n $ with $Z_n = O_\bbP(1)$.

For two matrices $A$ and $B$ of the same dimensions, we write $A \succcurlyeq B$ (or equivalently, $B \preccurlyeq A$) if $A - B$ is positive semi-definite.

Finally, although the contraction direction depends on the dimensionality parameter $n$ or the order parameter $r$, we do not display them as subscripts or superscripts to avoid notational clutter. For example, in our asymptotic results, when we state an asymptotic result about $\cG \cdot \bw^{\otimes (r - 2)}$, we actually mean that there are an underlying sequence of contraction directions $(\bw_{r_n, n})_{n \ge 1}$ and a sequence of random tensors $(\cG_{r_n, n})_{n \ge 1}$, and that the result is about the sequence of contractions $(\cG_{r_n, n} \cdot \bw_{r_n, n}^{\otimes (r - 2)})_{n \ge 1}$.

\subsection{Organisation} 
The rest of this paper is organised as follows. We give precise statements of our results in Section~\ref{sec:main}. In particular, Section~\ref{sec:pure} is devoted to pure contractions and Section~\ref{sec:mixed} considers mixed contractions in the case $r = 4$. Section~\ref{sec:proofs} collects all the main proofs. The proofs of the several supporting lemmas are given in Appendix~\ref{sec:supp_lemma}. A couple of universality results and a few other supporting lemmas are proved in the supplementary material \cite{mukherjee2026spectra}.

\subsection{Statement on AI assistance}
Most of this paper was written in 2025, with the first version of our arXiv preprint (\url{https://arxiv.org/abs/2505.10555v1}) appearing in May 2025. Generative AI was not used for proof assistance or writing. We used Gemini to produce a preliminary version of the TikZ diagram in Figure~\ref{fig:lss-domains}, which was then heavily edited manually.

\section{Main results}\label{sec:main}
\subsection{Pure contractions}\label{sec:pure}
We first state our results on pure contractions.
\subsubsection{Limiting spectral distribution}\label{sec:pure_lsd}
We begin with the LSD of pure contractions.
\begin{theorem}[Bulk limit]\label{thm:lsd_pure} 
    Let $\bw \in \bbS^{n - 1}$ and $M_n = \cG \cdot \bw^{\otimes (r - 2)}$. Then $\mubar_{\sqrt{\frac{r - 1}{n}}M_n} \convd \nu_{\sc, 1}$. In fact, for $r \ll n^2$, one has $\mu_{\sqrt{\frac{r - 1}{n}}M_n} \convd \nu_{\sc, 1}$
    in probability. The convergence is almost sure for $r \ll \frac{n^2}{\log n}$.
\end{theorem}
As mentioned in the introduction, the above result has previously been derived in \cite{goulart2022random, au2023spectral, bonnin2024universality} for $r$ fixed. Unlike those works, we may allow $r$ to grow with $n$.

In fact, we can prove a universality result, that allows us to prove the above result for more general tensors with the same covariance structure as the GOTE.

To that end, suppose $\cT \equiv \cT_n$ is a symmetric random tensor of order $r$ and dimension $n$ with 
\begin{align*}
    \bbE [\cT_{i_1, \ldots, i_r}] &= 0,\\
    \Var(\cT_{i_1, \ldots, i_r}) &= \sigma^2_{i_1, \ldots, i_r} = \frac{r}{\# \perms(i_1, \ldots, i_r)},
\end{align*}
and independent entries up to symmetry. Define for $\bw \in \bbS^{n-1}$ the contracted matrix
\begin{equation}\label{eq:tilde_M_n}
    \tilde{M}_n:= \cT\cdot \bw^{\otimes (r-2)}.
\end{equation}
Then, $\tilde{M}_n$ has the same covariance structure as $M_n$.

Let $\cS_{n,r}$ be the set of $r$-multisets with elements from $[n]$. For $S:= \{i_1, \ldots, i_r\} \in \cS_{n,r}$, write $\sigma_S^2 := \sigma^2_{i_1, \ldots, i_r} $ and let 
\begin{align*}
    \cT'_S := \frac{ \cT_{i_1, \ldots, i_r}}{\sigma_{S}}.
\end{align*}
Now we are ready to state our universality result for LSD.
\begin{theorem}[Universality of LSD]\label{thm:univ_LSD}
    Suppose $\mu_n: = \max_{S \in \cS_{n, r}} \bbE[|\cT'_S|^3] < \infty$. Under the delocalisation condition
    \begin{equation}\label{eq:univ_lsd_deloc}
        \begin{split}
            &3(r-2) \log \|\bw\|_\infty + \bigg(n+r -\frac{1}{2} \bigg) \log (n+r -1) + \frac{1}{4} (2 r +1) \log r \\
            &\qquad \qquad \qquad \qquad \qquad \qquad \qquad - (n+2) \log n + \log \max\{1, \mu_n\} \to - \infty,
        \end{split}
    \end{equation}
    we have that $\mubar_{\sqrt{\frac{r - 1}{n}}\tilde{M}_n} \convd \nu_{\sc, 1}$.
\end{theorem}
\begin{remark}
    If $r$ is fixed and $\sup_n \mu_n < \infty$, then the following delocalisation condition suffices:
    \[
        \|\bw\|_\infty \ll n^{-\frac{2r - 5}{6(r-2)}}.
    \]
    On the other hand, if $ \lim_n \frac{r}{n} =c >0$, then we need the condition
    \[
        \|\bw\|_\infty \ll n^{-\frac{1}{2} + \frac{1}{6 c}- \alpha},
    \]
    for some $\alpha>0$.
\end{remark}
\begin{remark}
    We have assumed finite third moments in the universality result in Theorem~\ref{thm:univ_LSD}. It is possible to prove a universality result under finite second moments and a Lindeberg type condition using Theorem~1.1 of \cite{chatterjee2005simple}. However, the resulting conditions turn out to be much more cumbersome to state, and require more stringent delocalisation of $\bw$.
\end{remark}
\begin{corollary}[Sparsification of tensors]\label{cor:sparse_tensor_lsd}
    Suppose $\cT_n$ is a symmetric, centered random tensor with 
    \[
        \Var(\cT_{i_1, \ldots, i_r}) = \sigma_{i_1, \ldots, i_r}^2, \qquad \bbE[|\cT_{i_1, \ldots, i_r}|^3] \le \nu_{n} \sigma_{i_1, \ldots, i_r}^3
    \]
    for some $\nu_n <\infty$. Moreover assume that the entries of $\cT_n$ are independent modulo symmetry. Let $p_{n} \in [0,1]$. Define a zero-one-valued symmetric tensor $\mathcal{B}_n$, independent of $\cT_n$, with $\mathcal{B}_{i_1,\ldots, i_r} \sim \Ber(p_n)$ and independent entries. Define a sparse tensor $\overline{\cT}_n$ with entries $
    \overline{\cT}_{i_1, \ldots, i_r} := \mathcal{B}_{i_1, \ldots, i_r} \mathcal{T}_{i_1, \ldots, i_r}$ and its contraction $\overline{M}_n := \overline{\cT}_n \cdot \bw^{\otimes {r-2}}$ for some $\bw \in \bbS^{n-1}$. Then 
    \[
        \mubar_{\sqrt{\frac{r - 1}{n p_n}}\overline{M}_n} \convd \nu_{\sc, 1}
    \]
    if \eqref{eq:univ_lsd_deloc} holds with $\mu_n$ replaced by $p^{-\frac{1}{2}} \nu_n$.
\end{corollary}

\subsubsection{Edge eigenvalues: first order results}\label{sec:pure_edge_first_order}
Now we look at the behaviours of the extreme edge eigenvalues. Interestingly, outlier eigenvalues emerge for $r \ge 4$.
Define
\[
    \varpi_r := \dfrac{1}{\sqrt{r - 1}} \bigg(\sqrt{r - 2} + \dfrac{1}{\sqrt{r - 2}}\bigg).
\]
Then for $r \ge 3$ fixed, the largest and the smallest eigenvalues of $n^{-1/2} M_n$ converge respectively to $\varpi_r$ and $-\varpi_r$.
We note that $\frac{2}{\sqrt{r - 1}}$ is the right end-point of the support of the LSD of $n^{-1/2}M_n$ and $\varpi_r > \frac{2}{\sqrt{r - 1}}$ for $r \ge 4$. This emergence of outlier eigenvalues for $r \ge 4$ was previously noted in \cite{mukherjee2024spectra} in the context of adjacency matrices of Erd\H{o}s-R\'{e}nyi hypergraphs.
\begin{theorem}[Edge eigenvalues: first order results]\label{thm:edge-limits}
For any fixed $r \ge 3$, we have
    \begin{equation}\label{eq:edge_prob_conv_r_fixed}
        \bigg(\frac{\lambda_1(M_n)}{\sqrt{n}}, \frac{\lambda_n(M_n)}{\sqrt{n}}\bigg) \convas (\varpi_r, -\varpi_r).
    \end{equation}
If $1 \ll r \ll n$, we have
    \begin{equation}\label{eq:edge_prob_conv_r<<n}
        \bigg(\frac{\lambda_1(M_n)}{\sqrt{n}}, \frac{\lambda_n(M_n)}{\sqrt{n}}\bigg) \convas (1, -1).
    \end{equation}
If further $r \ll \frac{n}{\log{n}}$, then the convergence is almost sure.
If $\frac{r}{n} \to c \in (0, \infty)$,
    \begin{equation}\label{eq:edge_dist_conv_r=n}
        \bigg(\frac{\lambda_1(M_n)}{\sqrt{n}}, \frac{\lambda_n(M_n)}{\sqrt{n}}\bigg) \convd \bigg(\xi, -\frac{1}{\xi}\bigg),
    \end{equation}
    where $\xi = \frac{\sqrt{c}\zeta + \sqrt{c\zeta^2 + 4}}{2}$ with $\zeta \sim \cN(0, 1)$.
    Finally, when $\frac{r}{n} \to \infty$,
    \begin{equation}\label{eq:edge_dist_conv_r>>n}
        \bigg(\frac{\lambda_1(M_n)}{\sqrt{r}}, \frac{\lambda_n(M_n)}{\sqrt{r}}\bigg) \convd \big(\zeta_+, -\zeta_-\big),
    \end{equation}
    where $\zeta \sim \cN(0, 1)$, and $x_+ = \max\{x, 0\}, x_- = \max\{-x, 0\}$ denote respectively the positive and the negative parts of $x$.
\end{theorem}

Next we give a universality result for the edge eigenvalues. Recall the contracted matrix $\tilde{M}_n$ from \eqref{eq:tilde_M_n}.

\begin{theorem}[Universality of the edge eigenvalues]\label{thm:univ_edge}
    Suppose that there exists $K_{n} >0$ such that the coordinates of $\cT$ satisfy
    \[
        |\cT_{i_1, \ldots, i_r}| \le K_{n} \sigma_{i_1, \ldots, i_r}.
    \]
    for all $i_1, \ldots, i_r \in [n]$. Moreover, assume that $r = o \left(\frac{n}{\log n} \right)$ and
    \begin{equation}\label{eq:univ_condition}
        K_n n^{-\frac{1}{2}} (r-1)!^{\frac{1}{2}} (\log n)^2 \|\bw\|_\infty^{r-2}  \to 0,
    \end{equation} 
    then the results in Theorem~\ref{thm:edge-limits} holds for $\tilde{M}_n$ for regimes of $r = o(\frac{n}{\log n})$ allowed by \eqref{eq:univ_condition}.
\end{theorem}

\begin{remark}
    Suppose $\sup_n K_n <\infty$. By Stirling's formula the LHS of \eqref{eq:univ_condition} can be bounded as follows:
    \begin{align*}
        K_n n^{-\frac{1}{2}} (r-1)!^{\frac{1}{2}} (\log n)^2 \|\bw\|_\infty^{r-2} &\lesssim  n^{-\frac{1}{2}} r^{\frac{r}{2} - \frac{1}{4}} e^{-\frac{r}{2} - \frac{1}{2}}(\log n)^2 \|\bw\|_\infty^{r-2} \\
        &\lesssim n^{-\frac{1}{2}} r^{\frac{3}{4}} (\log n)^2 \left(\frac{r^{\frac{1}{2}} \|\bw\|_\infty}{e^{\frac{1}{2}}} \right)^{r-2} \\
        &\ll r^{\frac{3}{4}}  \left(\frac{r^{\frac{1}{2}} \|\bw\|_\infty}{e^{\frac{1}{2}}} \right)^{r-2}.
    \end{align*}
    Now, under the delocalisation condition
    \[
        \limsup_{n \to \infty} r^{\frac{1}{2}} \|\bw\|_\infty < e^{\frac{1}{2}},
    \]
    \eqref{eq:univ_condition} holds. In particular, if $r$ is fixed, note that just $\|\bw\|_\infty = O(1)$ is sufficient (thus completely localised vectors such $\be_1$ are allowed). On the other hand, $r \le C n^\tau$ for some $\tau \in [0,1)$, then the delocalisation condition 
    \[
        \limsup_{n \to \infty} n^{\frac{\tau}{2}} \|\bw\|_\infty < C^{-\frac{1}{2}} e^{\frac{1}{2}}.
    \]
    suffices.
\end{remark}

\begin{corollary}[Sparsification of tensors]\label{cor:sparse_tensor}
    Suppose $\cT$ is a centered random tensor with $|\cT_{i_1, \ldots, i_r}| \le K_{n} \sigma_{i_1, \ldots, i_r}$ and $\Var(\cT_{i_1, \ldots, i_r}) = \sigma_{i_1, \ldots, i_r}^2$ with independent entries. Let $p_{n} \in [0,1]$. Define a zero-one-valued symmetric tensor $\mathcal{B}$, independent of $\cT$, with $\mathcal{B}_{i_1,\ldots, i_r} \sim \Ber(p_n)$ and independent entries. Define a sparse tensor $\overline{\cT}$ with entries $
    \overline{\cT}_{i_1, \ldots, i_r} := \mathcal{B}_{i_1, \ldots, i_r} \mathcal{T}_{i_1, \ldots, i_r}$ and its contraction $\overline{M}_n := \overline{\cT} \cdot \bw^{\otimes {r-2}}$. If $r = o\left(\frac{n}{\log{n}} \right)$ and 
    \[
        K_n n^{-\frac{1}{2}} (r-1)!^{\frac{1}{2}} (\log n)^2 \|\bw\|_\infty^{r-2} p_n^{-\frac{1}{2}}  \to 0,
    \]
    then the conclusions of Theorem~\ref{thm:edge-limits} hold for the tensor $\frac{1}{\sqrt{p_n}} \overline{\cT}$ for regimes of $r = o(\frac{n}{\log n})$ allowed by the above condition.
\end{corollary}

\subsubsection{Edge eigenvalues: second order results}\label{sec:pure_edge_second_order}
We now discuss fluctuations of the edge eigenvalues. Fluctuation results for $r = 2$ is known to be Tracy-Widom. For $r \ge 4$, we establish Gaussian fluctuations. The $r = 3$ is case is rather interesting. Here one encounters the largest eigenvalue of a stochastic Airy operator \citep{bloemendal2016limits}.
\begin{definition}[Stochastic Airy operator]
    Consider the vector-valued function space $L^2(\bbR^+, \bbR^s)$ with inner product $\langle \bff, \bg \rangle := \int_0^\infty \bff^\top \bg$. The random Schr\"{o}dinger operator
    \[
        H_1 := -\frac{d^2}{dx^2} + \sqrt{2} B_x' + sx
    \]
    is called the \emph{multivariate stochastic Airy operator}, where $B_x'$ is ``standard matrix white noise'', the generalised derivative of a standard matrix Brownian motion, and $sx$ is scalar. 
\end{definition}
It is well known that with a general self-adjoint boundary condition
    \[
        \bff'(0) = W \bff(0),   
    \]
the multivariate stochastic Airy operator is bounded below with purely discrete spectrum given by a variational principle. Write the self-adjoint matrix $W$ as $W=\sum_{i=1}^s d_i \bu_i \bu_i^\top$. We formally allow $d_i \in (-\infty, \infty]$. In this case, writing $f_i = \bu_i^\top \bff$, the initial condition can be interpreted as 
\begin{align*}
    f'_i(0) &=  d_i f_i(0) \,\,\,\,\text{if} \,\,d_i \in \bbR,\\
       f_i(0) &=  0 \,\,\,\,\text{if} \,\, d_i =\infty.
\end{align*}
The operator defined above along with the initial condition is denoted by $\cH_{1,(d_1, d_2, \ldots, d_s)}$. An interesting fact is that the distribution of the largest eigenvalue $\Lambda_0$ of $\cH_{1,(d_1, d_2, \ldots, \infty)}$ is same as that of $\cH_{1,(d_1, d_2, \ldots, d_{s-1})}$. Now we are ready to state the second order results at the edge.
\begin{theorem}[Edge eigenvalues: second order results]\label{thm:edge-fluctuations}
The following results hold.
    \begin{enumerate}
        \item[(i)] If $r = 3$,
        \begin{align}
            n^{2/3}\bigg(\frac{\lambda_1(M_n)}{\sqrt{n}} - \sqrt{2}\bigg) &\convd -\Lambda_0, \label{eq:edge_fluc_1_r=3}\\
            n^{2/3}\bigg(\frac{\lambda_n(M_n)}{\sqrt{n}} + \sqrt{2}\bigg) &\convd \Lambda_0,
            \label{eq:edge_fluc_n_r=3}
        \end{align}
        where $\Lambda_0$ is the largest eigenvalue of $\mathcal{H}_{1,(0)}$.
        
        \item[(ii)] If $r \ge 4$ is fixed, then we have Gaussian fluctuations:
        \begin{align}\label{eq:edge_fluc_r_fixed}
            \sqrt{n}\bigg(\frac{\lambda_1(M_n)}{\sqrt{n}} - \varpi_r, &\frac{\lambda_n(M_n)}{\sqrt{n}} + \varpi_r\bigg) \\ \nonumber
            &\xrightarrow{d} \cN_2\bigg((0,0), 
            \frac{r-3}{4(r-2)(r-1)}
            \begin{pmatrix}
               r^2 - 1 & r^2-9\\
               r^2 - 9 & r^2-1
            \end{pmatrix}\bigg).
        \end{align}
        
        \item[(iii)] If $\log{n} \ll r \ll n$, then
        \begin{equation}\label{eq:edge_fluc_log n<<r<<n}
            \sqrt{\frac{n}{r}}\bigg(\frac{\lambda_1(M_n)}{\sqrt{n}} - \sqrt{\frac{r-2}{r-1}}, \frac{\lambda_n(M_n)}{\sqrt{n}} + \sqrt{\frac{r-2}{r-1}}\bigg) \convd \cN_2\bigg((0,0), 
            \frac{1}{4}
            \begin{pmatrix}
               1 & 1\\
               1 & 1
            \end{pmatrix}\bigg).
        \end{equation}
         Further, if $n^{\frac{1}{3}} \ll r \ll n$, then in \eqref{eq:edge_fluc_log n<<r<<n}, the centerings can be replaced by $1$. 
    \end{enumerate}
\end{theorem}

\subsubsection{Eigenvector overlap}\label{sec:pure_eigenvector_overlap}
Our next result shows that for $r \ge 4$, there is non-trivial information about the contraction direction $\bw$ in the edge eigenvectors. In fact, in each of the regimes $1 \ll r \ll n$ and $r \gg n$, one may identify two explicit (data-dependent) vectors, one of which tends to become aligned with $\bw$.
\begin{theorem}[Overlap of edge-eigenvectors with $\bw$]\label{thm:overlap-with-w}
Let $r \ge 4$. Let $\bs_1$ and $\bs_n$ denote the largest and the smallest (normalised) eigenvectors of $M_n$ (unique up to signs). Then we have the following.
\begin{enumerate}
    \item [(i)] If $r$ is fixed, then $\dist\big(\bw, \spn\{\bs_1, \bs_n\}\big) \convp \frac{\theta}{\beta}$. 
    \item [(ii)] If $r \to \infty$, then $\dist\big(\bw, \spn\{\bs_1, \bs_n\}\big) \convp 0$.
\end{enumerate}
In fact, let $\bt_1 = \frac{\bs_1 + \bs_n}{\sqrt{2}}$ and $\bt_2 = \frac{\bs_1 - \bs_n}{\sqrt{2}}$, and define
\begin{align*}
    (\delta_1, \delta_n) &:= \bigg(\max\bigg\{|\bw^\top \bs_1|, |\bw^\top \bs_n|\bigg\}, \min\bigg\{|\bw^\top \bs_1|, |\bw^\top \bs_n|\bigg\} \bigg); \\
    (\tilde{\delta}_1, \tilde{\delta}_n) &:= \bigg(\max\bigg\{|\bw^\top \bt_1|, |\bw^\top \bt_2|\bigg\}, \min\bigg\{|\bw^\top \bt_1|, |\bw^\top \bt_2|\bigg\} \bigg).
\end{align*}
Then, for any fixed $r \ge 4$,
\[
     (\delta_1, \delta_n) \convp \frac{1}{\sqrt{2}}\sqrt{1 - \tfrac{\theta^2}{\beta^2}}(1, 1), \qquad
     (\tilde{\delta}_1, \tilde{\delta}_n) \convp \sqrt{1 - \tfrac{\theta^2}{\beta^2}}(1, 0).
\]
If $1 \ll r \ll n$, then
\[
    (\delta_1, \delta_n) \convp \frac{1}{\sqrt{2}}(1, 1), \qquad
    (\tilde{\delta}_1, \tilde{\delta}_n) \convp (1, 0).
\]
If $\frac{r}{n} \to c \in (0, \infty)$, then
\[
    (\delta_1, \delta_n) \convd \frac{1}{\sqrt{2}}\left(\frac{\xi}{\sqrt{\xi^2 + 1}}, \frac{1}{\sqrt{\xi^2 + 1}}\right), \quad
    (\tilde{\delta}_1, \tilde{\delta}_n) \convd \frac{1}{\sqrt{2}}\left(\frac{\xi + 1}{\sqrt{\xi^2 + 1}}, \frac{\xi - 1}{\sqrt{\xi^2 + 1}}\right),
\]
where $\xi = \frac{\sqrt{c}\zeta + \sqrt{c \zeta^2 + 4}}{2} \ge 1$ with $\zeta \sim \cN(0, 1)$.
Finally, if $r \gg n$, then
\[
    (\delta_1, \delta_n) \convp (1, 0), \qquad
    (\tilde{\delta}_1, \tilde{\delta}_n) \convp \frac{1}{\sqrt{2}}(1, 1).
\]
\end{theorem}

\begin{corollary}
    In the regime $1 \ll r \ll n$, either $\bt_1$ or $\bt_2$ tends to get aligned with $\bw$. On the other hand, when $r \gg n$, either $\bs_1$ or $\bs_n$ tends to get aligned with $\bw$.
\end{corollary}

\subsubsection{Directional spectral measure}\label{sec:pure_directional_spectral_measure}
Next we shall consider directional spectral measures of $M_n$ for fixed $r$.
\begin{definition}[Directional Spectral Measure]
Let $\bx \in \bbS^{n-1}$. The spectral measure of a symmetric matrix $A \in \Mat_n(\bbR)$ in the direction $\bx$ is defined as
\[
    \mu_{A, \bx} := \sum_{i = 1}^n |\bx^\top \bu_i|^2 \delta_{\lambda_i},
\]
where $\{(\lambda_i, \bu_i)\}_{i = 1}^n$ are the eigenvalue-eigenvector pairs of $A$.
\end{definition}
It turns out that the spectral measure in the direction $\bx$ is a mixture of four measures.
\begin{theorem}[Spectral measure in the direction $\bx$]\label{thm:spectral_measure}
Suppose $r$ is fixed, $\bx \in \bbS^{n-1}$ and $\rho_n:=\la \bx, \bw \ra \to \rho$. Then, we have that
\[
    \mu_{\frac{M_n}{ \sqrt{n}}, \bx} \convd (1-\rho^2)\nu_{\sc, \frac{1}{r-1}} + \frac{\rho^2}{r-2}\mu_{r} + \frac{\rho^2}{2}\bigg(1 - \frac{1}{r-2}\bigg) \delta_{-\varpi_r} + \frac{\rho^2}{2}\bigg(1 - \frac{1}{r-2}\bigg) \delta_{\varpi_r},
\]
almost surely, where
\[
    d\mu_{r}(x) =  \frac{(r-1)(r-2)}{2 \pi} \cdot \frac{\sqrt{4 - x^2}}{(r-1)^2 - (r-2)x^2} \cdot \mathbb{I}_{[-2, 2]}(x) \, dx.
\]
\end{theorem}
\begin{remark}
    The celebrated work \cite{mckay1981expected} proved that the LSD of the adjacency matrix of a random $d$-regular graph is given by the so called Kesten--McKay law with density
    \[
        \frac{d\sqrt{4(d-1) - x^2}}{2\pi (d^2 - x^2)} \cdot \mathbb{I}_{[-2 \sqrt{d-1}, 2 \sqrt{d-1}]}(x) \,dx.
    \]
    Interestingly, the measure $\mu_r$ is a Kesten--Mckay measure\footnote{We thank Dipranjan Pal for this observation.} with $d= r-1 $ and scaling $1/\sqrt{r-2}$. We do not know if there is a deeper structural explanation for this.
\end{remark}
\begin{remark}
        \cite{noiry2010spectral} obtained the limiting directional measure for rank-$1$ spiked Wigner matrices. The limit is known as a Meixner law. The measure $\mu_r$ is essentially sum of two free Meixner laws. 
\end{remark}

\subsubsection{Linear spectral statistics}\label{sec:pure_lss}
Our next result establishes Gaussian fluctuations for linear spectral statistics for pure contractions.
\begin{definition}[Chebyshev polynomials]
    The Chebyshev polynomial of first kind is a sequence of polynomials defined inductively as
    \begin{align*}
        T_0(x) &= 1,\\ 
        T_1(x) &= x,\\
        T_n(x) &= 2 x T_{n-1}(x) - T_{n-2} (x) \text{ for } n \geq 3.
    \end{align*}
\end{definition}
For example, $T_2(x) = 2x^2 -1$, $T_3(x) = 4 x^3 - 3x$, $ T_4(x) = 8 x^4 - 8x^2 +1$.
We also define a sequence of linear functionals $\{\tau_\ell\}_{\ell =0} ^\infty$ on $C([-2, 2])$ as follows:
\[
    \tau_\ell(\psi) := \frac{1}{\pi} \int_{-2}^2 \psi(x) \frac{T_\ell(x/2)}{\sqrt{4-x^2}} = \frac{1}{2\pi} \int_{-\pi}^\pi \psi(2 \cos{\eta}) \,\cos{(\ell \eta)}\,d\eta.
\]
The following theorem gives fluctuations of linear spectral statistics of $\sqrt{\frac{r-1}{n}}M_n$. Let $\lambda_1, \ldots, \lambda_n$ be the eigenvalues of $\sqrt{\frac{r-1}{n}}M_n$.
\begin{theorem}[Fluctuations of Linear Spectral Statistics for fixed $r$]\label{thm:lss}
    Suppose $\psi: \bbR \to \bbR$ be a function admitting a power series representation in an open interval containing $[-\sqrt{r-2}-2, \sqrt{r-2} +2]$. Moreover, assume that the power series of $\psi$ is absolutely convergent in the aforementioned interval. Then,
    \[
        \sum_{i = 1}^n \psi(\lambda_i) - n \int \psi(x) \, d\nu_{\sc, 1}(x) \convd \cN(m_\psi, \sigma_{\psi}^2),
    \]
    where
    \begin{align*}
        m_\psi &:= \frac{1}{4} (\psi(2) + \psi(-2)) - \frac{1}{2} \tau_0(\psi) + 2 \sum_{\ell =1}^\infty (r-2)^\ell \tau_{2\ell}(\psi),\\
        \sigma_{\psi}^2 &:= 2 \sum_{\ell=1}^\infty \ell\tau_\ell(\psi)^2.
    \end{align*}
\end{theorem}

\subsection{Mixed contractions}\label{sec:mixed}
In this subsection, we shall consider mixed contractions $\cG \cdot \bu \otimes \bv$ of a $\GOTE(4, n)$ tensor $\cG$, where $\bu, \bv \in \bbS^{n - 1}$. From Theorem~1.5 of \cite{au2023spectral}, one obtains the LSD of such contractions (we also provide a different proof).
\begin{theorem}\label{thm:mixed_bulk}
    Suppose that $\rho_n := \langle \bu, \bv\rangle \to \rho$. Then, almost surely,
    \[
        \mu_{n^{-1/2}\cG \cdot \bu \otimes \bv} \convd \nu_{\sc, (1 + \rho^2)/6}.
    \]
\end{theorem}
We are interested in the following question:
\begin{quote}
    \it
    How large does the overlap $\langle \bu, \bv\rangle$ need to be for $\cG \cdot \bu \otimes \bv$ and $\cG \cdot \bu \otimes \bu$ to have similar bulk/edge behaviour?
\end{quote}
Before we come to this, let us first look at a contiguity question. Let $\bbP_{\bu, \bv}$ denote the joint law of the entries of the $\cG \cdot \bu \otimes \bv$ on $\bbR^{n(n + 1)/2}$.
\begin{theorem}\label{thm:total_variation}
    There is an absolute constant $C > 0$ such that
    \[
        d_{\mathrm{TV}}(\bbP_{\bu, \bv}, \bbP_{\bu, \bu}) \le C n \|\bu - \bv\|_2.
    \]
\end{theorem}
\begin{remark}
    The upper bound in Theorem~\ref{thm:total_variation} is in general tight, i.e. there are $\bu, \bv \in \bbS^{n - 1}$ such that a matching lower bound holds.
\end{remark}
By Theorem~\ref{thm:total_variation} the laws of the two matrices $\cG \cdot \bu \otimes \bv$ and $\cG \cdot \bu \otimes \bv$ are contiguous when $\|\bu - \bv\| = o(n^{-1})$. This raises the question: when do their eigenvalues behave similarly? The following result shows that just $\|\bu - \bv\|_2 = o(1)$ suffices for the bulk limits to coincide.
\begin{proposition}\label{prop:mixed_vs_pure_lsd}
    If $\|\bu - \bv\|_2 = o(1)$, then, almost surely, $\mu_{n^{-1/2}\cG \cdot \bu \otimes \bv} \convd \nu_{\sc, \frac{1}{3}}$.
\end{proposition}
We next investigate how close the largest eigenvalues of $\cG \cdot \bu \otimes \bv$ and $\cG \cdot \bu \otimes \bu$ are.
\begin{theorem}\label{thm:mixed_vs_pure_top_evalue}
    Let $\cG \sim \GOTE(4, n)$. There exist universal constants $C_1, C_2 > 0$ such that for any $\varepsilon > 0$,
    \[
        |\bbE \lambda_1(n^{-1/2} \cG \cdot \bu \otimes \bv) - \bbE \lambda_1(n^{-1/2}\cG \cdot \bu \otimes \bu)| \le C_1 \varepsilon + C_2 \sqrt{\|\bu - \bv\|_2 \log \big(\tfrac{1}{\varepsilon}\big)}.
    \]
\end{theorem}
As a corollary, we obtain that if $\|\bu - \bv\|_2 = o(1)$, then the largest eigenvalue of $n^{-1/2} \cG \cdot \bu \otimes \bv$ has the same limit as that of $n^{-1/2} \cG \cdot \bu \otimes \bu$.
\begin{corollary}\label{cor:mixed_limit_largest_evalue}
    Suppose $\|\bu - \bv\|_2 = o(1)$. Then
    \[
        \lambda_1(n^{-1/2} \cG \cdot \bu \otimes \bv) \convas \varpi_4 = \frac{1}{\sqrt{3}}\left(\sqrt{2} + \frac{1}{\sqrt{2}}\right).
    \]
\end{corollary}
In fact, we can obtain the limits of two edge eigenvalues. 
\begin{theorem}[Edge eigenvalues for mixed contractions]\label{thm:mixed_edge}
    Suppose $\bu, \bv \in \bbS^{n-1}$ and $\rho_n:=\langle \bu, \bv \rangle \to \rho$. Then we have
    \begin{align*}
        \lambda_1(n^{-1/2}\cG\cdot\bu \otimes \bv) &\convas  \frac{2}{\sqrt{6}}\bigg(1+  \frac{\rho^2}{1 + |\rho|}\bigg),\\
        \lambda_n(n^{-1/2}\cG\cdot\bu \otimes \bv) &\convas  -\frac{2}{\sqrt{6}}\bigg(1+  \frac{\rho^2}{1 + |\rho|}\bigg).
    \end{align*}
\end{theorem}
\begin{remark}
    Note that the boundaries of the bulk in this case is $\pm\frac{2}{\sqrt{6}}$. Thus, there are outlier eigenvalues if and only if $\rho \ne 0$ that is $\bu, \bv$ are not asymptotically orthogonal.
\end{remark}
\begin{remark}
 Bulk limits of general mixed contractions has been obtained in \cite{au2023spectral} for any fixed $r$. Edge fluctuations and finer spectral properties of general mixed contractions will be studied in a future work. While our approach for $r = 4$ works in principle, the resulting low-rank representations are much more involved.
\end{remark}

\section{Proofs of our main results}\label{sec:proofs}
In this section, we prove the results stated in Section~\ref{sec:main}. The proofs of the various supporting lemmas are postponed to Appendix~\ref{sec:supp_lemma}. We have made a couple of exceptions in the cases of Lemma~\ref{lem:cov-structure-pure} and \ref{lem:rank-2-perturbation} as they are the key to all of our results on pure contractions.

We first introduce three parameters that make extensive appearances in our proofs:
\begin{align*}
    \alpha = \sqrt{\frac{(r - 2)(r - 3)}{(r - 1)}}, \quad
    \beta = \sqrt{\frac{r - 2}{r - 1}}, \quad
    \theta = \sqrt{\frac{1}{r - 1}}.
\end{align*}

\subsection{The correlation structure of pure contractions}
Although the entries of a GOTE tensor are themselves independent, contractions introduce non-trivial correlations. The following lemma describes the covariances between the entries of $M$.
\begin{lemma}\label{lem:cov-structure-pure}
    Suppose $i,j,k,l \in[n]$ are district indices. Then
    \begin{align} \label{eq:var-1}
        \Var(M_{ii}) &= 2 \theta^2 + 4\beta^2 w_i^2 + \alpha^2 w_i^4; \\ \label{eq:var-2}
        \Var(M_{ij}) &= \theta^2 + \beta^2 (w_i^2 + w_j^2) + \alpha^2 w_i^2 w_j^2; \\ \label{eq:cov-1}
        \Cov(M_{ij}, M_{kl}) &= \alpha^2 w_i w_j w_k w_l; \\ \label{eq:cov-2}
        \Cov(M_{ii}, M_{kl}) &= \alpha^2 w_i^2 w_k w_l; \\ \label{eq:cov-3}
        \Cov(M_{ii}, M_{kk}) &= \alpha^2 w_i^2 w_k^2; \\ \label{eq:cov-4}
        \Cov(M_{ij}, M_{il}) &= \beta^2 w_j w_l + \alpha^2 w_i^2 w_j w_l; \\ \label{eq:cov-5}
        \Cov(M_{ii}, M_{il}) &= 2 \beta^2 w_i w_l + \alpha^2 w_i^3 w_l.
    \end{align}
\end{lemma}
Our main observation is that the covariance structure of $M_n$ may be explained by a random rank-$2$ perturbation of a scaled GOE matrix.
\begin{lemma}\label{lem:rank-2-perturbation}
Let
\begin{equation}\label{eq:rank-2-perturbation}
    X_n := \alpha U \bw \bw^\top + \beta(\bV \bw^\top + \bw \bV^\top) + \theta Z_n,
\end{equation}
where $Z_n$ is a GOE random matrix, $\bV$ is an independent standard Gaussian vector, and $U$ is another independent Gaussian random variable. Then $M_n \overset{d}{=} X_n$, i.e. the entries of $M_n$ have the same joint distribution as those of $X_n$.
\end{lemma}

The above distributional representation unlocks a host of limiting spectral properties, which would otherwise be much more difficult to derive.

\begin{remark}\label{rem:comp-with-spherical-spin-glass}
A comparison with the distributional representation of the Hessian of the Hamiltonian of an $r$-spin spherical spin glasses is in order. \cite{auffinger2013random} showed that the distribution of the Hessian is essentially a GOE plus a (deterministic) multiple of the identity matrix (see Lemma~3.2 in their paper). As such the distribution does not contain any information about the contraction direction $\bw$. This fact is crucially in their analysis as they are able to immediately obtain the conditional distribution of the Hessian at a critical point. On the other hand, as Lemma~\ref{lem:rank-2-perturbation} shows, the variance profile of the GOTE ensures that its contractions have non-trivial information on $\bw$, manifesting essentially via edge-eigenvectors.
\end{remark}

We end this subsection with the proofs of Lemma~\ref{lem:cov-structure-pure} and \ref{lem:rank-2-perturbation}.

\begin{proof}[Proof of Lemma~\ref{lem:cov-structure-pure}]
    We first prove \eqref{eq:var-1}. Note that
    \begin{align*}
        \Var(M_{ii}) &=\sum_{\substack{i_3, \ldots, i_r \\ i_3', \ldots, i_r'}} \bbE[\cG_{iii_3\cdots i_r} \cG_{iii_3'\cdots i_r'}] w_{i_3} \cdots w_{i_r} w_{i_3'} \cdots w_{i_r'} \\
        &= \sum_{i_3, \ldots, i_r} \frac{r \#\perms(i_3, \ldots, i_r)}{\#\perms(i, i, i_3, \ldots, i_r)} w_{i_3}^2 \cdots w_{i_r}^2 \\
        &= \frac{1}{r - 1} \sum_{k = 0}^{r - 2} \binom{r - 2}{k} \frac{(k + 2)!}{k!} w_i^{2k}(1 - w_i^2)^{r - 2 - k} \\
        &=\frac{1}{r - 1} \bbE[(X + 2)(X + 1)] \quad (\text{where } X \sim \Binomial(r - 2, w_i^2)) \\
        &= \frac{2}{r - 1} + \frac{4 (r - 2)}{r - 1} w_i^2 + \frac{(r - 2)(r - 3)}{(r - 1)} w_i^4.
    \end{align*}
    For \eqref{eq:var-2}, note that
    \begin{align*} 
        &\Var(M_{ij}) \\
        &= \sum_{\substack{i_3, \ldots, i_r \\ i_3', \ldots, i_r'}} \bbE[\cG_{iji_3\cdots i_r} \cG_{iji_3'\cdots i_r'}] w_{i_3} \cdots w_{i_r} w_{i_3'} \cdots w_{i_r'} \\
        &= \sum_{i_3, \ldots, i_r} \frac{r \#\perms(i_3, \ldots, i_r)}{\#\perms(i, j, i_3, \ldots, i_r)} w_{i_3}^2 \cdots w_{i_r}^2 \\
        &= \frac{1}{r - 1} \sum_{0 \le k_i + k_j \le r - 2} \binom{r - 2}{k_i} \binom{r - 2 - k_i}{k_j} (k_i + 1)(k_j + 1) \\
        &\qquad\qquad\qquad\qquad\qquad\qquad\qquad \times w_i^{2k_i} w_j^{2k_j} (1 - w_i^2 - w_j^2)^{r - 2 - k_i - k_j} \\
        &= \frac{1}{r - 1} \bbE[(X + 1)(Y + 1)] \, (\text{where } (X, Y, Z) \sim \Multinomial(r - 2, w_i^2, w_j^2, 1 - w_i^2 - w_j^2)) \\
        &= \frac{1}{r - 1} + \frac{r - 2}{r -1}(w_i^2 + w_j^2) + \frac{(r - 2)(r - 3)}{r - 1} w_i^2 w_j^2.
    \end{align*}
    For \eqref{eq:cov-1},
    \begin{align*}
        \Cov(M_{ij}, M_{kl}) &= \sum_{\substack{i_3, \ldots, i_r \\ i'_3, \ldots, i'_r}} w_{i_3} \cdots w_{i_r} w_{i'_3} \cdots w_{i'_r}\bbE(\cG_{iji_3 \cdots i_r} \cG_{kl i'_3 \cdots i'_r}).
    \end{align*}
    The expectation appearing in the above summation is nonzero if and only if $\{i, j, i_3, \ldots, i_r\} =\{k, l, i'_3, \ldots, i'_r\}$ as multisets. Suppose in the multiset $\{i, j, i_3, \ldots, i_r\}$, $m$ appears $\gamma_m +1$ times if $m=i, j, k, l$ and $\gamma_m$ times if $m \neq i, j, k, l$. The $+1$ appears in the first case since $i, j, k, l$ must appear in the multiset $\{i, j, i_3, \ldots, i_r\}$. We must have $\alpha_m \geq 0$ for all $m$ and $\sum_{m=1}^n \alpha_m + 4= r$.Then, the $(r-2)$-tuple $(i_3, \ldots, i_r)$ can be chosen in 
    \[
        \frac{(r-2)!}{(1+\gamma_k)! (1+ \gamma_l)!\prod_{m\neq k,l} \gamma_m!} = \frac{(r-2)!}{(1+\gamma_k) (1+ \gamma_l)\prod_{m} \gamma_m!}
    \]
    ways and $(i'_3, \ldots, i'_r)$ can be chosen in 
    \[
        \frac{(r-2)!}{(1+\gamma_i)! (1+ \gamma_j)\prod_{m\neq i, j} \gamma_m!} = \frac{(r-2)!}{(1+\gamma_i) (1+ \gamma_j)\prod_{m} \gamma_m!}
    \]
    ways. In this set up,
    \begin{align*}
        \bbE(\cG_{iji_3 \cdots i_r} \cG_{kl i'_3 \cdots i'_r}) &= \frac{r}{\frac{r!}{(1+\gamma_i)! (1+\gamma_j)! (1+\gamma_k)! (1+\gamma_l)! \prod_{m \neq i, j, k, l} \gamma_m !}}\\
        &= \frac{(1+\gamma_i) (1+\gamma_j) (1+\gamma_k) (1+\gamma_l)\prod_m \gamma_m !}{(r-1)!}.
    \end{align*}
    Then,
    \begin{align*}
        \Cov(M_{ij}, M_{kl}) &= \sum_{\substack{\gamma_m \ge 0 \\ \sum_{m=1}^n\gamma_m = r-4}}  \frac{(r-2)!}{(1+\gamma_k) (1+ \gamma_l)\prod_{m} \gamma_m!} \times \frac{(r-2)!}{(1+\gamma_i) (1+ \gamma_j)\prod_{m} \gamma_m!}\\
        &\qquad \times w_iw_jw_k w_l \prod_m w_m^{2\gamma_m} \times \frac{(1+\gamma_i) (1+\gamma_j) (1+\gamma_k) (1+\gamma_l)\prod_m \gamma_m !}{(r-1)!}\\
        &= \frac{(r-2)(r-3)}{r-1}w_i w_j w_k w_l \sum_{\sum_m \gamma_m =r-4} \frac{(r-4)!}{\prod_m \gamma_m!} \prod w_m^{2\gamma_m}\\
        &=\frac{(r-2)(r-3)}{r-1}w_i w_j w_k w_l\bigg(\sum_m w_m^2\bigg)^{r-4}\\
        &= \frac{(r-2)(r-3)}{r-1}w_i w_j w_k w_l.
    \end{align*}
    In the next four computations, $\gamma_m$'s are defined similarly.
    \begin{align*}
        \Cov(M_{ii}, M_{kl}) &= \sum_{\substack{\gamma_m \ge 0 \\ \sum_m \gamma_m = r-4}} \frac{(r-2)!}{(1+\gamma_k)!(1+\gamma_l)!\prod_{m \neq k, l} \gamma_m!} \times \frac{(r-2!)}{(2+\gamma_i)!\prod_{m \neq i} \gamma_m!} \\
        &\qquad\qquad \times w_i^2 w_k w_l  \prod_m w_m^{2\gamma_m} \times \frac{r}{\frac{r!}{(2+\gamma_i)! (1+\gamma_k)!(1+\gamma_l)!\prod_{m \neq i, k, l} \gamma_m!}}\\
        &= \frac{(r-2)(r-3)}{r-1} w_i^2 w_k w_l \sum_{\sum_m \gamma_m = r-4} \frac{(r-4)!}{\prod_m \gamma_m!}\prod_m w_m^{2\gamma_m}\\
        &= \frac{(r-3)(r-2)}{r-1} w_i^2 w_k w_l.
    \end{align*}
    This proves \eqref{eq:cov-2}.
    \begin{align*}
        \Cov(M_{ii}, M_{kk}) &= \sum_{\substack{\gamma_m \ge 0 \\ \sum_m \gamma_m = r-4}} \frac{(r-2)!}{(2+\gamma_k)!\prod_{m \neq k} \gamma_m!} \times \frac{(r-2!)}{(2+\gamma_i)!\prod_{m \neq i} \gamma_m!} \\
        &\qquad\qquad \times w_i^2 w_k^2 \prod_m w_m^{2\gamma_m} \times \frac{r}{\frac{r!}{(2+\gamma_i)!(2+\gamma_k)!\prod_{m \neq i,k} \gamma_m!}}\\
        &= \frac{(r-2)(r-3)}{r-1} w_i^2 w_k^2 \sum_{\sum_m \gamma_m = r-4} \frac{(r-4)!}{\prod_m \gamma_m!}\prod_m w_m^{2\gamma_m}\\
        &= \frac{(r-2)(r-3)}{r-1} w_i^2 w_k^2.
    \end{align*}
    This proves \eqref{eq:cov-3}.
    \begin{align*}
        \Cov(M_{ij}, M_{il}) &= \sum_{\substack{\gamma_m \ge 0 \\ \sum_m \gamma_m = r-3}} \frac{(r-2)!}{(1+\gamma_l)! \prod_{m \neq l} \gamma_m! } \times \frac{(r-2)!}{(1+\gamma_j)!\prod_{m \neq j} \gamma_m!} \\
        &\qquad\qquad\times w_jw_l\prod_m w_m^{2\gamma_m} \times \frac{r}{\frac{r!}{(1+\gamma_i)!(1+\gamma_j)!(1+\gamma_l)!\prod_{m \neq i, j, l} \gamma_m!}}\\
        &= \frac{r-2}{r-1}w_jw_l \sum_{\sum_m \gamma_m =r-3} (1+\gamma_i) \frac{(r-3)!}{\prod_m \gamma_m!}\prod_m w_m^{2\gamma_m}\\
        &= \frac{r-2}{r-1} w_j w_l \bbE(1+X) \quad(\text{where}\, X \sim \Binomial(r-3, w_i^2))\\
        &= \frac{r-2}{r-1}w_j w_l + \frac{(r-2)(r-3)}{r-1} w_i^2 w_j w_l.
    \end{align*}
    This proves \eqref{eq:cov-4}. Finally, we prove \eqref{eq:cov-5}.
    \begin{align*}
        \Cov(M_{ii}, M_{il}) &= \sum_{\substack{\gamma_m \ge 0 \\ \sum_m \gamma_m = r-3}} \frac{(r-2)!}{(1+\gamma_l)!\prod_{m \neq l} \gamma_m!} \times \frac{(r-2!)}{(1+\gamma_i)!\prod_{m \neq i} \gamma_m!} \\
        &\qquad\qquad \times w_l w_i \prod_m w_m^{2\gamma_m} \times \frac{r}{\frac{r!}{(2+\gamma_i)!(1+\gamma_l)!\prod_{m \neq i,l} \gamma_m!}}\\
        &= \frac{r-2}{r-1} w_i w_l \sum_{\sum_m \gamma_m = r-3} (2+\gamma_i)\frac{(r-3)!}{\prod_m \gamma_m!}\prod_m w_m^{2\gamma_m}\\
        &= \frac{r-2}{r-1}w_i w_l\bbE(2+X) \quad (\text{where}\, X\sim \Binomial(r-3, w_i^2))\\
        &= \frac{2(r-2)}{r-1}w_i w_l + \frac{(r-2)(r-3)}{r-1} w_i^3 w_l.
    \end{align*}
    This completes the proof.
\end{proof}

\begin{proof}[Proof of Lemma~\ref{lem:rank-2-perturbation}]
    It can be checked easily using Lemma~\ref{lem:cov-structure-pure} and the definition of $X_n$ that for any $i, j, k, l \in [n]$, we have
    \[
        \Cov(M_{ij}, M_{kl}) = \Cov(X_{ij}, X_{kl}),
    \]
    which yields the desired conclusion since both $M_n$ and $X_n$ are Gaussian matrices.
\end{proof}

\subsection{Proof of Theorems~\ref{thm:lsd_pure}}
Our first goal is to prove Theorem~\ref{thm:lsd_pure}, for which we require one more ingredient. We need the following definitions first.
Let $(\cM, d)$ be a metric space. For a real-valued function $f$ on $\cM$, define its Lipschitz seminorm by $\|f\|_{\Lip} := \sup_{x \neq y} \frac{|f(x) - f(y)|}{d(x,y)}$. A function $f$ is called $l$-Lipschitz if $\|f\|_{\Lip} \le l$. Define the class of Bounded Lipschitz functions as 
\[
    \cF_{\BL} := \{ f \in \bbR^\cM : \|f\|_{\Lip} + \|f\|_{\infty} \le 1 \}.
\]
Then the bounded Lipschitz metric on the set $\cP(\cM)$ of probability measures on $\cM$ is defined as follows:
\[
    d_{\BL}(\mu, \nu) := \sup_{f \in \cF_{\BL}} \bigg\{\bigg|\int f \,d\mu - \int f \,d\nu\bigg|\bigg\}.
\]
It is well known that $d_{\BL}$ metrises weak convergence of probability measures on $\cP(\cM
)$ (see, e.g., \cite[Chap.~11]{dudley2018real}). Below we have $\cM = \bbR$.

The following lemma provides concentration inequalities for both the edge and bulk spectrum of mixed contractions.
\begin{lemma}\label{lem:concentration}
    Let $r \ge 3$ and $\bx_3, \ldots, \bx_r \in \bbS^{n - 1}$. Write $L_n := \cG \cdot \bx_3 \otimes \cdots \otimes \bx_r$. Then, for any $\epsilon > 0$,
    \begin{enumerate}
    \item[(i)] $\bbP(|\lambda_1(n^{-1/2}L_n) - \bbE \lambda_1(n^{-1/2}L_n)| > \epsilon) \le 2 \exp(- n\epsilon^2/ 2r)$, 
    \item[]
    \item[(ii)] $\bbP(d_\BL(\mu_{n^{-1/2}L_n}, \bar{\mu}_{n^{-1/2}L_n}) > \epsilon) \le \frac{2}{\epsilon^{3/2}} \exp(- n^2\epsilon^2/ 2r).$
    \end{enumerate}
\end{lemma}

\begin{proof}[Proof of Theorem~\ref{thm:lsd_pure}]
    Since LSD is invariant under finite rank perturbations, by virtue of Lemma~\ref{lem:rank-2-perturbation}, we may conclude that $\mubar_{\frac{1}{\theta\sqrt{n}}M_n} \convd \nu_{\sc, 1}$. Moreover the convergence is in probability because of the concentration result proved in Lemma~\ref{lem:concentration} whenever $r \ll n^2$. An additional application of the Borel-Cantelli Lemma proves the almost sure convergence in the regime $r \ll \frac{n^2}{\log {n}}$.   
\end{proof}

\subsection{Proof of Theorem~\ref{thm:univ_LSD}}
We introduce a few more notations. Recall that $\cS_{n,r}$ is the set of $r$-multisets of the set $[n]$. Define $m:= |\cS_{n,r}| = \binom{n+r-1}{r}$. Note that each of $\cT$ and $\cG$ consists of $m$ independent random variables and they can be naturally indexed by the elements of $\cS_{n,r}$. Fix an enumeration $\cS_{n,r} = \{S_1, \ldots, S_m\}$. Moreover, for any $S:= \{i_1, \ldots, i_r\} \in \cS_{n,r}$, let $\sigma_S^2 = \sigma^2_{i_1, \ldots, i_r}$ and define 
\begin{align}\label{eq:defn_T'_w_S_G'}
    w_{S} := \prod_{k=1}^r w_{i_k},\qquad
    \cT'_S := \frac{ \cT_{i_1, \ldots, i_r}}{\sigma_{S}},\qquad
    \cG'_S := \frac{ \cG_{i_1, \ldots, i_r}}{\sigma_S}.
\end{align}

Note that
\begin{align*}
    M_n &= \sum_{i_1, \ldots, i_r\in [n]} \cG_{i_1, \ldots, i_r } w_{i_3} \cdots w_{i_r} \be_{i_1} \be_{i_2}^\top\\
    &= \sum_{S \in \cS_{n,r}} \cG_S \sum_{\substack{i_1, i_2\\ \{i_1, i_2\} \subseteq' S}} \#\perms(S \backslash \{i_1, i_2\}) w_{S\backslash \{i_1, i_2\}} \be_{i_1}\be_{i_2}^\top,\\
    &= \sum_{S \in \cS_{n,r}} \cG'_S \sum_{\substack{i_1, i_2\\ \{i_1, i_2\} \subseteq' S}} \frac{\sqrt{r} \#\perms(S \backslash \{i_1, i_2\})}{\sqrt{\# \perms(S)}} w_{S\backslash \{i_1, i_2\}} \be_{i_1}\be_{i_2}^\top
    \end{align*}
where $\subseteq'$ denotes multiset inclusion. Define for $S \in \cS_{n,r}$,
\begin{equation}\label{eq:C_S_defn}
    C_S := \sum_{\substack{ i_1, i_2 \\ \{i_1, i_2\} \subseteq' S}} \frac{\sqrt{r} \#\perms(S \backslash \{i_1, i_2\})}{\sqrt{\# \perms(S)}} w_{S\setminus \{i_1, i_2\}} \be_{i_1}\be_{i_2}^\top.
\end{equation}
Then we have the following representations of $M_n$ and $\tilde{M}_n$.
\begin{align*}
    &M_n = \sum_{S \in \cS_{n,r}} \cG'_S C_S = \sum_{k = 1}^m \cG'_{S_k} C_{S_k},
    &\tilde{M}_n = \sum_{S \in \cS_{n,r}} \cT'_S C_S = \sum_{k = 1}^m \cT'_{S_k} C_{S_k}.
\end{align*}

The proof is accomplished by comparing the Stieltjes transform of $\tilde{M}_n$ with that of $M_n$. For $z= u + \iota v \in \bbC_+$, define the resolvents of $\sqrt{\frac{r - 1}{n}}M_n$ and $\sqrt{\frac{r - 1}{n}}\tilde{M}_n$ by
\begin{align*}
    R_n(z) := \left(\sqrt{\frac{r-1}{n}} M_n - zI\right)^{-1},\qquad \tilde{R}_n(z) := \left(\sqrt{\frac{r-1}{n}} \tilde{M}_n - zI\right)^{-1},
\end{align*}
respectively, and the corresponding Stieltjes transforms as 
\begin{align*}
    s_n(z) := \frac{1}{n} \Tr R_n(z),\qquad \tilde{s}_n(z) := \frac{1}{n} \Tr\tilde{R}_n(z),
\end{align*}
respectively. We recall the following universality result from \cite{vanhandel2016probability}, which will be used for proving Theorem \ref{thm:univ_LSD}.
\begin{theorem}[Theorem 9.2 in \cite{vanhandel2016probability}]\label{thm:univ_lsd_van_handel}
    Let $\bX$ and $\bY$ be random vectors in $\bbR^d$ with independent coordinates, and suppose that $\bbE[X_i]=\bbE[Y_i]$ and $\bbE[X_i^2]=\bbE[Y_i^2]$, for all $i$. Then, for every thrice differentiable $f:\bbR^d\to\bbR$, we have
    \begin{equation}\label{eq:univ_lsd_van_handel}
        |{\mathbb{E}f(\bX)-\mathbb{E}f(\bY)}|
        \le
        \frac{1}{6}
        \sum_{i=1}^d
        \left\|\frac{\partial^3 f}{{\partial x_i}^3}\right\|_{\infty}
        \bbE \left[|X_i|^3+|Y_i|^3\right].
    \end{equation}
\end{theorem}
We shall also need following result in the proof of Theorem \ref{thm:univ_LSD}.

\begin{lemma}\label{lem:C_S_bounds}
    Let $C_S$ be defined as in \eqref{eq:C_S_defn}. Then, we have the following bound:
    \begin{align}
        \|C_S\|_F \le 2 \sqrt{(r - 1)!} \|\bw\|_{\infty}^{r - 2}.\label{eq:C_S_frobenius} 
    \end{align}
\end{lemma}

\begin{proof}[Proof of Theorem~\ref{thm:univ_LSD}]
    Recall that $m = |\cS_{n, r}|$. Let for $\bx \in \bbR^m$
    \begin{align*}
        M_n(\bx) &:= \sum_{k=1}^m x_k C_{S_k},\\
        R_n(\bx) &:= \left( \sqrt{\frac{r-1}{n}} M_n(\bx) - zI\right)^{-1}.
    \end{align*}
    
    Keeping Theorem~\ref{thm:univ_lsd_van_handel} in mind, define, for each fixed $z \in \bbC_+$, functions $h_1, h_2 : \bbR^m \to \bbR$, by 
    \begin{align*}
        h_1(\bx) &:= \frac{1}{n} \Re \Tr R_n(\bx),\\
        h_2(\bx) &:= \frac{1}{n} \Im \Tr R_n (\bx),
    \end{align*}
    where $\bx := (x_1, \ldots, x_m)$. We also define $h : \bbR^m \to \bbC$ by $h= h_1 + \iota h_2$. Then,
    \begin{align*}
        &s_n(z) = h(\cG'),
        &\tilde{s}_n(z) &= h(\cT'),
    \end{align*}
    where $\cG'= (\cG'_{S_k})_{k=1}^m$ and $\cT'= (\cT'_{S_k})_{k=1}^m$, respectively.
    By an application of Theorem~\ref{thm:univ_lsd_van_handel}, we can write
    \begin{align*}
        |\bbE [s_n(z)] - \bbE[\tilde{s}_n (z)]| &\le |h_1(\cG') - h_1(\cT')| + |h_2(\cG') - h_2(\cT')|\\
        &\le \frac{1}{6}
        \sum_{k=1}^m \left(
        \left\|\frac{\partial^3 h_1}{{\partial x_k}^3}\right\|_{\infty} + \left\|\frac{\partial^3 h_2}{{\partial x_k}^3}\right\|_{\infty} \right)
        \bbE \left[|\cG'_{S_k}|^3+|\cT'_{S_k}|^3\right] \\
        &\le \frac{1}{3}
        \sum_{k=1}^m 
        \left\|\frac{\partial^3 h}{{\partial x_k}^3}\right\|_{\infty}
        \bbE \left[|\cG'_{S_k}|^3+|\cT'_{S_k}|^3\right].
    \end{align*}
    Differentiating the relation \[
        \left(\sqrt{\frac{r-1}{n}} M_n(\bx) - z I\right) R_n(\bx) = I,
    \]
    one obtains that
    \begin{align*}
        \frac{\partial R_n(\bx) }{\partial x_{k}} &= - \sqrt{\frac{r-1}{n}} R_n(\bx) \frac{\partial M_n(\bx) }{\partial x_{k}} R_n(\bx), \\
        \frac{\partial^2 R_n(\bx) }{\partial x^2_{k}} &= \frac{2(r-1)}{n}  R_n(\bx) \frac{\partial M_n(\bx) }{\partial x_{k}} R_n(\bx) \frac{\partial M_n(\bx) }{\partial x_{k}} R_n(\bx), \\
        \frac{\partial^3 R_n(\bx) }{\partial x^3_{k}} &= - \frac{6(r-1)^{\frac{3}{2}}}{n^{\frac{3}{2}}}  R_n(\bx) \frac{\partial M_n(\bx) }{\partial x_{k}} R_n(\bx) \frac{\partial M_n(\bx) }{\partial x_{k}} R_n(\bx)\frac{\partial M_n(\bx) }{\partial x_{k}} R_n(\bx),
    \end{align*}
    and thus
    \begin{align*}
        \frac{\partial^3 h(\bx) }{\partial x^3_{k}} &= - \frac{6(r-1)^{\frac{3}{2}}}{n^{\frac{5}{2}}} \Tr\bigg(\frac{\partial M_n(\bx)}{\partial x_{k}} R_n(\bx)  \frac{\partial M_n(\bx) }{\partial x_{k}} R_n(\bx) \frac{\partial M_n(\bx) }{\partial x_{k}} R^2_n(\bx)\bigg).
    \end{align*}
    Now using the fact that $\| R_n(z) \|_\op \le \frac{1}{\Im(z)}$ and the usual matrix norm inequalities, we get
    \begin{align*}
        \left| \frac{\partial^3 h(\bx) }{\partial x^3_{k}} \right| &\le  \frac{6 r^{\frac{3}{2}}}{n^{\frac{5}{2}} (\Im(z))^4} \left\|  \frac{\partial M_n(\bx) }{\partial x_{k}} \right\|_F^3. 
    \end{align*}
    Note that
    \[
        \frac{\partial M_n(\bx) }{\partial x_{k}} = C_{S_k}.
    \]
    By virtue of Lemma~\ref{lem:C_S_bounds}, we get
    \begin{align*}
        \left\| \frac{\partial^3 h(\bx) }{\partial x^3_{k}} \right\|_\infty &\le  \frac{48 (r!)^{\frac{3}{2}} \| \bw \|_\infty^{3(r-2)}}{n^{\frac{5}{2}} (\Im(z))^4}.
    \end{align*}
    Thus,
    \begin{align*}
        \bbE [s_n(z)] - \bbE[\tilde{s}_n (z)] = O \left( \frac{ m (r!)^{\frac{3}{2}} \| \bw \|_\infty^{3(r-2)} \max\{1, \mu_n\}}{n^{\frac{5}{2}}} \right).
    \end{align*}
    Now,
    \begin{align*}
        & \frac{m (r!)^{\frac{3}{2}} \| \bw \|_\infty^{3(r-2)} \max\{1, \mu_n\}}{n^{\frac{5}{2}}} \\
        &= \frac{(n+r -1)! (r!)^{\frac{1}{2}} \| \bw \|_\infty^{3(r-2)} \max\{1, \mu_n\}}{n^{\frac{5}{2}} (n-1)!} \\
        &\asymp \exp \bigg( (\log(n+r -1) - 1)(n+r-1) + \frac{1}{2} \log (n+r-1) + \frac{1}{2} (\log r -1) r + \frac{1}{4} \log r \\
        & - (\log n -1) n - \frac{1}{2} \log n - \frac{3}{2} \log n + 3(r-2) \log \|\bw\|_\infty + \log \max\{1, \mu_n\}  \bigg) \\
        &\asymp \exp \bigg( \bigg(n+r -\frac{1}{2} \bigg) \log (n+r -1) + \frac{1}{4} (2 r +1) \log r - (n+2) \log n \\
        &+ 3(r-2) \log \|\bw\|_\infty + \log \max\{1, \mu_n\}  \bigg).
    \end{align*}
    The delocalisation condition in \eqref{eq:univ_lsd_deloc} ensures the term in the last display converges to $0$.
\end{proof}

\begin{proof}[Proof of Corollary~\ref{cor:sparse_tensor_lsd}]
    Clearly, $\overline{\cT}_n$ is centered and
    \begin{align*}
        \Var(\overline{\cT}_{i_1, \ldots, i_r}) &= \bbE \mathcal{B}_{i_1, \ldots, i_r}^2 \bbE \cT_{i_1, \ldots, i_r}^2\\
        &= p_n \sigma_{i_1, \ldots, i_r}^2
    \end{align*}
    The proof is complete by an application of Theorem~\ref{thm:univ_LSD} on the tensor $\frac{1}{\sqrt{p_n}}\overline{\cT}_n$. 
\end{proof}

\subsection{Proof of Theorem~\ref{thm:edge-limits}}
To get a handle on the edge eigenvalues, we need to understand the eigen-structure of the rank-$2$ perturbation in \eqref{eq:rank-2-perturbation}. This is the content of the following lemma.
\begin{lemma}\label{lem:eigen-structure}
    Let $B = a \bx \bx^\top + b(\bx\by^\top + \by \bx^\top)$, where $\bx, \by \in \bbS^{n - 1}$ are linearly independent, $a, b \in \bbR$ and $b \ne 0$. Let $\lambda_{\pm}$ be the two non-trivial eigenvalues of $B$ and let $\bv_{\pm}$ denote the corresponding eigenvectors (up to signs). The we have the following.
    \begin{enumerate}
    \item [(i)] $\lambda_{\pm} = \frac{a + 2b \bx^\top \by \pm \sqrt{a^2 + 4ab \bx^\top \by + 4 b^2}}{2}$.
    \item [(ii)]  Writing $\bv_{\pm} = \gamma_{\pm} \bx + \delta_{\pm} \by$ with $\delta_{\pm} \ge 0$, we have
    \[
        \gamma_{\pm} = \sign(b)\frac{\lambda_{\pm} - b \bx^\top \by}{s_{\pm}}, \qquad \delta_{\pm} = \frac{|b|}{s_{\pm}}, 
    \]
    where
    \[
        s_{\pm} = \sqrt{\lambda_{\pm}^2 + b^2 (1 - (\bx^\top \by)^2)}.
    \]
    \item [(iii)] $\bx^\top \bv_{\pm} = \sign(b) \frac{\lambda_{\pm}}{s_{\pm}}$.
    \end{enumerate}
\end{lemma}

In the next lemma, we look at the fluctuations of the maximum and minimum eigenvalues of the rank-$2$ perturbation
\[
    P_n := \alpha U \bw \bw^\top + \beta (\bw \bV^\top + \bV \bw^\top).
\]
\begin{lemma}\label{lem:P_n_fluctuation}
    \begin{enumerate}
        \item[]
        \item[]
        \item[(i)] If $r$ is fixed, then 
            \begin{equation}\label{eq:P_n_fluctuation_r_fixed}
                \sqrt{n}\bigg(\frac{\lambda_1(P_n)}{\sqrt{n}} - \beta, \frac{\lambda_n(P_n)}{\sqrt{n}} + \beta\bigg) \convd \cN_2\left(\begin{pmatrix} 0\\0\end{pmatrix}, \begin{pmatrix}\frac{\alpha^2}{4} + \frac{3\beta^2}{2}& \frac{\alpha^2}{4} + \frac{\beta^2}{2}\\ \frac{\alpha^2}{4} + \frac{\beta^2}{2} & \frac{\alpha^2}{4} + \frac{3\beta^2}{2}\end{pmatrix}\right).
    \end{equation}
        \item[(ii)] If $1\ll r\ll n$, then
        \begin{equation}\label{eq:P_n_fluctuation_1<<r<<n}
            \sqrt{\frac{n}{r}}\bigg(\frac{\lambda_1(P_n)}{\sqrt{n}} - \beta, \frac{\lambda_n(P_n)}{\sqrt{n}} + \beta\bigg) \convd \cN_2\left(\begin{pmatrix} 0\\0\end{pmatrix}, \frac{1}{4}\begin{pmatrix}
                1 & 1\\
                1 & 1
            \end{pmatrix}\right).
        \end{equation}
        If $\sqrt[3]{n} \ll r \ll n$, then $\beta$ can be replaced by $1$.
        \item[(iii)] If $r/n \to c \in (0, \infty)$, then
            \begin{equation}\label{eq:P_n_fluctuation_r/n=c}
                \bigg(\frac{\lambda_1(P_n)}{\sqrt n}, \frac{\lambda_n(P_n)}{\sqrt n}\bigg) \convd \bigg(\xi, -\frac{1}{\xi}\bigg),
            \end{equation}
            where $\xi = \frac{\sqrt{c}\zeta + \sqrt{c\zeta^2 + 4}}{2}$ with $\zeta \sim \cN(0, 1)$.
        \item[(iv)] If $r/n \to \infty$, then
            \begin{equation}\label{eq:P_n_fluctuation_r>>n}
                \bigg(\frac{\lambda_1(P_n)}{\sqrt r}, \frac{\lambda_n(P_n)}{\sqrt r}\bigg) \convd (\zeta_+, -\zeta_-),
            \end{equation}
            where $\zeta \sim \cN(0, 1)$, and $x_+ = \max\{x, 0\}, x_- = \max\{-x, 0\}$ denote respectively the positive and the negative parts of $x$.
    \end{enumerate}
\end{lemma}
\begin{proof}[Proof of Theorem~\ref{thm:edge-limits}]
     Write $P_n = \alpha U \bw \bw^\top + \beta (\bw \bV^\top + \bV \bw^\top)$. Then
    \begin{equation}\label{eq:closensess_of_G_and_P}
        \max_{1 \le i \le n}|\lambda_i(X_n) - \lambda_i(P_n)| \leq \theta \|Z_n\|_{\op} = O_P(\sqrt n).
    \end{equation}
    Noting that $\bw$, $\bV$ are almost surely linearly independent, from Lemma~\ref{lem:eigen-structure}, we immediately obtain
    the eigenvalues of $P_n$:
    \begin{align}
        \lambda_1(P_n) &=  \frac{1}{2} \alpha U + \beta \bw^\top \bV + \frac{1}{2} \sqrt{\alpha^2 U^2 + 4\alpha\beta U \bw^\top \bV + 4 \beta^2 \|\bV\|^2}, \label{eq:eigen_lambda_1}\\
        \lambda_n(P_n) &= \frac{1}{2} \alpha U + \beta \bw^\top \bV - \frac{1}{2} \sqrt{\alpha^2 U^2 + 4\alpha\beta U \bw^\top \bV + 4 \beta^2\|\bV\|^2}, \label{eq:eigen_lambda_n}\\
        \lambda_k(P_n) &= 0 \quad \text{for } k = 2, 3, \ldots, n - 1.
    \end{align}
    We think of $X_n$ as a low rank deformation of a scaled GOE matrix:
    \[
        X_n = P_n + \theta Z_n.
    \]
    It is clear that (since $\frac{\bw^\top \bV}{\sqrt{n}} \convas 0$ as $\bw^\top \bV \sim \cN(0, 1)$, and $\frac{\|\bV\|}{\sqrt{n}} \convas 1$)
    \begin{align}\label{eq:lambda_1_and_n_convergence}
        \frac{\lambda_1(P_n)}{\sqrt n} \convas \beta \quad \text{and} \quad \frac{\lambda_n(P_n)}{\sqrt n} \convas -\beta.
    \end{align}
    Suppose $\bu_1$ and $\bu_n$ are a orthonormal pair of eigenvectors corresponding to $\lambda_1$ and $\lambda_n$, respectively. Then $ P_n$ has the representation
    \[
        P_n = \lambda_1(P_n) \bu_1 \bu_1^\top + \lambda_n(P_n) \bu_n \bu_n^\top.
    \]
    Define $\tilde P_n = \sqrt{n}\beta(\bu_1 \bu_1^\top - \bu_n \bu_n^\top)$. Now, by virtue of \eqref{eq:lambda_1_and_n_convergence}, we may conclude that $\frac{1}{\sqrt n}\|P_n - \tilde{P}_n\|_\op \convas 0$. Therefore by Weyl's inequality, it is enough to consider $\frac{1}{\sqrt{n}}\tilde{X}_n$ instead of $\frac{1}{\sqrt n} X_n$, where $\tilde{X}_n$ is defined as 
    \[
        \tilde{X}_n = \tilde{P}_n + \theta Z_n.
    \]
    Now note that $\tilde{P}_n$ and $Z_n$ are independent and $Z_n$ is orthogonally invariant. Further, the LSD of $\theta Z_n$ is the semicircle law with variance $\theta$. Hence one may apply Theorem~2.1 of \cite{benaych2011eigenvalues} on $\tilde X_n$ to conclude that
    \begin{align*}
        \frac{\lambda_1(\tilde{X_n})}{\sqrt n} \convas  \begin{cases}
            \beta + \frac{\theta^2}{\beta} & \text{if } \beta > \theta, \\
            2\theta & \text{otherwise}.
        \end{cases}
    \end{align*}
    Similarly,
    \begin{align*}
        \frac{\lambda_n(\tilde{X_n})}{\sqrt n} \convas \begin{cases}
            -\big(\beta + \frac{\theta^2}{\beta}\big) & \text{if } \beta > \theta, \\
            -2\theta & \text{otherwise}.
        \end{cases}
    \end{align*}
    Note that $\beta > \theta$ if and only if $r - 2 > 1$, i.e. $r \ge 4$. From this we conclude that
    \[
        \frac{\lambda_1(X_n)}{\sqrt{n}} \convas \varpi_r \qquad \text{and} \qquad \frac{\lambda_n(X_n)}{\sqrt{n}} \convas -\varpi_r.
    \]
    This completes the proof of \eqref{eq:edge_prob_conv_r_fixed}.

    By Weyl's inequality,
    \[
        \max_{1 \le i \le n} |\lambda_i(X_n) - \lambda_i(P_n)| \le \theta \|Z_n\|_{\op} = O_{\bbP}(\sqrt{n/r}).
    \]
    Therefore, when $r \gg 1$, it is enough to consider the limits of $\lambda_1(P_n)$ and $\lambda_n(P_n)$.
    
    In the regime $1 \ll r \ll n$, one has that \begin{align}\label{eq:lambda_1_and_n_convergence_r_gg_1}
        \frac{\lambda_1(P_n)}{\sqrt n} \convas 1 \quad \text{and} \quad \frac{\lambda_n(P_n)}{\sqrt n} \convas -1
    \end{align}
    proving \eqref{eq:edge_prob_conv_r<<n}. In order to prove the almost sure convergence we make use of Lemma~\ref{lem:concentration}-(i) and a Borel-Cantelli argument. Similarly, \eqref{eq:edge_dist_conv_r=n} and \eqref{eq:edge_dist_conv_r>>n} follow from \eqref{eq:P_n_fluctuation_r/n=c} and \eqref{eq:P_n_fluctuation_r>>n}, respectively. This finishes the proof.
\end{proof}

\subsection{Proof of Theorem~\ref{thm:univ_edge}}
We now prove the universality result for the extreme edge eigenvalues of the contracted matrix $\tilde{M}_n$ from \eqref{eq:tilde_M_n}. As in the proof of Theorem~\ref{thm:univ_LSD}, we write $\tilde{M}_n$ as a sum of independent matrices:
\[
    \tilde{M}_n = \sum_{S \in \cS_{n,r}} \cT'_S C_S,
\]
where $\cT'_S$ and $C_S$ are defined in \eqref{eq:defn_T'_w_S_G'} and \eqref{eq:C_S_defn}, respectively. Define
\begin{align*}
    \sigma(\tilde{M}_n) &:= \|\bbE (\tilde{M}_n^2)\|_{\op}^{1/2}, \\
    \sigma_*(\tilde{M}_n) &:= \sup_{\|\bv\| = \|\bw\| = 1} \bbE[|\langle \bv, \tilde{M}_n  \bw\rangle|^2]^{1/2}, \\
    R(\tilde{M}_n) &:= \bigg\| \max_{1 \le i \le m} \cT'_S \|C_S\|_{\op}\bigg\|_{\infty}.
\end{align*}
The following lemma is used in the proof. We state it here.
\begin{lemma}\label{lem:comp_spectrum}
    We have the following estimates.
    \begin{enumerate}
        \item[(i)] $\sigma(\tilde{M}_n) = O(\sqrt{n} + \sqrt{r})$,
        \item[(ii)] $\sigma_*(\tilde{M}_n)= O(\sqrt{r})$,
        \item[(iii)] $R(\tilde{M}_n)=O(K_n \sqrt{(r-1)!} \|\bW\|_\infty^{r-2} )$.
    \end{enumerate}
\end{lemma}

\begin{proof}[Proof of Theorem~\ref{thm:univ_edge}]
    Since we have a representation of $\tilde{M}_n$ as a sum of independent matrices, by Corollary~2.7 of \cite{brailovskaya2024universality}, there exists a universal constant $C >0$ such that for any $t >0$,
    \begin{equation}\label{eq:lambda_1_univ}
        \bbP(|\lambda_{1}(\tilde{M}_n) - \lambda_1(G_n)|) > C \varepsilon(t)) \le n e^{-t},
    \end{equation}
    where $G_n$ is a Gaussian random matrix with the same expectation and covariance structure as $\tilde{M}_n$ and 
    \[
        \varepsilon(t) = \sigma_*(\tilde{M}_n) t^{1/2} + R(\tilde{M}_n)^{1/3} \sigma(\tilde{M}_n)^{2/3} t^{2/3} + R(\tilde{M}_n) t,
    \]
    We shall put $t = 3 \log n$ so as to obtain an error probability of $\frac{1}{n^2}$ in \eqref{eq:lambda_1_univ}. 
    By Lemma~\ref{lem:comp_spectrum}, $\varepsilon(3 \log n)$ turns out to be 
    \begin{align*}
        \varepsilon(3 \log n) &= O( (\log n)^{\frac{1}{2}} r^{\frac{1}{2}} +  K_n^{\frac{1}{3}} (r-1)!^{\frac{1}{6}} (\log n)^{\frac{2}{3}} \|\bw\|_\infty^{\frac{r-2}{3}} \max \{n^{\frac{1}{3}}, r^{\frac{1}{3}}\} \\ 
        &\qquad\qquad + K_n (\log n) (r-1)!^{\frac{1}{2}} \|\bw\|_\infty^{r-2}).
    \end{align*}
    We need to make sure that $\frac{1}{\sqrt{n}}\varepsilon(3 \log n) \to 0$. For the first term to be negligible, we need
    \[
        r = o \left(\frac{n}{\log n} \right).
    \]
    This in particular implies that $r = o(n)$. So, in the second term in the expression of $ \varepsilon(3 \log n)$ becomes $O( K_n^{\frac{1}{3}} (r-1)!^{\frac{1}{6}}(\log n)^{\frac{2}{3}} \|\bw\|_\infty^{\frac{r-2}{3}} n^{\frac{1}{3}})$. We thus need to make sure that 
    \[ 
        \max \{ K_n^{\frac{1}{3}} (r-1)!^{\frac{1}{6}} n^{-\frac{1}{6}} (\log n)^{\frac{2}{3}} \|\bw\|_\infty^{\frac{r-2}{3}}, K_n n^{-\frac{1}{2}} (r-1)!^{\frac{1}{2}} \log n \|\bw\|_\infty^{r-2} \} \to 0,
    \]
    or, equivalently,
    \[
        \max \{ K_n (r-1)!^{\frac{1}{2}} n^{-\frac{1}{2}} (\log n)^2 \|\bw\|_\infty^{r-2}, K_n n^{-\frac{1}{2}} (\log n) (r-1)!^{\frac{1}{2}} \|\bw\|_\infty^{r-2} \} \to 0,
    \]
    i.e., 
    \[
        K_n n^{-\frac{1}{2}} (r-1)!^{\frac{1}{2}} (\log n)^2 \|\bw\|_\infty^{r-2}  \to 0.
    \]
    This finishes the proof of Theorem~\ref{thm:univ_edge}.
\end{proof}

\begin{proof}[Proof of Corollary~\ref{cor:sparse_tensor}]
    Clearly, $\overline{\cT}_n$ is centered and
    \begin{align*}
        \Var(\overline{\cT}_{i_1, \ldots, i_r}) &= \bbE \mathcal{B}_{i_1, \ldots, i_r}^2 \bbE \cT_{i_1, \ldots, i_r}^2\\
        &= p_n \sigma_{i_1, \ldots, i_r}^2
    \end{align*}
    The proof is complete by an application of Theorem~\ref{thm:univ_edge} on the tensor $\frac{1}{\sqrt{p_n}}\overline{\cT}_n$. 
\end{proof}

\subsection{Proof of Theorem~\ref{thm:edge-fluctuations}}
\begin{proof}[Proof of Theorem~\ref{thm:edge-fluctuations}]
    \item[(i)] We work with $X_n$ instead of $M_n$. For $r=3$, as one can check, the two non-zero eigenvalue-eigenvector pairs of $\frac{P_n}{\sqrt{n}}$ are given by
    \begin{align}
        \bigg(\frac{\lambda_1(P_n)}{\sqrt{n}}, \bu_1\bigg) &= \bigg(\frac{1}{\sqrt{n}}(\|\bV\|+ \bw^\top \bV), \,\frac{1}{2\big(1 + \frac{\bw^\top \bV}{\|\bV\|}\big)} \bigg( \bw + \frac{\bV}{\|\bV\|}\bigg)\bigg),\label{eq:lambda_1_P_n}\\ \bigg(\frac{\lambda_n(P_n)}{\sqrt{n}}, \bu_n\bigg) &= \bigg(\frac{1}{\sqrt{n}}(-\|\bV\|+ \bw^\top \bV), \,\frac{1}{2\big(1 - \frac{\bw^\top \bV}{\|\bV\|}\big)} \bigg( \bw - \frac{\bV}{\|\bV\|}\bigg)\bigg). \label{eq:lambda_n_P_n}
    \end{align}
    We now define a new basis of $\bbR^n$. Let $\mathcal{B} := \{\bu_1, \bu_2, \ldots, \bu_n\}$, where the $\bu_2, \ldots, \bu_{n-1}$ are chosen so that $\mathcal{B}$ is an orthonormal basis of $\bbR^n$. Let $[A]_{\mathcal{B}}$ denote the representation of a matrix $A$ in the basis $\mathcal{B}$. Since the GOE is orthogonally invariant, by an abuse of notation $[Z_n]_\mathcal{B}$ will also be written as $Z_n$. So we have
    \[
        \frac{1}{\sqrt{n}}[X_n]_\mathcal{B} =  \frac{1}{\sqrt{n}}(\|\bV\|+ \bw^\top \bV) \be_1 \be_1^\top + \frac{1}{\sqrt{n}}(-\|\bV\|+ \bw^\top \bV) \be_n \be_n^\top + \frac{1}{\sqrt{n}}Z_n.
    \]
    Since eigenvalues are invariant under change of basis $[X_n]_\mathcal{B}$ has the same eigenvalues as $X_n$. We are now in a position to apply Theorem~1.2 of \cite{bloemendal2016limits}. To this end notice that
    \begin{align*}
        n^{1/3}\bigg\{1 - \frac{1}{\sqrt{n}}(\|\bV\| + \bw^\top \bV )\bigg\} &= n^{-1/6} \frac{\sqrt{n}\big(1 - \frac{\|\bV\|^2}{n})}{1 + \frac{\|\bV\|}{\sqrt{n}}} + n^{-1/6} \bw^\top \bV \convas 0,\\
        n^{1/3}\bigg\{1 - \frac{1}{\sqrt{n}}(-\|\bV\| + \bw^\top \bV )\bigg\} &\convas \infty.
    \end{align*}
    Now an application of Theorem~1.2 of \cite{bloemendal2016limits} yields
    \[
        n^{2/3}\bigg(\frac{1}{\sqrt{n}}\lambda_1(X_n) - 2\bigg) \convd -\Lambda_0.
    \]
    where $\Lambda_0$ is the largest eigenvalue of $\mathcal{H}_{1,(0, \infty)}$. Further, a remark on page 9 of \cite{bloemendal2016limits} shows that the distribution of the largest eigenvalue is same for $\mathcal{H}_{1, (0, \infty)}$ and $\mathcal{H}_{1, (0)}$. This proves \eqref{eq:edge_fluc_1_r=3}. The proof is identical for \eqref{eq:edge_fluc_n_r=3}. 

    \item[(ii)] We now prove \eqref{eq:edge_fluc_r_fixed}. Since $M_n$ has the same distribution as $X_n$, in the proof we shall work with $X_n$ instead of $M_n$. Define 
   \begin{align*}
   X'_n&:=\frac{X_n}{\theta \sqrt n},\\
   P'_n  &:= \frac{\alpha'}{\sqrt n} U \bw \bw^{\top} + \frac{\beta'}{\sqrt n} (\bV\bw^\top + \bw \bV^\top),\\ 
   Z'_n &= \frac{1}{\sqrt n} Z_n,\\
   \alpha' &= \frac{\alpha}{\theta}=\sqrt{(r-2)(r-3)},\\ 
   \beta' &= \frac{\beta}{\theta} = \sqrt{r-2}.
   \end{align*}
   Notice that $Z'_n$ is a GOE and hence $X'_n = P'_n + Z'_n$ is a rank-2 perturbation of a GOE. Now we can apply Theorem 2.11 of \cite{knowles2014outliers} which provides the joint fluctuation of the outlier eigenvalues (defined appropriately) for spiked Wigner matrices. In this theorem, the perturbation is assumed to be deterministic. But, in the current scenario $P'_n$ is random, but independent of $Z'_n$. So, the idea is to condition on $U, \bV$ and apply the theorem conditionally. In order to apply the mentioned theorem, one needs several conditions on the eigenvalues of $P'_n$. These conditions are translated in the following events:
    \begin{align*}
        E^{(1)}_n &= \{\lambda_1(P'_n) \leq \beta' + 1,\, \lambda_n(P'_n) \geq -\beta' -1\}, \\
        E^{(2)}_n &= \{|\lambda_1(P'_n)| \geq 1 + (\log\log n)^{k\log n} n^{-1/3}, \,|\lambda_n(P'_n)| \geq 1 + (\log\log n)^{k\log n} n^{-1/3} \},\\
        E^{(3)}_n &= \{\sqrt n (|\lambda_1(P'_n)| - 1)^{1/2}|\lambda_1(P'_n) - \lambda_n(P'_n)| \geq n^{1/3}\},\\
        E_n &= E^{(1)}_n \cap E^{(2)}_n \cap E^{(3)}_n . 
    \end{align*}
    The expressions for $\lambda_1(P_n)$ and $\lambda_n(P_n)$ are available in equations~\eqref{eq:eigen_lambda_1} and \eqref{eq:eigen_lambda_n}. From the proof of Lemma~\ref{lem:P_n_fluctuation}, $\lambda_1(P'_n) \convas \beta'$ and $\lambda_n(P'_n) \convas -\beta'$. Thus, $\bbP(E^{(1)}_1) \to 1$. Since $\beta' = \sqrt{r-2} \geq \sqrt{2}$ for $r\geq 4$ and 
    \[
        (\log\log n)^{k\log n} n^{-1/3} \to 0,
    \]
    $\bbP(E^{(2)}_n) \to 1$. Also, on $E^{(2)}_n$,
    \[
        (|\lambda_1(P'_n)| - 1) \geq (\log \log n)^{k \log n} n^{-1/3}.
    \] 
    Moreover, $\lambda_1(P'_n) - \lambda_n(P'_n) \convas 2\beta'$. Therefore, 
    \[
        \sqrt n (|\lambda_1(P'_n)| - 1)^{1/2}|\lambda_1(P'_n) - \lambda_n(P'_n)| = \Omega_\bbP((\log \log n)^{k\log n/2} n^{1/3}),
    \]
    in particular, $\bbP(E^{(3)}) \to 1$ as $n \to \infty$. Thus, $\bbP(E_n) \to 1$ as $n \to \infty$. Now, fix arbitrary $f : \bbR^2 \to \bbR$ which is bounded and continuous. Also, define $\tilde{f}: \bbR \to \bbR$ as $\tilde{f} (x) = 0$ for all $x\in \bbR$. Fix $\epsilon > 0$. On the set $E_n$, $(\lambda_1(X'_n), \lambda_n(X'_n))|(U, \bV)$ satisfies the conditions required in Theorem 2.11 of \cite{knowles2014outliers}. Applying the theorem for $\epsilon/2$ and the functions $\tilde{f}$ and $f$, we can get $s_0 \geq 0$ and $N_0 \geq 1$. Further we can find $N_1$, such that $s_0 \leq N_1$. Note that if $n \geq N_1$, then on $E_n$, $\lambda_1(P'_n)$ and $\lambda_n(P'_n)$ are in different groups of outliers. Define $N = \max\{N_0, N_1\}$. Then, for $n \geq N$, 
    \begin{equation}\label{eq:closeness_zeta_xi}
        |\bbE(f(\zeta_n, \tilde \zeta_n)|U, \bV) - \bbE(f(\eta(\lambda_1(P'_n))\xi, \eta(\lambda_n(P'_n))\tilde\xi)|U,\bV)| \leq \epsilon/2,
    \end{equation}
    where 
    \begin{align*}
    \zeta_n &= \sqrt n (|\lambda_1(P'_n)| -1)^{-1/2} (\lambda_1(X'_n) -\sigma(\lambda_1(P'_n))),\\
    \tilde\zeta_n &= \sqrt n (|\lambda_n(P'_n)| -1)^{-1/2} (\lambda_n(X'_n) -\sigma(\lambda_n(P'_n))),\\
    \end{align*}
    $\xi, \tilde\xi$ are standard Gaussian random variables independent of everything, and $\sigma, \eta:\bbR\backslash\{0\} \to \bbR$ are defined as
    \begin{align*}
        &\sigma(x) = x+1/x,   &\eta(x) = \frac{\sqrt{2(|x|+1)}}{|x|}.
    \end{align*}
    Since $f$ is arbitrary, for arbitrary $g :\bbR^2 \to \bbR$, we can look at the (potentially $U$, $\bV$ dependent) function 
    \begin{align*}
        f((x, x')) = g(((|\lambda_1(P'_n)| -1)^{1/2}x+\sqrt n(\sigma(\lambda_1(P'_n)) - \sigma(\beta')),\\ (|\lambda_n(P'_n)| -1)^{1/2}x' + \sqrt n(\sigma(\lambda_n(P'_n)) - \sigma(-\beta')))).
    \end{align*}
    First notice that with this choice of $f$, 
    \[
        f(\zeta, \tilde\zeta) = g(\sqrt n (\lambda_1(X'_n) - \sigma(\beta')), \sqrt n(\lambda_n(X'_n) - \sigma(-\beta'))).
    \]
    So, \eqref{eq:closeness_zeta_xi} can be rewritten as
    \begin{align*}
        |\bbE(g(\sqrt n (\lambda_1(X'_n) - \sigma(\beta')), \sqrt n(\lambda_n(X'_n) - \sigma(-\beta')))|U,\bV) -\\ \bbE(f(\eta(\lambda_1(P'_n))\xi, \eta(\lambda_n(P'_n))\tilde\xi)|U,\bV)| \leq \epsilon/2,
    \end{align*}
    on the set $E_n$ for all large $n$. Since $\bbP(E_n) \to 1$ and $g$ and hence $f$ are bounded functions, using DCT type argument, we can say that the previous inequality is true unconditionally. Thus, it is now enough to find the limit of $\bbE(f(\eta(\lambda_1(P'_n))\xi, \eta(\lambda_n(P'_n))\tilde\xi))$, that is of
    \begin{align*}
        \bbE(g(((|\lambda_1(P'_n)| -1)^{1/2}\eta(\lambda_1(P'_n)) \xi+\sqrt n(\sigma(\lambda_1(P'_n)) - \sigma(\beta')),\\ (|\lambda_n(P'_n)| -1)^{1/2}\eta(\lambda_n(P'_n))\tilde\xi + \sqrt n(\sigma(\lambda_n(P'_n)) - \sigma(-\beta'))))).
    \end{align*}
    Equivalently, we shall find the distributional limit of the bivariate random variable above. Using the facts that $\lambda_1(P'_n) \to \beta'$ and $\lambda_n(P'_n) \to -\beta'$ almost surely, we can conclude that
    \begin{align}\label{eq:as_convergence}
        (|\lambda_1(P'_n)| -1)^{1/2}\eta(\lambda_1(P'_n)) &\to (\beta' -1)^{1/2}\eta(\beta'),\\
        (|\lambda_n(P'_n)| -1)^{1/2}\eta(\lambda_n(P'_n)) &\to (\beta' -1)^{1/2}\eta(\beta')
    \end{align}
    almost surely. Let us look at $\sqrt n(\sigma(\lambda_1(P'_n)) - \sigma(\beta'))$ and $\sqrt n(\sigma(\lambda_n(P'_n)) - \sigma(-\beta'))$. Note that its limit is available in \eqref{eq:P_n_fluctuation_r_fixed}. Fix $a, b\in \bbR$. Define the function $\tilde\sigma: \bbR^2\backslash((\{0\}\times \bbR) \cup (\bbR \times \{0\})) \to \bbR$, 
    \[
        \tilde\sigma(x,y) = a\sigma(x) + b\sigma(y).
    \]
    We apply the Delta method on $(\lambda_1(P'_n), \lambda_n(P'_n))$ with the centering $(\beta', -\beta')$ respect to the scalar function $\tilde\sigma$. We can compute that
    \[
        \nabla \tilde\sigma (x,y) = (a \sigma'(x), b\sigma'(y)) = \bigg(a - \frac{a}{x^2}, b- \frac{b}{y^2}\bigg).
    \]
    Then, by the Delta method, 
    \[
        a\sqrt n(\sigma(\lambda_1(P'_n)) - \sigma(\beta')) +b \sqrt n(\sigma(\lambda_n(P'_n)) + \sigma(\beta'))
    \]
    converges in distribution to a Gaussian random variable with variance
    \begin{align*}
        &\nabla\tilde\sigma(\beta', -\beta')^\top\begin{pmatrix}\frac{\alpha'^2}{4} + \frac{3\beta'^2}{2}& \frac{\alpha'^2}{4} + \frac{\beta'^2}{2}\\ \frac{\alpha'^2}{4} + \frac{\beta'^2}{2} & \frac{\alpha'^2}{4} + \frac{3\beta'^2}{2}\end{pmatrix}\nabla\tilde\sigma(\beta', -\beta')\\
        & = (a^2 +b^2)\bigg(\frac{\alpha'^2}{4} + \frac{3\beta'^2}{2}\bigg)\bigg(1-\frac{1}{\beta'^2}\bigg)^2 + 2ab\bigg(\frac{\alpha'^2}{4} + \frac{\beta'^2}{2}\bigg)\bigg(1-\frac{1}{\beta'^2}\bigg)^2.
    \end{align*}
    One can now conclude that
    \begin{align}\label{eq:delta_method} \nonumber
        &\sqrt{n}(\sigma(\lambda_1(P'_n)) - \sigma(\beta'), \sigma(\lambda_n(P'_n)) - \sigma(-\beta'))\\
        &\qquad\qquad\convd \cN_2 \left((0,0),\bigg(1-\frac{1}{\beta'^2}\bigg)^2\begin{pmatrix}\frac{\alpha'^2}{4} + \frac{3\beta'^2}{2}& \frac{\alpha'^2}{4} + \frac{\beta'^2}{2}\\ \frac{\alpha'^2}{4} + \frac{\beta'^2}{2} & \frac{\alpha'^2}{4} + \frac{3\beta'^2}{2}\end{pmatrix} \right).
    \end{align}
    Finally, using \eqref{eq:as_convergence}, \eqref{eq:delta_method} and the independence of $\xi$, $\tilde\xi$ and $(\lambda_1(P'_n), \lambda_n(P'_n))$, we can finally conclude that
    \begin{align*}
        \sqrt{n} (\lambda_1(X'_n) - \sigma(\beta'), \lambda_n(X'_n) - \sigma(-\beta')) 
    \end{align*}
    converges in distribution to a centered Gaussian distribution with the covariance matrix given by
    \[
    \begin{pmatrix}\big(1-\frac{1}{\beta'^2}\big)^2(\frac{\alpha'^2}{4} + \frac{3\beta'^2}{2}) + (\beta' -1)\eta(\beta')^2& \big(1-\frac{1}{\beta'^2}\big)^2(\frac{\alpha'^2}{4} + \frac{\beta'^2}{2})\\ \big(1-\frac{1}{\beta'^2}\big)^2(\frac{\alpha'^2}{4} + \frac{\beta'^2}{2}) & \big(1-\frac{1}{\beta'^2}\big)^2(\frac{\alpha'^2}{4} + \frac{3\beta'^2}{2})+(\beta' -1)\eta(\beta')^2\end{pmatrix}.
    \]
    We may simplify the covariance matrix to finish the proof of \eqref{eq:edge_fluc_r_fixed}.\\
    The proof of \eqref{eq:P_n_fluctuation_1<<r<<n} is much simpler. It follows from the fact that when $\sqrt n \ll r\ll n$
    \begin{equation}\label{eq:closeness_X_n_P_n}
        \frac{1}{\sqrt r}\bigg\|\frac{X_n}{\sqrt n} - \frac{1}{\sqrt n}P_n\bigg\|_\op = \frac{\theta}{\sqrt r}\| Z_n\|_\op =\sqrt{\frac{ n}{r(r-1)}} \bigg\|\frac{1}{\sqrt n}Z_n \bigg\|_\op = o(1) 
    \end{equation}
    almost surely.

    \item[(iii)] We now prove \eqref{eq:edge_fluc_log n<<r<<n}. We will be only proving the fluctuation for $\lambda_1(M_n)$. Then the joint fluctuation as in \eqref{eq:edge_fluc_log n<<r<<n} follows with a similar argument. To that end, we shall use Theorem~1 of \cite{fan2020edge}. We start with the equation
    \[
        \sqrt{n(r-1)}X_n = \sqrt{r-1}P_n + Z_n.
    \]
    Condition 1 of the aforementioned theorem from \cite{fan2020edge} is satisfied since
    \[
        \upsilon_n := \|\bbE[Z_n^2]\|_\op^{1/2} = \sqrt{n+1} \to \infty.
    \]
    Notice that 
    \begin{align*}
        &\bigg|\frac{\lambda_1(\sqrt{r-1}P_n)}{\lambda_n(\sqrt{r-1}P_n)} - 1\bigg| \to 2,\\
        &\min\{|\lambda_1(\sqrt{r-1}P_n)|, |\lambda_n(\sqrt{r-1}P_n)|\} \asymp \sqrt{n(r-2)} \gg \upsilon_n \log{n},\\
        &\max_{i,j} \Var(Z_{ij}) = 2 \le \frac{2\upsilon_n^2}{n}.
    \end{align*}
    Investigation of the proof of the theorem reveals that this verifies a variant of condition 2. Lastly observe that for any $\bu \in \bbS^{n-1}$, $\bu^\top Z_n  \bu \sim \cN(0,2)$ and hence
    \[
        \Var(\bu^\top Z_n \bu) = 2 \gg \frac{1}{r} \asymp \frac{\upsilon^2}{\min\{|\lambda_1(\sqrt{r-1}P_n)|, |\lambda_n(\sqrt{r-1}P_n)|\}^2}.
    \]
    Now Theorem~1 coupled with Lemma~3 of \cite{fan2020edge} implies that
    \[
        \frac{\lambda_1(\sqrt{n(r-1)}X_n) - t}{\sqrt{2}} \convd \cN(0,1)
    \]
    for some $t$ which satisfies that 
    \[
        \frac{t}{\min\{|\lambda_1(\sqrt{r-1}P_n)|, |\lambda_n(\sqrt{r-1}P_n)|\}} \to 1,
    \]
    that is $\frac{t}{\lambda_1(\sqrt{r-1}P_n)} \to 1$. Hence we get that
    \[
        \sqrt{n(r-1)}\bigg(\lambda_1(X_n)-\frac{1}{\sqrt{n}}\lambda_1(P_n)\bigg) \convd \cN(0,2).
    \]
    In particular the last term is $O_\bbP(1)$. Thus, 
    \[
        \sqrt{\frac{n}{r}}\bigg(\lambda_1(X_n)-\frac{1}{\sqrt{n}}\lambda_1(P_n)\bigg) = o_\bbP(1).
    \]
    Now from \eqref{eq:P_n_fluctuation_1<<r<<n}, we have that
    \[
        \sqrt{\frac{n}{r}}\bigg(\lambda_1(X_n) - \beta\bigg) \convd \cN(0, 1/4).
    \]
    Adding the last two terms we finally have that
    \[
        \sqrt{\frac{n}{r}}\bigg(\frac{\lambda_1(X_n)}{\sqrt{n}} - \beta\bigg) \convd \cN(0, 1/4).
    \]
    This completes the proof.
\end{proof}

\subsection{Proof of Theorem~\ref{thm:overlap-with-w}}
Since $M_n \overset{d}{=} X_n$, it is enough to prove the stated results for the edge eigenvectors of $X_n$. We shall in fact get almost sure results for $X_n$, which, as we only have a distributional identity, gets translated into in-probability results.
\begin{proof}[Proof of Theorem~\ref{thm:overlap-with-w}]
    As discussed, we shall prove the stated results for the edge eigenvectors of $X_n$. Let $P_n = \lambda_1 \bu_1 \bu_1^\top + \lambda_n \bu_n \bu_n^\top$. We know from the proof of Theorem~\ref{thm:edge-limits} that $\frac{\lambda_1}{\sqrt{n}} \convas \beta$ and $\frac{\lambda_n}{\sqrt{n}} \convas -\beta$. Let $\tilde{P}_n = \sqrt{n} \beta(\bu_1\bu_1^\top - \bu_n \bu_n^\top)$. Set
    \[
        \tilde{X}_n = \tilde{P}_n + \theta Z_n.
    \]
    Let $\tilde{\bv}_1$ and $\tilde{\bv}_n$ denote respectively the largest and the smallest eigenvectors of $\tilde{X}_n$. Then by Theorem~2.2 of \cite{benaych2011eigenvalues}, we have
    \[
        |\bu_1^\top \tilde{\bv}_1|^2 \convas 1 - \frac{\theta^2}{\beta^2} \qquad \text{and} \qquad |\bu_1^\top \tilde{\bv}_n|^2 \convas 0,
    \]
    and
    \[
        |\bu_n^\top \tilde{\bv}_n|^2 \convas 1 - \frac{\theta^2}{\beta^2} \qquad \text{and} \qquad |\bu_n^\top \tilde{\bv}_1|^2 \convas 0.
    \]
    Let $\bv_1$ and $\bv_n$ denote respectively the largest and the smallest eigenvectors of $X_n$. Then by the Davis-Kahan theorem (see, e.g., Theorem~2 of \cite{yu2015useful}), we have that for some absolute constant $C > 0$,
    \[
        |\bv_1^\top \tilde{\bv}_1| \ge 1 - \frac{C \|P_n - \tilde{P}_n\|_{\op}}{\sqrt{n}\beta} \convas 1.
    \]
    By a similar argument,
    \[
        |\bv_n^\top \tilde{\bv}_n| \convas 1.
    \]
    It follows that
    \[
        |\bu_1^\top \bv_1|^2 \convas 1 - \frac{\theta^2}{\beta^2} \qquad \text{and} \qquad |\bu_1^\top \bv_n|^2 \convas 0.
    \]
    Similarly,
    \[
        |\bu_n^\top \bv_n|^2 \convas 1 - \frac{\theta^2}{\beta^2}\qquad \text{and} \qquad |\bu_n^\top \bv_1|^2 \convas 0.
    \]
    Now by Lemma~\ref{lem:eigen-structure}-(iii), we have
    \[
        \bw^\top \bu_1 = \frac{\lambda_1}{\sqrt{\lambda_1^2 + \beta^2 (\|\bV\|^2 - (\bw^\top\bV)^2)}} \convas \frac{1}{\sqrt{2}},
    \]
    and similarly,
    \[
        \bw^\top \bu_n \convas -\frac{1}{\sqrt{2}}.
    \]
    Suppose $\eta_i = \sign(\bv_i^\top \bu_i)$, $i = 1, n$. Let $\bw_1 = \frac{\eta_1\bv_1 - \eta_n \bv_n}{\sqrt{2}}$ and $\bw_2 = \frac{\eta_1\bv_1 + \eta_n \bv_n}{\sqrt{2}}$. We see that
    \[
        \bw_1^\top \frac{\bu_1 - \bu_n}{\sqrt{2}} = \frac{1}{2} [|\bu_1^\top \bv_1| + |\bu_n^\top \bv_n| - \eta_1 \bv_1^\top \bu_n - \eta_n \bv_n^\top \bu_1] \convas \sqrt{1 - \frac{\theta^2}{\beta^2}}
    \]
    and
    \[
        \bw_2^\top \frac{\bu_1 - \bu_n}{\sqrt{2}} = \frac{1}{2} [|\bu_1^\top \bv_1| - |\bu_n^\top \bv_n| - \eta_1 \bv_1^\top \bu_n + \eta_n \bv_n^\top \bu_1] \convas 0.
    \]
    Since
    \[
        \bw^\top \frac{\bu_1 - \bu_n}{\sqrt{2}} \convas 1,
    \]
    it follows that
    \[
        \bw^\top \bw_1 \convas \sqrt{1 - \frac{\theta^2}{\beta^2}} \qquad \text{and} \qquad \bw^\top \bw_2 \convas 0.
    \]
    In other words,
    \[
        \max\bigg\{\bigg|\bw^\top \frac{\bv_1 + \bv_n}{\sqrt{2}}\bigg|^2, \bigg|\bw^\top \frac{\bv_1 - \bv_n}{\sqrt{2}}\bigg|^2\bigg\} \convas 1 - \frac{\theta^2}{\beta^2}
    \]
    and
    \[
        \min\bigg\{\bigg|\bw^\top \frac{\bv_1 + \bv_n}{\sqrt{2}}\bigg|^2, \bigg|\bw^\top \frac{\bv_1 - \bv_n}{\sqrt{2}}\bigg|^2\bigg\} \convas 0.
    \]
    The rest of the proof now follows since $M_n \overset{d}{=} X_n$.

    When $r \to \infty$, we have by the Davis-Kahan Theorem that
    \[
        |\bu_1^\top \bv_1| \ge 1 - \frac{C \theta \|Z_n\|_{\op}}{\lambda_1} \convas 1.
    \]
    Similarly,
    \[
        |\bu_n^\top \bv_n| \ge 1 - \frac{C \theta \|Z_n\|_{\op}}{|\lambda_n|} \convas 1.
    \]
    Here we have used that $\|Z_n\|_{\op} = O_P(\sqrt{n})$ and when $r \ll n$,
    \[
        \frac{\lambda_1}{\sqrt{n}} \convas 1.
    \]
    When $\frac{r}{n} \to c \in (0, \infty)$, we have
    \[
        \bigg(\frac{\lambda_1}{\sqrt{n}}, \frac{\lambda_n}{\sqrt{n}}\bigg) \convd \bigg(\frac{\sqrt{c}\zeta + \sqrt{c \zeta^2 + 4}}{2}, \frac{\sqrt{c}\zeta - \sqrt{c \zeta^2 + 4}}{2}\bigg) = \bigg(\xi, - \frac{1}{\xi}\bigg),
    \]
    where $\xi := \frac{\sqrt{c}\zeta + \sqrt{c \zeta^2 + 4}}{2} \ge 1$.
    Hence
    \begin{align*}
        (\bw^\top \bu_1, \bw^\top \bu_n) &= \bigg(\frac{\lambda_1}{\sqrt{\lambda_1^2 + \beta^2 (\|\bV\|^2 - (\bw^\top\bV)^2)}}, \frac{\lambda_n}{\sqrt{\lambda_n^2 + \beta^2 (\|\bV\|^2 - (\bw^\top\bV)^2)}}\bigg) \\
        &\convd \bigg(\frac{\xi}{\sqrt{ \xi^2 + 1}}, \frac{-1}{\sqrt{\xi^2 + 1}}\bigg).
    \end{align*}
It follows that
\begin{align*}
    \bigg(\bw^\top \frac{\bu_1 - \bu_n}{\sqrt{2}}, \bw^\top \frac{\bu_1 + \bu_n}{\sqrt{2}}\bigg) \convd \frac{1}{\sqrt{2}}\bigg(\frac{\xi + 1}{\sqrt{\xi^2 + 1}}, \frac{\xi - 1}{\sqrt{\xi^2 + 1}}\bigg).
\end{align*}
Hence
\begin{align*}
    \bigg(\max\bigg\{\bigg|\bw^\top \frac{\bv_1 - \bv_n}{\sqrt{2}}\bigg|, \bigg|\bw^\top \frac{\bv_1 + \bv_n}{\sqrt{2}}\bigg|\bigg\}, & \min\bigg\{\bigg|\bw^\top \frac{\bv_1 - \bv_n}{\sqrt{2}}\bigg|, \bigg|\bw^\top \frac{\bv_1 + \bv_n}{\sqrt{2}}\bigg|\bigg\}\bigg)\\
    &\convd \frac{1}{\sqrt{2}}\bigg(\frac{\xi + 1}{\sqrt{\xi^2 + 1}}, \frac{\xi - 1}{\sqrt{\xi^2 + 1}}\bigg).
\end{align*}
Finally, when $r / n \to \infty$, we have that
\[
    \bigg(\frac{\lambda_1}{\sqrt{r}}, \frac{\lambda_n}{\sqrt{r}}\bigg) \convd \big(\zeta_+, - \zeta_-\big).
\]
Hence
\begin{align*}
    (\bw^\top \bu_1, \bw^\top \bu_n) &= \bigg(\frac{\lambda_1}{\sqrt{\lambda_1^2 + \beta^2 (\|\bV\|^2 - (\bw^\top\bV)^2)}}, \frac{\lambda_n}{\sqrt{\lambda_n^2 + \beta^2 (\|\bV\|^2 - (\bw^\top\bV)^2)}}\bigg) \\
    &\convd (\ind_{\{\zeta > 0\}}, -\ind_{\{\zeta < 0\}}) \\
    &\overset{d}{=} (B, -(1 - B)),
\end{align*}
where $B \sim \Ber(1/2)$. It follows that
\begin{align*}
    \bigg(\bw^\top \frac{\bu_1 - \bu_n}{\sqrt{2}}, \bw^\top \frac{\bu_1 + \bu_n}{\sqrt{2}}\bigg) \convd \frac{1}{\sqrt{2}}\big(1, 2B - 1\big).
\end{align*}
Hence
\begin{align*}
    \bigg(\max\bigg\{\bigg|\bw^\top \frac{\bv_1 - \bv_n}{\sqrt{2}}\bigg|, \bigg|\bw^\top \frac{\bv_1 + \bv_n}{\sqrt{2}}\bigg|\bigg\}, & \min\bigg\{\bigg|\bw^\top \frac{\bv_1 - \bv_n}{\sqrt{2}}\bigg|, \bigg|\bw^\top \frac{\bv_1 + \bv_n}{\sqrt{2}}\bigg|\bigg\}\bigg)\\
    &\convd \bigg(\frac{1}{\sqrt{2}}, \frac{1}{\sqrt{2}}\bigg).
\end{align*}
This completes the proof.
\end{proof}

\subsection{Proof of Theorem~\ref{thm:spectral_measure}}
The proof uses the isotropic local semicircle law. We state a version of this here. Let us introduce the notation 
    \begin{align}\label{eq:inner_prod}
        &\langle \bx, \by\rangle_A := \bx^\top (A- zI)^{-1} \by,
        &\langle \bx, \by \rangle := \bx^\top \by,
    \end{align}
    for $\bx, \by \in \bbR^n$, $A \in \mathbb M_n(\bbR)$, $z \in \bbC$, whenever the resolvent $(A - zI)^{-1}$ exists.
    
    We now define the concept of stochastic domination.
    \begin{definition}[Stochastic Domination]
    Consider two sequences of families of nonnegative random variables 
    \begin{align*}
        &\xi= (\xi_n(u): u\in U_n)_{n=1}^\infty, &\zeta= (\zeta_n(u): u\in U_n)_{n=1}^\infty,
    \end{align*}
    where $(U_n)_{n=1}^\infty$ denotes a sequence of possibly $n$-dependent parameter sets. We say that $\xi \prec \zeta$, if for any $\epsilon, \, D >0$, there exists $n_0(\epsilon, \, D) \in \bbN$ such that 
    \[
        \sup_{u \in U_n} \bbP(\xi_n(u) > n^\epsilon \zeta_n(u)) \leq n^{-D}
    \]
    for all $n \geq n_0(\epsilon, D)$. Given two (not necessarily nonnegative) families $\xi, \zeta$  and a nonnegative family $\eta$, we say $\xi = \zeta + O_\prec(\eta)$ if $|\xi - \zeta| \prec \eta$.
    \end{definition}
    Given four nonnegative families $\xi_1,\, \xi_2,\, \zeta_1,\,\zeta_2$, if $\xi_1 \prec \zeta_1$ and $\xi_2 \prec \zeta_2$, then $\xi_1 + \xi_2 \prec \zeta_1+ \zeta_2$ and $\xi_1 \xi_2 \prec \zeta_1 \zeta_2$.

    Define 
    \begin{align}
        &M'_n = \frac{1}{\theta \sqrt n} M_n,
        &Z'_n = \frac{1}{\sqrt n}Z_n\\
        &\alpha' =\alpha/\theta, 
        &\beta' = \beta/\theta. 
    \end{align}
    Then, we have the following relationship between $M'_n$ and $Z'_n$:
    \[
        M'_n = \frac{\alpha' U}{\sqrt n} \bw \bw^{\top} +\frac{\beta'}{\sqrt n} (\bw \bV ^\top + \bV \bw^\top) + Z'_n,
    \]
    that is $M'_n$ is a rank-2 perturbation of $Z'_n$.
    
   We shall need a weak version of the isotropic local semicircle law which we now state below.
    
    \begin{lemma}[Isotropic local semicircle law]\label{lem:iso_local_semicircle_law}
        For a fixed $z \in \bbC^+$, 
        \begin{align}\label{eq:iso_local_semicircle}
        \langle \bx, \by \rangle_{Z'_n} = \langle \bx, \by \rangle s_{\sc} (z) + O_\prec\bigg(\frac{\|{\bx}\|\|\by\|}{\sqrt n}\bigg)
        \end{align}
        uniformly for $\bx,\, \by \in \bbR^n$.
    \end{lemma} 
    
\begin{lemma} \label{lem:local_law_step_2}
    If $\bx \in \bbS^{n-1}$ and $\bx \perp \bw$, then
    \begin{align}
        \la \bw, \bw \ra_{M'_n} &= \frac{ s_\sc(z)}{1- \beta'^2 s^2_\sc(z)} + O_\prec\bigg(\frac{1}{\sqrt n}\bigg), \label{eq:local_law_step_2_eq_1}\\
        \la \bx, \bw \ra_{M'_n} &=  O_\prec\bigg(\frac{1}{\sqrt n}\bigg), \label{eq:local_law_step_2_eq_2}\\
        \la \bx, \bx \ra_{M'_n} &= s_\sc(z) + O_\prec\bigg(\frac{1}{\sqrt n}\bigg). \label{eq:local_law_step_2_eq_3}
    \end{align}
\end{lemma}

\begin{lemma} \label{lem:Stieltjes_inversion}
     Define the measure $\mu$ as
     \[
        \mu = \rho^2\nu +(1-\rho^2) \mu_\sc+ \frac{\rho^2}{2}\bigg(1 - \frac{1}{\beta'^2}\bigg)(\delta_{(\beta' + 1/\beta')} + \delta_{-(\beta' + 1/\beta')})
     \]
     where $d\nu (x) = \frac{(1+\beta'^2)}{(1+\beta'^2)^2 - \beta'^2 x^2}. f_\sc(x) dx$. Then, $s(z)$, as defined in equation~\eqref{eq:limiting_Stieltjes_transform}, is the Stieltjes transform of $\mu$.
\end{lemma}

Throughout the proof we will use \eqref{eq:iso_local_semicircle} repeatedly without explicitly mentioning.
\begin{proof}[Proof of Theorem~\ref{thm:spectral_measure}]
    Given $A\in \mathbb M_n(\bbR)$, suppose its spectral decomposition is given by $A = \sum_{i=1}^n \lambda_i \bu_i \bu_i^\top$. With these notations the Stieltjes transform of $\mu_{M',\bx}$ can be expressed as  
    \[
        S_{M'_n, \bx}(z) = \sum_{i = 1}^n\frac{|\bx^\top \bu_i|^2}{\lambda_i - z} = \bx^\top \bigg(\sum_{i = 1}^n\frac{\bu_i \bu_i^\top}{\lambda_i - z}\bigg) \bx = \bx^\top (M'_n - zI)^{-1} \bx = \langle \bx, \bx \rangle _{M'_n}.
    \]
    So, it is enough to find the limit of $\langle \bx, \bx \rangle_{M'_n}$ for each fixed $z \in \bbC^+$. Decompose $\bx$ as 
    \[
        \bx = \rho_n \bw + \sqrt{1-\rho_n^2} \by
    \]
    for some $\by \in \bbS^{n-1}$ and $\by \perp \bw$. Then, using lemma~\ref{lem:local_law_step_2} we get
    \begin{align*}
        \la \bx, \bx \ra_{M'_n} &= \rho_n^2 \la \bw, \bw\ra_{M'_n} + 2\rho_n\sqrt{1-\rho_n^2} \Re(\la \bx, \bw \ra_{M'_n}) + (1- \rho_n^2) \la \by, \by\ra_{M'_n}\\
        &= \rho_n^2 \frac{ s_\sc(z)}{1- \beta'^2 s^2_\sc(z)} + (1-\rho_n^2) s_\sc(z) + O_\prec\bigg(\frac{1}{\sqrt n}\bigg).
    \end{align*}
    This shows that
    \begin{equation}\label{eq:limiting_Stieltjes_transform}
        s_{M'_n, \bx}(z) \convas s(z) := \rho^2\frac{s_\sc(z)}{1 - \beta'^2 s^2_\sc(z)} + (1-\rho^2)s_\sc(z).
    \end{equation}
    here $s$ is the Stieltjes transform of the limiting measure. The proof is finished using lemma~\ref{lem:Stieltjes_inversion}.
\end{proof}

\subsection{Proof of Theorem~\ref{thm:lss}}
The following lemma shows concentration of the nonzero eigenvalues of $P_n':= \sqrt{\frac{r-1}{n}} P_n$ which is crucial for the proof of Theorem~\ref{thm:lss}. 
\begin{lemma}\label{lem:concen_lambda_1_P_n} For $r \ge 3$ fixed, we have
    \begin{align}
        \lambda_1(P_n') &= \sqrt{r-2} + O_\prec\left(\frac{1}{\sqrt{n}}\right), \label{eq:concen_P_n_lambda_1} \\
        \lambda_n(P_n') &= - \sqrt{r-2} + O_\prec\left(\frac{1}{\sqrt{n}}\right). \label{eq:concen_P_n_lambda_n}
    \end{align}
\end{lemma}
\begin{proof}[Proof of Theorem~\ref{thm:lss}]
    As before we shall work with $X_n$ instead of $M_n$. More precisely, in what follows, $\lambda_1, \ldots, \lambda_n$ will denote the eigenvalues of 
    \[
        X_n' := \sqrt{\frac{r-1}{n}} X_n.
    \]
    The proof is based on the resolvent approach pioneered in \cite{bai2005convergence} for deriving fluctuations of LSS of Wigner matrices. Let $\epsilon >0$ be such that $\psi$ is real analytic on $(-\sqrt{r-2}-2 - \epsilon, \sqrt{r-2} +2 + \epsilon)$ with an absolutely convergent power series representation. One can then extend $\psi$ to a holomorphic function on the open disc $\{z \in \bbC: |z| < \sqrt{r-2} +2 + \epsilon \}$. To accommodate spikes, we use the interpolation approach of \cite{jung2023detection}.
    For $t \in [0, 1]$, let
    \[
        Y_n(t) = t P_n' + Z_n'.
    \]
    From Weyl's inequality, 
    \[
        t\lambda_n(P_n') + \lambda_n(Z_n') \le \lambda_n(Y_n(t)) \le \lambda_1(Y_n(t)) \le t\lambda_1(P_n') + \lambda_1(Z_n').
    \]
    This, coupled with Lemma~\ref{lem:concen_lambda_1_P_n}, yields that  
    Suppose $\bu_1$ and $\bu_n$ are the largest and smallest eigenvectors of $P_n'$. Then, by Lemma~\ref{lem:concen_lambda_1_P_n}, the following set
    \[
        A_n:= \left\{-\sqrt{r-2} - 2 - \frac{\epsilon}{2} < \lambda_n(Y_n(t)) \le \lambda_1(Y_n(t)) < \sqrt{r-2} + 2 + \frac{\epsilon}{2}: t\in [0,1] \right\}
    \]
    is a high probability set, i.e., $\bbP(A_n) \to 1$.
    
    Define the anti-clockwise oriented closed rectangle $\Gamma$ with vertices $(\pm (\sqrt{r-2} + 2 + \frac{3}{4} \epsilon ), \pm \delta)$, where $\delta>0$ is chosen depending on $\epsilon$, so that $\Gamma$ is inside the disc $\{z \in \bbC: |z| < \sqrt{r-2} +2 + \epsilon \}$. Also denote by $\Lambda(t)$, the set of eigenvalues of $Y_n(t),\,\, t \in [0,1]$, i.e.,
    \[
        \Lambda(t):= \{\lambda_i(Y_n(t)): i\in [n]\},
    \]
    and set
    \[
        \Lambda := \cup_{t \in [0,1]} \Lambda(t).
    \]
    Note that on the set $A_n$,  $d_\bbC (\Gamma, \Lambda) > \frac{\epsilon}{4}$, where for $A, B \subseteq \bbC$, $d_\bbC(A, B) := \inf \{|a-b|: a \in A, \,\, b \in B\}$. From this onward, we only work on the set $A_n$. Let $D$ be the open box with vertices $(\pm (\sqrt{r-2} + 2 + \frac{7}{8} \epsilon ), \pm \delta')$ and 
    \begin{align*}
        \tilde{D}&:= \bigg\{z \in D: \left|\Re(z) - (\sqrt{r-2} + 2 + \frac{3}{4} \epsilon) \right| < \frac{\epsilon}{8},\,\, \text{or},\,\, \left|\Re(z) + (\sqrt{r-2} + 2 + \frac{3}{4} \epsilon) \right| < \frac{\epsilon}{8} ,\\
        &\text{or}, \delta'' < \Im(z) < \delta' \, \text{or} \,-\delta' < \Im(z) < - \delta''\bigg\},
    \end{align*}
    where $\delta', \delta''$ are chosen ensuring that $\delta'' < \delta < \delta'$ and $\tilde{D} \subseteq \{z \in \bbC: |z| < \sqrt{r-2} +2 + \epsilon \}$ (see Figure~\ref{fig:lss-domains}).
    \begin{figure}[!t]
        \centering
        \begin{tikzpicture}[scale=1.3]
    \def\xmax{3.5}       
    \def\xmin{2.5}       
    \def\ydprime{2}      
    \def\ydpprime{1}     
    
    \draw[->] (-\xmax-1, 0) -- (\xmax+1, 0) node[right] {$\Re(z)$};
    \draw[->] (0, -\ydprime-1) -- (0, \ydprime+1) node[above] {$\Im(z)$};
    
    \draw[dashed, thick, blue, fill=blue!10, fill opacity=0.5] 
        (-\xmax, -\ydprime) rectangle (\xmax, \ydprime);

    \draw[dashed, thick, red, fill=red!10, fill opacity=0.5, even odd rule] 
        (-\xmax, -\ydprime) rectangle (\xmax, \ydprime)
        (-\xmin, -\ydpprime) rectangle (\xmin, \ydpprime);
        
    \node[blue] at (0.25, 0.25) {$D$};
    
    \node[red] at (\xmin+0.5, \ydpprime+0.5) {$\tilde{D}$};
    \node[red] at (-\xmin-0.5, \ydpprime+0.5) {$\tilde{D}$};
    \node[red] at (\xmin+0.5, -\ydpprime-0.5) {$\tilde{D}$};
    \node[red] at (-\xmin-0.5, -\ydpprime-0.5) {$\tilde{D}$};
    
    \draw (\xmax, 0.1) -- (\xmax, -0.1) 
        node[below right, scale=0.8] {$\sqrt{r-2} + 2 + \frac{7}{8}\epsilon$};
    \draw (\xmin, 0.1) -- (\xmin, -0.1) 
        node[below left, scale=0.8] {$\sqrt{r-2} + 2 + \frac{5}{8}\epsilon$};
    \draw (-\xmax, 0.1) -- (-\xmax, -0.1) 
        node[below left, scale = 0.8] {$-\left(\sqrt{r-2} + 2 + \frac{7}{8}\epsilon\right)$};
    \draw (-\xmin, 0.1) -- (-\xmin, -0.1) 
        node[below right, scale=0.8] {$-\left(\sqrt{r-2} + 2 + \frac{5}{8}\epsilon\right)$};
        
    \draw (0.1, \ydprime) -- (-0.1, \ydprime) node[above left] {$\delta'$};
    \draw (0.1, \ydpprime) -- (-0.1, \ydpprime) node[left] {$\delta''$};
    \draw (0.1, -\ydpprime) -- (-0.1, -\ydpprime) node[left] {$-\delta''$};
    \draw (0.1, -\ydprime) -- (-0.1, -\ydprime) node[below left] {$-\delta'$};
\end{tikzpicture}
        \caption{The sets $D$ and $\tilde{D}$.}
        \label{fig:lss-domains}
    \end{figure}
    Define for $z \in \bbC \backslash [- \sqrt{r-2} - 2 - \frac{\epsilon}{2},  \sqrt{r-2} + 2 + \frac{\epsilon}{2}]$,
    \begin{align*}
        R_n(t, z) &:= (Y_n(t) - zI)^{-1},\\
        s_n(t, z) &:= \frac{1}{n}\Tr R_n(t, z).
    \end{align*}
    Note that for each fixed $t$, $s_n(t,\cdot)$ is holomorphic on $\bbC \backslash [- \sqrt{r-2} - 2 - \frac{\epsilon}{2},  \sqrt{r-2} + 2 + \frac{\epsilon}{2}]$, in particular, on $\tilde{D}$, and meromorphic on $D$ with poles at the points of $\Lambda(t)$. Let 
    \[
        s_{\sc}(z) = \frac{-z + \sqrt{z^2 - 4}}{2}, 
    \]
    denote the Stieltjes transform of $\nu_{\sc, 1}$ where we take the convention that $\sign ( \Im(s_{\sc}(z))) = \sign (z)$. 
    Now,
    \begin{align*}
        &\sum_{i = 1}^n \psi(\lambda_i) - n \int \psi(x) \, d\nu_{\sc, 1}(x) \\
        &= -\frac{n}{2\pi \iota} \oint_{\Gamma} \psi(z) \big[s_n(1, z) - s_{\sc}(z)\big] \, dz \\
        &= -\frac{1}{2\pi \iota} \oint_{\Gamma} \psi(z) \big[\Tr R_n(1, z) - n s_{\sc}(z)\big] \, dz \\
        &= \underbrace{-\frac{1}{2\pi \iota} \oint_{\Gamma} \psi(z) \big[\Tr R_n(1, z) - \Tr R_n(0, z)\big] \, dz}_{I_1} - \underbrace{\frac{1}{2\pi \iota} \oint_{\Gamma} \psi(z) \big[\Tr R_n(0, z) - n s_{\sc}(z)\big] \, dz}_{I_2}.
    \end{align*}
    Here $I_2$ captures the fluctuations of the eigenvalues of a GOE matrix. Using Theorem~1.1 from \cite{bai2005convergence}, we get that
    \[
        I_2 \xrightarrow{d} \cN ( m_{\psi, 2}, \sigma_{\psi, 2}^2),
    \]
    where 
    \begin{align}
        m_{\psi, 2} &= \frac{1}{4} (\psi(2) + \psi(-2)) - \frac{1}{2} \tau_0(\psi),  \\
        \sigma_{\psi, 2}^2 &= 2 \sum_{\ell=1}^\infty \ell \tau_\ell(\psi)^2.
    \end{align}
    We shall now show that 
    \begin{equation}\label{eq:lss_I_1}
        I_1 \xrightarrow{\bbP} 2 \sum_{\ell =1}^\infty (r-2)^\ell \tau_{2\ell}(\psi).
    \end{equation}
    To this end notice that
    \begin{align*}
        \Tr R_n(1,z) - \Tr R_n(0,z) &= \int_0^1\frac{\partial}{\partial t} \Tr R_n(t, z)dt.
    \end{align*}
    So, we get that
    \begin{align*}
        I_1 &= -\frac{1}{2\pi \iota} \oint_{\Gamma} \psi(z) \big[\Tr R_n(1, z) - \Tr R_n(0, z)\big] \, dz \\
        &= -\frac{1}{2\pi \iota} \oint_{\Gamma} \int_0^1 \psi(z) \frac{\partial}{\partial t} \Tr R_n(t, z) \, dz \, dt \\
        &=  \lambda_1(P_n') \frac{1}{2\pi \iota}  \oint_{\Gamma} \int_0^1 \psi(z) \frac{\partial}{\partial z} \bu_1^\top R_n(t, z)\bu_1 \, dt \, dz  \\
        &+ \lambda_n(P_n') \frac{1}{2\pi \iota}  \oint_{\Gamma} \int_0^1 \psi(z) \frac{\partial}{\partial z} \bu_n^\top R_n(t, z)\bu_n \, dt \, dz.
    \end{align*}
    Let $\bu_1$ and $\bu_n$ are the largest and smallest eigenvectors of $P'_n$ respectively. Then
    \[
        R_n(0,z) - R_n(t,z) = t \lambda_1(P_n') R_n(0,z)\bu_1 \bu_1^\top R_n(t,z) +  t\lambda_n(P_n') R_n(0,z)\bu_n \bu_n^\top R_n(t,z).
    \]
    Now notice that
    \begin{align*}
        \bu_1^\top R_n(0,z) \bu_1 - \bu_1^\top R_n(t,z) \bu_1 &= t \lambda_1(P_n') \bu_1^\top R_n(0,z)\bu_1 \bu_1^\top R_n(t,z)\bu_1 \\
        &+ t\lambda_n(P_n') \bu_1^\top R_n(0,z)\bu_n \bu_n^\top R_n(t,z) \bu_1. 
    \end{align*}
    Let $\kappa \in (0, \frac{1}{2})$. Define 
    \[
        \Delta_\kappa := \{w \in \bbC: \Im(w)> \kappa\}.
    \]
    Now by Theorem~2.3 of \cite{knowles2013isotropic} and the fact $\|R_n(\cdot, z)\|_\op \leq 1/\Im(z)$, we obtain that for $z \in \Gamma \cap \Delta_\kappa$,
    \[
        \bu_1^\top R_n(t,z) \bu_1 = s_{\nu_\sc}(z) - t \sqrt{r-2} s_{\nu_\sc}(z) \bu_1^\top R_n(t,z)\bu_1 + O_{\prec} \left(\sqrt{\frac{\Im(s_\sc(z))}{nv}}\right),
    \]
    which, by virtue of Lemma~3.2 of \cite{knowles2013isotropic}, ultimately yields that
    \begin{equation}\label{eq:quadratic_form_lambda_1}
        \bu_1^\top R_n(t,z) \bu_1 =  \frac{ s_{\nu_\sc}(z)}{1 + t\sqrt{r-2} s_{\nu_\sc}(z)} + O_{\prec} \bigg(\frac{1}{\sqrt n}\bigg).
    \end{equation}
    Of course a similar argument provides that for $z \in \Gamma \cap \Delta_\kappa$,
    \begin{equation}\label{eq:quadratic_form_lambda_n}
        \bu_n^\top R_n(t,z) \bu_n =  \frac{ s_{\nu_\sc}(z)}{1 - t\sqrt{r-2} s_{\nu_\sc}(z)} + O_{\prec} \bigg(\frac{1}{\sqrt n}\bigg).
    \end{equation}
    By Theorem~5.3 of \cite{stein2003complex}, we can conclude that uniformly for $z \in \Gamma \cap \Delta_\kappa$, we have
    \begin{align*}
        \frac{\partial}{\partial z} \bu_1^\top R_n(t,z) \bu_1 &= \frac{s_{\nu_\sc}'(z)}{(1 + t \sqrt{r-2} s_{\nu_\sc}(z))^2} + O_{\prec} \bigg(\frac{1}{\sqrt n}\bigg), \\
        \frac{\partial}{\partial z} \bu_n^\top R_n(t,z) \bu_n &= \frac{s_{\nu_\sc}'(z)}{(1 - t \sqrt{r-2} s_{\nu_\sc}(z))^2} + O_{\prec} \bigg(\frac{1}{\sqrt n}\bigg).
    \end{align*}
    Moreover, for $z \in \Gamma$, 
    \[
        \left\| \frac{\partial}{\partial z} R_n(t,z) \right\|_\op = \| (R_n(t,z))^{-2} \|_\op =\max_{\lambda \in \Lambda(t)} |\lambda - z|^{-2} \le \left( \frac{4}{\epsilon} \right)^2 
    \]
    and hence
    \begin{align*}
        \left| \frac{\partial}{\partial z} \bu_1^\top R_n(t,z) \bu_1 \right| &\le \frac{16}{\epsilon^2}, \qquad \left| \frac{\partial}{\partial z} \bu_n^\top R_n(t,z) \bu_n \right| \le \frac{16}{\epsilon^2}.
    \end{align*}
    The last two equations coupled with Lemma~\ref{lem:concen_lambda_1_P_n}, we get that
    \begin{align*}
        I_1 
        &=  \sqrt{r-2} \frac{1}{2\pi \iota}\oint_{\Gamma} \int_0^1 \psi(z) \frac{\partial}{\partial z} \bu_1^\top R_n(t, z)\bu_1 \, dt \, dz \\
        &- \sqrt{r-2} \frac{1}{2\pi \iota}  \oint_{\Gamma} \int_0^1 \psi(z) \frac{\partial}{\partial z} \bu_n^\top R_n(t, z)\bu_n \, dt \, dz\\
        &= \frac{\sqrt{r-2}}{2\pi \iota}  \int_{\Gamma \cap \Delta_\kappa} \int_0^1 \psi(z) \left[ \frac{s'(z)}{(1 + t \sqrt{r-2} s(z))^2} - \frac{s'(z)}{(1 - t \sqrt{r-2} s(z))^2}\right] \, dt \, dz  \\
        &+ \frac{\sqrt{r-2}}{2\pi \iota} \int_{\Gamma \cap \Delta_\kappa^c} \int_0^1 \psi(z) \left( \frac{\partial}{\partial z} \bu_1^\top R_n(t,z) \bu_1 - \frac{\partial}{\partial z} \bu_n^\top R_n(t,z) \bu_n \right) \, dt \, dz + O_\prec \left( \frac{1}{\sqrt{n}} \right)\\
        &=  - \frac{4 (r-2)}{2 \pi \iota} \int_{\Gamma \cap \Delta_\kappa} \int_0^1 \psi(z) \left[ t \frac{s(z) s'(z)}{(1 - t^2 (r-2) (s(z))^2)^2} \right] \, dt \, dz  \\
        &+ O_\prec(\kappa) + O_\prec \left( \frac{1}{\sqrt{n}} \right).
    \end{align*}
    Taking limits first as $n \to \infty$ and then as $\kappa \to 0+$, we get
    \begin{align*}
        I_1 &\xrightarrow{\bbP} - \frac{4 (r-2)}{2 \pi \iota} \int_{\Gamma} \int_0^1 \psi(z) \left[ t \frac{s(z) s'(z)}{(1 - t^2 (r-2) (s(z))^2)^2} \right] \, dt \, dz \\
        &= - \frac{2 (r-2)}{2 \pi \iota} \int_{\Gamma} \psi(z) \left[ \frac{s(z) s'(z)}{1 - (r-2) (s(z))^2} \right] \, dz \\
        &= \frac{\sqrt{r-2}}{2 \pi \iota} \int_{\Gamma} \psi(z) \left[ \frac{ s'(z)}{1 + \sqrt{r-2} s(z)} - \frac{s'(z)}{1 - \sqrt{r-2} s(z)} \right] \, dz \\
        &= \frac{1}{2 \pi \iota} \int_{\Gamma} \psi(z) \frac{ \sqrt{r-2} s'(z)}{1 + \sqrt{r-2} s(z)} \, dz + \frac{1}{2 \pi \iota} \int_{\Gamma} \psi(z) \frac{ -\sqrt{r-2} s'(z)}{1 - \sqrt{r-2} s(z)} \, dz.
    \end{align*}
    Finally, following the proof of Lemma~4.4 in \cite{baik2017fluctuations}, we can deduce that
    \begin{align*}
        \frac{1}{2 \pi \iota} \int_{\Gamma} \psi(z) \frac{ \sqrt{r-2} s'(z)}{1 + \sqrt{r-2} s(z)} &\to \sum_{\ell =1}^\infty (r-2)^{\frac{\ell}{2}} \tau_{\ell}(\psi),\\
        \frac{1}{2 \pi \iota} \int_{\Gamma} \psi(z) \frac{ - \sqrt{r-2} s'(z)}{1 - \sqrt{r-2} s(z)} &\to \sum_{\ell =1}^\infty 
        (-1)^\ell (r-2)^{\frac{\ell}{2}} \tau_{\ell}(\psi).
    \end{align*}
    Hence we get \eqref{eq:lss_I_1} completing the proof of Theorem~\ref{thm:lss}.
\end{proof}

\subsection{The correlation structure of mixed contractions}
We shall first prove a result explaining the correlation structure of $\cG \cdot \bu \otimes \bv$ in terms of a low-rank perturbation of a scaled GOE matrix. First we recall the definition of Kronecker product of matrices (denoted, by an abuse of notation, by $\otimes$).
\begin{definition}[Kronecker Product of Matrices]
    Given $A \in \mathrm{Mat}_{k, l}(\bbR)$ and $B \in \mathrm{Mat}_{k',l'}(\bbR)$, the Kronecker product $A \otimes B$ of $A$ and $B$ is the $kk' \times ll'$ matrix defined as
    \[
        A \otimes B := 
        \begin{pmatrix}
            a_{11}B & a_{12}B & \cdots & a_{1l}B \\
            a_{21}B & a_{22}B & \cdots & a_{2l}B \\
            \vdots & \vdots & \ddots & \vdots \\
            a_{k1}B & a_{k2}B & \cdots & a_{kl}B
        \end{pmatrix}.
    \]
\end{definition}
In particular, notice that for two column vectors $\bu$ and $\bv$, $\bu \otimes \bv^\top = \bu \bv^\top$. It is also easy to see that $\| A \otimes B \|_{\mathrm{F}} = \| A\|_{\mathrm{F}} \| B\|_{\mathrm{F}}$.

The following lemma describes the correlation structure of the mixed contraction $M_n = \cG \cdot \bu \otimes \bv$.
\begin{lemma}[Correlation structure of $\GOTE(4, n)$]\label{lem:cov-structure-mixed-4}
    Suppose $i,j,k,l \in[n]$ are distinct indices. Then
    \begin{align} \label{eq:var-1-mixed}
        \Var(M_{ii}) &= \frac{1}{3}(1 + (\bu^\top \bv)^2) + \frac{4}{3} (\bu^\top \bv) u_i v_i + \frac{2}{3} (u_i^2 + v_i^2) + \frac{2}{3} u_i^2 v_i^2; \\ \label{eq:var-2-mixed}
        \Var(M_{ij}) &= \frac{1}{6}(1 + (\bu^\top \bv)^2) + \frac{1}{3} (\bu^\top \bv) (u_i v_i + u_j v_j) \\ \nonumber
        &\qquad\qquad + \frac{1}{6}(u_i^2 + u_j^2 + v_i^2 + v_j^2) + \frac{1}{6}(u_i v_j + u_j v_i)^2; \\ \label{eq:cov-1-mixed}
        \Cov(M_{ij}, M_{kl}) &= \frac{1}{6} (u_i v_j + u_j v_i) (u_k v_l + u_l v_k); \\ \label{eq:cov-2-mixed}
        \Cov(M_{ii}, M_{kl}) &= \frac{1}{3} u_i v_i (u_k v_l + u_l v_k); \\ \label{eq:cov-3-mixed}
        \Cov(M_{ii}, M_{kk}) &= \frac{2}{3} u_i v_i u_k v_k; \\ \label{eq:cov-4-mixed}
        \Cov(M_{ij}, M_{il}) &= \frac{1}{6} [u_j u_l + v_j v_l + (\bu^\top \bv) (u_j v_l + u_l v_j)] \\ \nonumber 
        &\qquad\qquad + \frac{1}{6} (u_i v_j + u_j v_i) (u_i v_l + u_l v_i); \\ \label{eq:cov-5-mixed}
        \Cov(M_{ii}, M_{il}) &= \frac{1}{3}(u_iu_l + v_iv_l) + \frac{1}{3}(\bu^\top\bv) (u_iv_l + u_lv_i) + \frac{1}{3}u_iv_i(u_iv_l + u_lv_i).
    \end{align}
\end{lemma}

\begin{lemma}\label{lem:cov-representation-mixed-4}
    Suppose $\bu, \bv \in \bbS^{n - 1}$. Let $U \sim \cN(0, 1)$, 
    \[
    (\bV_1, \bV_2) \sim \cN_{2n} \left(\begin{pmatrix} \bzero_n \\ \bzero_n \end{pmatrix}, \begin{pmatrix}
        1 & \langle \bu, \bv \rangle \\
        \langle \bu, \bv \rangle & 1
    \end{pmatrix} \otimes I_n\right)
    \]
    and $Z_n \sim \GOE(n)$, with them being mutually independent. Then
    \begin{align}\label{eq:mixed_low_rank}
        \cG \cdot \bu \otimes \bv &\overset{d}{=} \frac{1}{\sqrt{6}} U (\bu \bv^\top + \bv \bu^\top) \\ \nonumber
        &\qquad\qquad + \frac{1}{\sqrt{6}} (\bV_1 \bu^\top + \bu \bV_1^\top + \bV_2 \bv^\top + \bv \bV_2^\top) + \frac{1}{\sqrt{6}} \sqrt{1 + \langle \bu, \bv \rangle^2} Z_n.
    \end{align}
\end{lemma}

\subsection{Proof of Theorem~\ref{thm:mixed_bulk}}
\begin{proof}[Proof of Theorem~\ref{thm:mixed_bulk}]
    By the rank inequality, one obtains that the EESD of the right-hand side of \eqref{eq:mixed_low_rank} has the same limit as that of $\frac{1}{\sqrt{6}}\sqrt{1 + \langle \bu, \bv\rangle^2} Z_n$, which is $\mu_{\sc, \frac{1 + \rho^2}{6}}$. Therefore, the EESD of LHS also converges weakly to the same limit law. Now, by Lemma~\ref{lem:concentration}-(ii), the ESD is exponentially concentrated around the EESD in the bounded Lipschitz metric, whence the desired almost sure convergence follows.
\end{proof}

\subsection{Proof of Theorem~\ref{thm:total_variation}}
We shall now prove Theorem~\ref{thm:total_variation}. We need to set up some notations first.

For an $m \times n$ matrix $A$, we denote by $\vec(A)$, the $mn \times 1$ vector obtained by stacking the columns of $A$:
\[
    \vec(A) = (A_{11}, \ldots, A_{n, 1}, A_{12}, \ldots, A_{n2}, \ldots, A_{1n}, \ldots, A_{nn})^\top.
\]
For a symmetric $n \times n$ matrix $A$, $\vech(A)$ denotes the vector obtained by stacking the entries on and above the diagonal in a columnwise fashion:
\[
    \vech(A) = (A_{11}, A_{12}, A_{22}, \ldots, A_{1n}, \ldots, A_{nn})^\top.
\]
For a symmetric $n \times n$ matrix $A$, one may write $\vech(A) = L \, \vec(A)$, where $L$ is an $n(n + 1) / 2 \times n$ matrix called the \emph{elimination matrix} (see, e.g., \cite{magnus1980elimination}). It can be checked that $L$ is row-orthogonal, i.e. $LL^\top = I$.

With the above notations set, define
\[
    \Sigma_{\bu, \bv} := \Cov(\vech(\cG \cdot \bu \otimes \bv)),
\]
i.e. $\Sigma_{\bu, \bv}$ is the covariance matrix of the distinct random variables in $\cG \cdot \bu \otimes \bv$. It turns out that $\lambda_{\min}(\Sigma_{\bu, \bu}) = \Theta(1)$, uniformly in $\bu$.
\begin{lemma}\label{lem:cov-mat-eval}
For any $\bu \in \bbS^{n - 1}$,
\[
    \frac{1}{3} \le \lambda_{\min}(\Sigma_{\bu, \bu}) \le \lambda_{\min}(\Sigma_{\bu, \bu}) \le 4.
\]
\end{lemma}

Next we shall need an upper bound on $\|\Sigma_{\bu, \bv} - \Sigma_{\bu, \bv}\|_{\mathrm{F}}$. 
\begin{lemma}\label{lem:frob_norm}
    $\|\Sigma_{\bu, \bv} - \Sigma_{\bu, \bu}\|_{\mathrm{F}} \le \frac{5}{2} n \|\bu - \bv\|_2$.
\end{lemma}

\begin{remark}
    The upper bound in Lemma~\ref{lem:frob_norm} is tight in general. For $\bu = \be_1$, $\Sigma_{\bu, \bu}$ is diagonal with $\lambda_{\min}(\Sigma_{\bu, \bu}) = 1/3$. Let us take $\bv = \sqrt{1 - \gamma^2} \be_1 + \gamma \be_2$, where $\gamma \in (0, 1)$. Then $\|\bu - \bv\|_2 = \sqrt{2(1 - \sqrt{1 - \gamma^2})} = \Theta(\gamma)$. In this case,  we can exactly compute $\|\Sigma_{\bu, \bv} - \Sigma_{\bu, \bu}\|_{\mathrm{F}}$. Note that
    \[
        (\cG \cdot \bu \otimes \bv)_{ij} = \sum_{i_3, i_4} \cG_{ij i_3 i_4}\delta_{i_3, 1}(\sqrt{1 - \gamma^2} \delta_{i_4, 1} + \gamma \delta_{i_4, 2}) = \sqrt{1 - \gamma^2} \cG_{ij11} + \gamma \cG_{ij12}.
    \]
    It is then clear that
    \[
        \Cov((\cG \cdot \bu \otimes \bv)_{ij}, (\cG \cdot \bu \otimes \bv)_{kl}) = 0
    \]
    if $\{i, j\} \ne \{k, l\}$ and $i, j, k, l \notin \{1, 2\}$. It follows that
    \[
        \|\Sigma_{\bu, \bv} - \Sigma_{\bu, \bu}\|_{\mathrm{F}} = \Theta(\gamma n).
    \]
\end{remark}
We are now ready to prove Theorem~\ref{thm:total_variation}. A key ingredient in the proof of Theorem~\ref{thm:total_variation} is the following estimates on the total variation distance between two multivariate Gaussian distributions due to \cite{devroye2018total} and \cite{arbas2023polynomial}.
\begin{proposition}[Theorem~1.2 of \cite{devroye2018total}]\label{prop:totvar_guassians}
    \begin{align*}
        \frac{1}{100} \min \{1, \|\Sigma_1^{-1/2} &\Sigma_2 \Sigma_1^{-1/2} -  I\|_{\mathrm{F}}\} \\
        &\le d_{\mathrm{TV}}(\cN(\bzero, \Sigma_1), \cN(\bzero, \Sigma_2)) \le \frac{3}{2} \min \{1, \|\Sigma_1^{-1/2} \Sigma_2 \Sigma_1^{-1/2} -  I\|_{\mathrm{F}}\}.
    \end{align*}
\end{proposition}   
\begin{proof}[Proof of Theorem~\ref{thm:total_variation}]
By Proposition~\ref{prop:totvar_guassians}, we have
\begin{align*}
     &\frac{1}{100} \min\left\{1, \|\Sigma_{\bu, \bu}^{-1/2} \Sigma_{\bu, \bv} \Sigma_{\bu, \bu}^{-1/2} - I\|_{\mathrm{F}}\right\} \\
     &\hspace{10em}\le d_{\mathrm{TV}}(\bbP_{\bu, \bv}, \bbP_{\bu, \bu}) \le \frac{3}{2} \min\left\{1, \|\Sigma_{\bu, \bu}^{-1/2} \Sigma_{\bu, \bv} \Sigma_{\bu, \bu}^{-1/2} - I\|_{\mathrm{F}}\right\}.
\end{align*}
Now
\begin{align*}
    \|\Sigma_{\bu, \bu}^{-1/2} \Sigma_{\bu, \bv} \Sigma_{\bu, \bu}^{-1/2} - I\|_{\mathrm{F}} &= \|\Sigma_{\bu, \bu}^{-1/2} (\Sigma_{\bu, \bv} - \Sigma_{\bu, \bu}) \Sigma_{\bu, \bu}^{-1/2}\|_{\mathrm{F}} \\
    &\le \|\Sigma_{\bu, \bu}^{-1/2}\|_{\op}^2 \|\Sigma_{\bu, \bv} - \Sigma_{\bu, \bu}\|_{\mathrm{F}} \\
    &\le \frac{\|\Sigma_{\bu, \bv} - \Sigma_{\bu, \bu}\|_{\mathrm{F}}}{\lambda_{\min}(\Sigma_{\bu, \bu})}.
\end{align*}
Similarly, one also has a lower bound:
\[
    \|\Sigma_{\bu, \bu}^{-1/2} \Sigma_{\bu, \bv} \Sigma_{\bu, \bu}^{-1/2} - I\|_{\mathrm{F}} \ge \frac{\|\Sigma_{\bu, \bv} - \Sigma_{\bu, \bu}\|_{\mathrm{F}}}{\lambda_{\max}(\Sigma_{\bu, \bu})}.
\]
The desired result now follows by plugging in the estimates from Lemma~\ref{lem:cov-mat-eval} and Lemma~\ref{lem:frob_norm}.
\end{proof}
\begin{proof}[Proof of Proposition~\ref{prop:mixed_vs_pure_lsd}]
    First note that
    \begin{align*}
        \|\cG \cdot \bu \otimes \bv - \cG \cdot \bu \otimes \bu\|_{\mathrm{F}}^2 &= \sum_{i_1, i_2} \bigg(\sum_{i_3, i_4} \cG_{i_1i_2i_3i_4} u_{i_3} (v_{i_4} - u_{i_4})\bigg)^2. 
    \end{align*}
    Clearly, $\Var(\cG_{i_1, i_2, i_3, i_4}) \leq 4$. Hence
    \begin{align*}
        \bbE \bigg(\sum_{i_3, i_4} &\cG_{i_1i_2i_3i_4} u_{i_3} (v_{i_4} - u_{i_4})\bigg)^2 \\
        &= \sum_{i_3 < i_4} \Var(\cG_{i_1, i_2, i_3, i_4}) ((u_{i_3}(v_{i_4} - u_{i_4})+ u_{i_4} (v_{i_3} - u_{i_3}))^2 \\
        &\qquad\qquad + \sum_{i_3} \Var(\cG_{i_1, i_2, i_3, i_3}) u^2_{i_3}(v_{i_3} - u_{i_3})^2\\
        &\leq 2\sum_{i_3 < i_4} \Var(\cG_{i_1, i_2, i_3, i_4}) ((u^2_{i_3}(v_{i_4} - u_{i_4})^2+ u^2_{i_4} (v_{i_3} - u_{i_3})^2) \\
        &\qquad\qquad + \sum_{i_3} \Var(\cG_{i_1, i_2, i_3, i_3}) u^2_{i_3}(v_{i_3} - u_{i_3})^2\\
        &\leq 8 \sum_{i_3, i_4} u_{i_3} (v_{i_4} - u_{i_4})^2\\
        &= 8\|\bv -\bu\|_{\mathrm{F}}^2.
    \end{align*}
    It follows that $\bbE\|\cG \cdot \bu \otimes \bv - \cG \cdot \bu \otimes \bu\|_{\mathrm{F}}^2 \leq 8n^2\|\bv - \bu \|_{\mathrm{F}}^2$. By the Hoffman-Wielandt inequality,
    \begin{align*}
        d_{W_2}(\mubar_{n^{-1/2} \cG \cdot \bu \otimes \bv}, \mubar_{n^{-1/2} \cG \cdot \bu \otimes \bv})^2 &\le \frac{1}{n} \bbE\|n^{-1/2}\cG \cdot \bu \otimes \bv - n^{-1/2} \cG \cdot \bu \otimes \bu\|_{\mathrm{F}}^2 \\
        &\le 8\|\bu - \bv\|_2^2.
    \end{align*}
    Since $\bar{\mu}_{n^{-1/2}\cG \cdot \bu \otimes \bu} \convd \nu_{\sc, \frac{1}{3}}$, we also have $\bar{\mu}_{n^{-1/2}\cG \cdot \bu \otimes \bv} \convd \nu_{\sc, \frac{1}{3}}$ if $\|\bu - \bv\| = o(1)$. An application of Lemma~\ref{lem:concentration} completes the proof.
\end{proof}

\subsection{Proof of Theorem~\ref{thm:mixed_vs_pure_top_evalue}}
We will now prove Theorem~\ref{thm:mixed_vs_pure_top_evalue}. The proof uses a covering argument together with Gaussian comparison inequalities. We first need the following lemma, which is proved using an interpolation argument.

\begin{lemma}\label{lem:interpolation}
Let $\bs, \bt \in \bbS^{n - 1}$. There is an universal constant $C > 0$ such that
\[
    \big|\Cov\big(\bs^\top (\cG \cdot \bu \otimes \bv) \bs, \bt^\top (\cG \cdot \bu \otimes \bv) \bt\big) - \Cov\big(\bs^\top (\cG \cdot \bu \otimes \bu) \bs, \bt^\top (\cG \cdot \bu \otimes \bu) \bt\big)\big| \le C \|\bu - \bv\|_2.
\]
\end{lemma}

\begin{proof}[Proof of Theorem~\ref{thm:mixed_vs_pure_top_evalue}]
    Let $\bu_1, \ldots, \bu_N$ be an $\varepsilon$-net for $\bbS^{n - 1}$. We first observe that for any matrix $A$,
    \[
        |\lambda_1(A) - \max_{1 \le j \le N} \bu_j^\top A \bu_j| \le 2 \varepsilon \|A\|_{\op}.
    \]
    Indeed for any $\bu \in \bbS^{n - 1}$, we have $1 \le j \le N$, s.t. $\|\bu - \bu_j\| \le \varepsilon$. Hence
    \[
        |\bu^\top A \bu - \bu_j^\top A \bu_j| \le |\bu^\top A (\bu - \bu_j)| + |(\bu - \bu_j)^\top A \bu_j| \le 2 \varepsilon \|A\|_{\op}.
    \]
    Therefore
    \begin{align*}
        |\bbE &\lambda_1(M(1)) - \bbE \lambda_1(M(0))| \\
        &\le 2 \varepsilon (\bbE \|M(1)\|_{\op} + \bbE \|M(0)\|_{\op}) + |\bbE \max_{1 \le j \le N} \bu_j^\top M(1) \bu_j - \bbE \max_{1 \le j \le N} \bu_j^\top M(0) \bu_j|.
    \end{align*}
    Since $\bbE\|Z_n\|_{\op} = O(\sqrt{n})$, we have that
    \[
        \bbE \|M(\alpha)\|_{\op} \le \frac{C_1\sqrt{n}}{4}
    \]
    for some universal constant $C_1 > 0$.
    
    On the other hand, an application of the Sudakov-Fernique inequality of Chatterjee \cite[Theorem 1.2]{chatterjee2005error}\footnote{See also Theorem~1 and Comment~1 of \cite{chernozhukov2015comparison}.} together with Lemma~\ref{lem:interpolation}, gives
    \begin{align*}
        |\bbE \max_{1 \le j \le N} \bu_j^\top M(1) \bu_j &- \bbE \max_{1 \le j \le N} \bu_j^\top M(0) \bu_j| \\
        &\le C' \sqrt{2 C \|\bu - \bv\|_{2} \log N} \\
        &\le C_2 \sqrt{n\|\bu - \bv\|_{2} \log\big(\tfrac{1}{\varepsilon}\big)},
    \end{align*}
    where we have used the trivial covering number bound $N = O((\frac{1}{\varepsilon})^{n - 1})$. It follows that
    \[
        |\bbE \lambda_1(M(1)) - \bbE \lambda_1(M(0))| \le C_1 \varepsilon \sqrt{n} + C_2 \sqrt{n\|\bu - \bv\|_{2} \log\big(\tfrac{1}{\varepsilon}\big)}.
    \]
    This completes the proof.
\end{proof} 
\subsection{Proof of Theorem~\ref{thm:mixed_edge}}
\begin{proof}[Proof of Theorem~\ref{thm:mixed_edge}]
Equation \eqref{eq:mixed_low_rank} states that
    \[
    \cG \cdot \bu \otimes \bv \overset{d}{=} \sqrt{\frac{1 + \rho_n^2}{6}}\bigg( \frac{1}{\sqrt{1 + \rho_n^2}}Q +  Z_n\bigg),
    \]
    where $Q_n := U (\bu \bv^\top + \bv \bu^\top) + \bV_1 \bu^\top + \bu \bV_1^\top + \bV_2 \bv^\top + \bv \bV_2^\top$
    is of rank at most $4$. We first obtain the eigenvalues of $Q_n$. Define $\bV'_1:= \bV_1/\|\bV_1\|$ and $\bV'_2:=\bV_2/\|\bV_2\|$. Now we represent $Q_n$ with respect to a basis of the form $\{\bu, \bv, \bV'_1, \bV'_2, \bw_5, \bw_6, \ldots, \bw_n\}$ where $\{\bw_5, \bw_6, \ldots, \bw_n\}$ is a set of orthonormal vectors which is orthogonal to $\{\bu, \bv, \bV_1, \bV_2\}$. It is clear that only the top-left $4 \times 4$ submatrix of $Q_n$ can be non-zero. More precisely, this submatrix is given by
    \[
        \begin{pmatrix}
            \rho_n U + (\bu^\top \bV_1) & U + (\bv^\top \bV_1) &  U(\bv^\top \bV'_1) + \|\bV_1\|& U (\bv^\top \bV_2) + ({\bV}_1^\top \bV'_2)\\
            U+(\bu^\top \bV_2) & \rho_n U + (\bv^\top \bV_2) & U(\bu^\top \bV'_1) + ({\bV'_1}^\top \bV_2) & U(\bu^\top \bV'_2) + \|\bV_2\|\\
            \|\bV_1\| & \rho_n\|\bV_1\| & \|\bV_1\|({\bu}^\top \bV'_1) & \|\bV_1\|(\bu^\top \bV'_2)\\
            \rho_n \|\bV_2\| & \|\bV_2\| & \|\bV_2\|(\bv^\top \bV'_1) & \|\bV_2\|(\bv^\top \bV'_2)
        \end{pmatrix}.
    \]
    Using the facts 
    \begin{align*}
        &\|\bV_1\|^2/n =1 + O_p(1/\sqrt{n}), &\|\bV_2\|^2/n =1 + O_p(1/\sqrt{n}),\\ &\bu^\top \bV_1 \sim \cN(0,1), &\bv^\top \bV_1 \sim \cN(0,1),\\
        &\bu^\top \bV_2 \sim \cN(0,1), &\bv^\top \bV_2 \sim \cN(0,1),\\
        &(\bV_1^\top \bV_2)/n = \rho_n + O_p(1/\sqrt{n}),\\
    \end{align*}
    we can conclude that the previous matrix, scaled by $\sqrt{n}$, is $O_p(1 /\sqrt{n})$-close to the following matrix in the operator norm:
    \[
        \begin{pmatrix}
            0 & 0 & 1 & \rho_n\\
            0 & 0 & \rho_n & 1\\
            1 & \rho_n & 0 & 0\\
            \rho_n & 1 & 0 & 0
        \end{pmatrix} = \begin{pmatrix}
            0 & 1\\
            1 & 0
        \end{pmatrix}
        \otimes
        \begin{pmatrix}
            1 & \rho_n \\
            \rho_n & 1
        \end{pmatrix}.
    \]
    The eigenvalues of this matrix are given by $\{\pm(1 + \rho_n), \pm(1 - \rho_n)\}$. Therefore
    \[
        \frac{1}{\sqrt{n}}Q_n = (1+\rho)\bx_1 \bx_1^\top + (1-\rho)\bx_2 \bx_2^\top+ (-1+\rho)\bx_3 \bx_3^\top +(-1-\rho) \bx_4 \bx_4^\top + E,
    \]
    where $\bx_1,\, \bx_2,\, \bx_3,\, \bx_4 $ are orthonormal and $\|E\|_\op = O_p(1/\sqrt{n}) + O(|\rho_n -\rho|)$. Notice that if $\rho >0$, then only $\pm(1 + \rho)$ exceed the BBP threshold of $1$. Similarly, if $\rho < 0$, then only $\pm(1 - \rho)$ exceed the BBP threshold. Finally, if $\rho= 0$, then none of the eigenvalues is bigger than $1$ in absolute value. Thus, We have two outlier eigenvalues (one at the left and one at the right of the bulk) if $\rho \ne 0$ and no outlier eigenvalue if $\rho =0$. Now Theorem~\ref{thm:mixed_edge} follows using Theorem~2.1 of \cite{benaych2011eigenvalues} as in the proof of Theorem~\ref{thm:edge-limits}-(i).
\end{proof}

\section*{Acknowledgements}
The authors would like to thank Dipranjan Pal for helpful comments. They would also like to thank Luca Avena, Elia Bisi and Eleonora Bordiga for bringing the bulk universality result in \cite{vanhandel2016probability} to their attention.
\section*{Funding}
SSM was partially supported by the Prime Minister Early Career Research Grant \seqsplit{ANRF/ECRG/2024/006704/PMS} from the Anusandhan National Research Foundation, Govt.~of India.
\bibliographystyle{apalike}
\bibliography{refs.bib}

\appendix
\section{Proofs of supporting lemmas}\label{sec:supp_lemma}

\begin{proof}[Proof of Lemma~\ref{lem:concentration}]
    We first show that $\lambda_1(\cG \cdot \bx_3 \otimes \cdots \otimes \bx_r)$ can be viewed as a $\sqrt {r}$-Lipschitz function of a collection of $\binom{n+r-1}{r}$ many i.i.d. standard Gaussian variables. Towards that end, first notice that there are $\binom{n+r-1}{r}$ many distinct $r$-multisets with entries from the set $[n]$. Let $\mathcal{H} = \mathcal{H}_{\{i_1, \ldots,i_r\}}$ be a collection of i.i.d. Gaussian variables indexed by the $r$-multisets with entries from $[n]$. Now, one can construct $\tilde{\mathcal{H}} = ((\tilde{\mathcal{H}}_{i_1 \cdots i_r}))$, a tensor of order $r$, as follows 
    \[
        \tilde{\mathcal{H}}_{i_1 \cdots i_r} := \sqrt{\frac{r}{\#\perms(i_1, \ldots, i_r)}}\mathcal{H}_{\{i_1, \ldots, i_r\}}.
    \]
    Clearly, $\tilde{\mathcal{H}} \in \GOTE(r, n)$. Also, the map from $\bbR^{\binom{n+r-1}{r}}$ to $\fS_{r, n}$ defined by $\mathcal{H} \mapsto \tilde{\mathcal{H}}$ is $\sqrt{r}$-Lipschitz. Indeed,
    \begin{align*}
        \|\tilde{\mathcal{H}} - \tilde{\mathcal{H}}'\|^2_{\mathrm{F}} &= \sum_{i_1, \cdots, i_r} (\tilde{\mathcal{H}}_{i_1 \cdots i_r} - \tilde{\mathcal{H}}'_{i_1 \cdots i_r})^2\\
        &= \sum_{i_1, \ldots, i_r} (\mathcal{H}_{\{i_1, \ldots, i_r\}} - \mathcal{H}'_{\{i_1, \ldots, i_r\}})^2\\
        &= \sum_{\{i_1, \ldots, i_r\}} \#\perms(i_1, \ldots, i_r)(\mathcal{H}_{\{i_1, \ldots, i_r\}} - \mathcal{H}'_{\{i_1, \ldots, i_r\}})^2\\
        &= r\|\mathcal{H} - \mathcal{H}'\|_2^2.
    \end{align*}
    We now show that the map from $(\fS_{r, n}, \|\cdot\|_{\mathrm{F}})$ to $\bbR$ defined by $\mathcal{A} \mapsto \lambda_1(\mathcal{A}\cdot \bx_3 \otimes \cdots \otimes \bx_r)$ is $1$-Lipschitz. For $\mathcal{A}, \mathcal{B} \in \fS_{r, n}$, we have by the Cauchy-Schwarz inequality,
    \begin{align*}
        &\|\mathcal{A} \cdot \bx_3 \otimes \cdots \otimes  \bx_r - \mathcal{B} \cdot \bx_3 \otimes \cdots \otimes \bx_r\|_{\mathrm{F}}^2\\
        &\qquad= \sum_{i_1, i_2} \bigg(\sum_{i_3,\ldots, i_r} (\mathcal{A}_{i_1 i_2 i_3 \cdots i_r} - \mathcal{B}_{i_1 i_2 i_3 \cdots i_r})(\bx_3)_{i_3} \ldots (\bx_r)_{i_r}\bigg)^2 \\
        &\qquad\leq \sum_{i_1,i_2} \bigg(\sum_{i_3, \ldots, i_r} (\mathcal{A}_{i_1 i_2 i_3 \cdots i_r} - \mathcal{B}_{i_1 i_2 i_3 \cdots i_r})^2\bigg) \bigg(\sum_{i_3,\ldots, i_r} (\bx_3)_{i_3}^2 \cdots (\bx_r)_{i_r}^2\bigg) \\
        &\qquad= \|\mathcal{A} - \mathcal{B}\|_{\mathrm{F}}^2, 
    \end{align*}
    that is, the map from $(\fS_{r, n}, \|\cdot\|_{\mathrm{F}})$ to $(M_n(\bbR), \|\cdot\|_{\mathrm{F}})$ defined by $\mathcal{A} \mapsto \mathcal{A} \cdot \bx_3 \otimes \cdots \otimes \bx_r $ is $1$-Lipschitz. Further, for a symmetric matrix $A$, the map $A \mapsto \lambda_1(A)$ is $1$-Lipschitz with respect to the operator and hence the Frobenius norm. Therefore the composite map $\mathcal{A} \mapsto \lambda_1 (\mathcal{A} \cdot \bx_3 \otimes \cdots \otimes \bx_r) $ is also $1$-Lipschitz. Now, by the concentration of Lipschitz functions of i.i.d. standard Gaussian random variables (see, e.g., Theorem~5.6 in \cite{boucheron2013concentration}), we have
    \[
        \bbP(|\lambda_1(n^{-1/2}\cG \cdot \bx_3 \otimes \cdots \otimes \bx_r) - \bbE \lambda_1(n^{-1/2}\cG \cdot \bx_3 \otimes \cdots \otimes \bx_r)| > \epsilon) \le 2 \exp(- n\epsilon^2/ 2r).
    \]
    This completes the proof of the first concentration inequality.
    
    For the second one, start with any $1$-Lipschitz function $f$. By virtue of the Cauchy-Schwarz and Hoffman-Wielandt inequalities we get that the real-valued function on $(M_n(\bbR), \|\cdot\|_{\mathrm{F}})$ defined by $A \mapsto \langle f, \mu_A\rangle$ is $\frac{1}{\sqrt{n}}$-Lipschitz. Thus, $\langle f, \mu_{n^{-1/2}L_n}\rangle$ can be viewed as a $\frac{\sqrt{r}}{n}$-Lipschitz function of $\binom{n+r-1}{r}$ many i.i.d. standard Gaussian variables. So, again by Gaussian Lipschitz concentration, we get that
    \[
        \bbP(\langle f, \mu_{n^{-1/2}L}\rangle > \epsilon) \le 2 \exp(- n^2\epsilon^2/ 2r).
    \]
    Now, using a covering argument as in the proofs of Theorems 1.3 and 1.4 in \cite{guionnet2000concentration}, we obtain the desired inequality. 
\end{proof}

\begin{proof}[Proof of Lemma~\ref{lem:C_S_bounds}.]
    Let $t_i = t_i(S)$ denote the multiplicity of $i$ in $S$ and let $t_{\max} \equiv t_{\max}(S) = \max_{i} t_i$. Clearly, $\sum_i t_i = r$. For $i,j\in[n]$, define
    \[
        \gamma_{ij}:= \begin{cases}
            t_i t_j, & i \ne j,\\
            t_i(t_i - 1), & i=j.
        \end{cases}
    \]
    We now prove \eqref{eq:C_S_frobenius}. To that end
    \begin{align*}
        \|C_S\|_F^2 &= \sum_{i,j}(C_S)_{ij}^2 \\
        &= \frac{\#\perms(S)}{r(r-1)^2} \sum_{i, j} \gamma_{ij}^2 w_{S\setminus\{i,j\}}^2 \\&
        \le \frac{\#\perms(S)}{r(r-1)^2} \|\bw\|_{\infty}^{2(r-2)} \sum_{i,j}\gamma_{ij}^2 .
    \end{align*}
    Now
    \begin{align*}
    \sum_{i,j}\gamma_{ij}^2 &= \sum_{i\neq j}t_i^2 t_j^2 + \sum_i t_i^2 (t_i - 1)^2 \\
    &= \left(\sum_i t_i^2\right)^2
    -2\sum_i t_i^3
    +\sum_i t_i^2.
    \end{align*}
    Using $A := \sum_i t_i^2\le r t_{\max}$, and the fact (by Cauchy-Schwarz) that
    \[
        \sum_i t_i^3 \ge A^2 / r,
    \]
    we obtain
    \[
        \sum_{i, j} \gamma_{ij}^2 \le A^2 - 2 A^2 / r + A \le r (r - 2) t_{\max}^2 + r t_{\max}.
    \]
    This is sharp when $S$ does not have any repeated entry, as well as when $S$ has only one entry repeated $r$ times (equaling $r(r - 1)$ and $r^2 (r - 1)^2$, respectively). This gives the bound
    \[
        \|C_S\|_F \le \frac{\sqrt{\#\perms(S)}}{(r-1)} \|\bw\|_{\infty}^{(r-2)} \sqrt{t_{\max} ((r - 2) t_{\max} + 1)}.
    \]
    We shall maximise
    \[
        \#\perms(S) \cdot t_{\max} ((r - 2) t_{\max} + 1)
    \]
    over $S$. When $S$ has no repeated entries, this equals
    \[
        r! (r - 1).
    \]
    On the other hand, when $S$ has only one distinct entry, it simplifies to
    \[
        r (r - 1)^2.
    \]
    So it is clear that the maximum will be close the case where $S$ has many distinct entries. In fact, clearly,
    \[
        \prod_{i} t_i! \ge t_{\max}!.
    \]
    So let us upper bound
    \[
        \frac{t_{\max} ((r - 2) t_{\max} + 1)}{t_{\max}!}.
    \]
    Consider $f : \{ 1, \ldots, r \} \to \bbR$ defined by
    \[
        f(k) := \frac{ak + 1}{(k - 1)!}
    \]
    with $a:= r-2$. Then 
    \[
        \frac{f(k + 1)}{f(k)} = \frac{1}{k} \frac{a(k + 1) + 1}{ak + 1}.
    \]
    This is $< 1$ if and only if
    \[
        ak + a + 1 < ak^2 + k,
    \]
    i.e.,
    \[
        a (k^2 - k - 1) + k - 1 > 0 
    \]
    Now $k$ ranges over $1$ to $r$. So we see that this is true for any $k \ge 2$. For $k = 1$, the above condition is false. We conclude that $f(1) < f(2) > f(3) > f(4) \cdots > f(r)$. That is, the maximiser is $k = 2$. Hence, we conclude
    \[
        \#\perms(S) \cdot t_{\max} ((r - 2) t_{\max} + 1) \le r! (2r - 3).
    \]
    This leads to the bound
    \[
        \|C_S\|_F \le \frac{\sqrt{2r - 3}}{r - 1} \sqrt{r!} \|\bw\|_{\infty}^{r - 2} \le 2 \sqrt{(r - 1)!} \|\bw\|_{\infty}^{r - 2}.
    \]
    This finishes the proof of \eqref{eq:C_S_frobenius}.
\end{proof}

\begin{proof}[Proof of Lemma~\ref{lem:eigen-structure}]
    Consider an ordered basis $\mathfrak{B} = \{\bx, \by, \bv_3, \bv_4, \ldots, \bv_n\}$ of $\bbR^n$, where $\{\bv_3, \bv_4, \ldots, \bv_n\}$ is a linearly independent set which is orthogonal to both $\bx$ and $\by$. In the basis $\mathfrak{B}$, the matrix $B$ has the following representation:
    \[
        \begin{pmatrix}
            a + b \bx^\top \by  & a \bx^\top \by + \beta & 0 & 0 & \cdots & 0 \\
            b & b \bx^\top \by & 0 & 0 & \cdots & 0 \\
            0 & 0 & 0 & 0 & \cdots & 0 \\
            \vdots &\vdots &\vdots &\vdots & & \vdots \\
            0 & 0 & 0 & 0 & \cdots & 0
        \end{pmatrix}.
    \]
    Therefore $\lambda_{\pm}$ are the roots of the equation
    \[
        \lambda^2 + (a + 2 b \bx^\top \by) \lambda - b^2 (1 - (\bx^{\top} \by)^2) = 0,
    \]
    from which (i) follows.

    Now we have
    \begin{align*}
        \bx^\top \bv_{\pm} &= \gamma_{\pm} + \delta_{\pm} \bx^\top \by, \\
        \by^\top \bv_{\pm} &= \gamma_{\pm} \bx^\top \by + \delta_{\pm}.
    \end{align*}
    Using the eigenvalue equations, we have
    \begin{align*}
        \lambda_{\pm} (\gamma_{\pm} \bx + \delta_{\pm} \by) &= \lambda_{\pm} \bv_{\pm} \\
        &= B \bv_{\pm} \\
        &= (a \bx\bx^\top + b(\bx \by^\top + \by \bx^\top)) \bv_{\pm} \\
        &= (a \bx^\top \bv_{\pm} + b \by^\top \bv_{\pm}) \bx + b (\bx^\top \bv_{\pm}) \by \\
        &= \big(a (\gamma_{\pm} + \delta_{\pm} \bx^\top \by) + b (\gamma_{\pm} \bx^\top \by + \delta_{\pm})\big) \bx + b (\gamma_{\pm} + \delta_{\pm} \bx^\top \by) \by,
    \end{align*}
    which gives
    \[
        \lambda_{\pm} \delta_{\pm} = b(\gamma_{\pm} + \delta_{\pm} \bx^\top \by),
    \]
    which yields the relation
    \begin{equation}\label{eq:gamma-delta}
        \gamma_{\pm} = \frac{\lambda_{\pm} - b \bx^\top \by}{b} \delta_{\pm}.
    \end{equation}
    We also have
    \begin{equation}\label{eq:norm-1}
        1 = \|\bv_{\pm}\|^2 = \|\gamma_{\pm} \bx + \delta_{\pm} \by\|^2 = \gamma_{\pm}^2 + \delta_{\pm}^2 + 2 \gamma_{\pm} \delta_{\pm} \bx^\top \by.
    \end{equation}
    Together \eqref{eq:gamma-delta} and \eqref{eq:norm-1} give
    \begin{align*}
        1 &= \delta_{\pm}^2 \bigg[1 + \bigg(\frac{\lambda_{\pm} - b \bx^\top \by}{b}\bigg)^2 + 2 \frac{\lambda_{\pm} - b \bx^\top \by}{b} \bx^\top \by \bigg] \\
        &=\delta_{\pm}^2 \frac{\lambda_{\pm}^2 + b^2 (1 - (\bx^\top \by)^2)}{b^2} \\
        &= \delta_{\pm}^2 \frac{s_{\pm}^2}{b^2}.
    \end{align*}
    Since $\delta_{\pm}$ are assumed to be non-negative, we have $\delta_{\pm} = \frac{|b|}{s_{\pm}}$. Then from \eqref{eq:gamma-delta}, we get that $\gamma_{\pm} = \sign(b) \frac{\lambda_{\pm} - b \bx^\top \by}{s_{\pm}}$. This gives us (ii). Then (iii) follows immediately.
\end{proof}

\begin{proof}[Proof of Lemma~\ref{lem:P_n_fluctuation}]
    We give a detailed proof of \eqref{eq:P_n_fluctuation_r_fixed}. One can show \eqref{eq:P_n_fluctuation_1<<r<<n}, \eqref{eq:P_n_fluctuation_r/n=c} and \eqref{eq:P_n_fluctuation_r>>n} similarly. Using Lemma~\ref{lem:eigen-structure}, we can find the eigenvalues of $P_n$:
   \begin{align*}
        \lambda_1(P_n) &=   \frac{\alpha}{2} U + {\beta} \bw^\top \bV + \frac{1}{2} \sqrt{\alpha^2 U^2 + 4\alpha\beta U \bw^\top \bV + 4 \beta^2 \|\bV\|^2}, \\
        \lambda_1(P_n) &= \frac{\alpha}{2} U + \beta \bw^\top \bV - \frac{1}{2} \sqrt{\alpha^2 U^2 + 4\alpha\beta U \bw^\top \bV + 4 \beta^2 \|\bV\|^2}, \\
        \lambda_k(P_n) &= 0 \quad \text{for } k = 2, 3, \ldots, n - 1.
    \end{align*}Notice that
    \begin{align*}
        \sqrt{n}&\bigg(\frac{\lambda_1(P_n)}{\sqrt{n}} - \beta\bigg)\\
        &= \frac{\alpha}{2}U + \beta \bw^\top \bV + \frac{\frac{\alpha^2}{4\sqrt n} U^2 + \frac{\alpha\beta}{\sqrt n} U \bw^\top \bV +\beta^2\sqrt n \bigg(\frac{||\bV||^2}{n} -1\bigg)}{\sqrt{\frac{\alpha^2}{4 n} U^2 + \frac{\alpha\beta}{ n} U \bw^\top \bV +\beta^2\frac{||\bV||^2}{n}} + \beta}\\
        &= \frac{\alpha}{2}U + \beta \bw^\top \bV + \frac{\beta^2\sqrt n \bigg(\frac{||\bV||^2}{n} -1\bigg)}{\sqrt{\frac{\alpha^2}{4 n} U^2 + \frac{\alpha\beta}{ n} U \bw^\top \bV +\beta^2\frac{||\bV||^2}{n}}+\beta} + O_\bbP(1)
    \end{align*}
    By the central limit theorem, 
    \[
        \sqrt n \bigg(\frac{||\bV||^2}{n} -1\bigg) \convd \cN(0,2);
    \]
    in particular, it is $O_\bbP(1)$. Note the following elementary fact: if $\{X_n\}$ is an $O_\bbP(1)$-sequence of random variables and $\{Y_n\}$ is such that $Y_n$ converges in probability to some non-zero constant $y$, then $\frac{X_n}{Y_n} - \frac{X_n}{y} = o_\bbP(1)$. Using this we can further write
    \begin{align*}
        \sqrt{n}&\left(\frac{\lambda_1(P_n)}{\sqrt n} - \beta \right) \\
        &= \frac{\alpha}{2}U + \beta \bw^\top \bV + \frac{\beta}{2} \sqrt n\bigg(\frac{||\bV||^2}{n}-1\bigg) + o_\bbP(1)\\
        &= \frac{\alpha}{2}U + \beta \bw^\top \bV + \frac{\beta}{2} \sqrt n\bigg(\frac{\bV^\top (I-\bw \bw^\top) \bV}{n}-1\bigg) + \frac{\beta}{2\sqrt n} (\bw^\top\bV)^2+ o_\bbP(1)\\
        &= \frac{\alpha}{2}U + \beta \bw^\top \bV + \frac{\beta}{2} \sqrt n\bigg(\frac{\bV^\top (I-\bw \bw^\top) \bV}{n}-1\bigg)+ o_\bbP(1).
    \end{align*}
    Notice that $\bw^\top \bV$ is independent of $\bV^\top (I - \bw \bw^\top) \bV$. Similarly, we can show that
    \[
        \sqrt n\bigg(\frac{\lambda_n(P_n)}{\sqrt{n}} + \beta\bigg) = \frac{\alpha}{2}U + \beta \bw^\top \bV - \frac{\beta}{2} \sqrt n\bigg(\frac{\bV^\top (I-\bw \bw^\top) \bV}{n}-1\bigg)+ o_\bbP(1).
    \]
    Fix arbitrary $a, b\in \bbR$. Then,
    \begin{align*}
        a\sqrt n\bigg(\frac{\lambda_1(P_n)}{\sqrt{n}} - \beta\bigg) +b \sqrt n\bigg(\frac{\lambda_n(P_n)}{\sqrt{n}} + \beta\bigg) =& (a+b)\bigg(\frac{\alpha}{2}U + \beta \bw^\top \bV\bigg)\\
        &+ (a-b)\frac{\beta}{2} \sqrt n\bigg(\frac{\bV^\top (I-\bw \bw^\top) \bV}{n}-1\bigg).
    \end{align*}
    Notice that $U, \bw^\top \bV$ are $\cN(0,1)$ and 
    \[
        \sqrt n\bigg(\frac{\bV^\top (I-\bw \bw^\top) \bV}{n}-1\bigg) = \sqrt n\bigg(\frac{||\bV||^2}{n} + -1\bigg) + o_\bbP(1) \convd \cN(0,2).
    \]
    Further, they are all independent. Then, $a\sqrt n(\frac{\lambda_1(P_n)}{\sqrt{n}} - \beta) +b \sqrt n(\frac{\lambda_n(P_n)}{\sqrt{n}} + \beta)$ converges in distribution to a Gaussian random variable with variance $(a+b)^2(\alpha^2/4 + \beta^2) + (a-b)^2 \frac{\beta^2}{2}$. Suppose $\tau, \tilde \tau$ are i.i.d. $\cN(0,1)$. Then, it is easy to see that $a(\sqrt{\alpha^2/4 + \beta^2}\tau +  \frac{\beta}{\sqrt 2} \tilde\tau) + b(\sqrt{\alpha^2/4 + \beta^2}\tau -  \frac{\beta}{\sqrt 2} \tilde\tau)$ is Gaussian and has the same variance as above. Since the asymptotic distribution of a random vector is uniquely determined by the distribution of the linear combinations of the coordinates, we can conclude that
    \begin{align*}
        &\bigg(\sqrt n\bigg(\frac{\lambda_1(P_n)}{\sqrt{n}} - \beta\bigg),\sqrt n\bigg(\frac{\lambda_n(P_n)}{\sqrt{n}} + \beta\bigg)\bigg) \\
        &\qquad\qquad\qquad\qquad\convd \bigg(\sqrt{\alpha^2/4 + \beta^2}\tau + \frac{\beta}{\sqrt 2} \tilde\tau, \sqrt{\alpha^2/4 + \beta^2}\tau - \frac{ \beta}{\sqrt 2} \tilde\tau\bigg),
    \end{align*}
    and hence we get \eqref{eq:P_n_fluctuation_r_fixed}.
\end{proof}        

\begin{proof}[Proof of Lemma~\ref{lem:comp_spectrum}]
    We prove the following estimates.
    
    Note that from Lemma~3.1 of the main text we can get that for $i \ne j$,
    \begin{align*}
        (\bbE \tilde{M}_n^2)_{ij} &= (n+2) \beta^2 w_i w_j + \alpha^2 w_iw_j,\\
        (\bbE \tilde{M}_n^2)_{ii} &= (n+2) \beta^2w_i^2 + \alpha^2 w_i^2 + (n+1) \theta^2 + \beta^2,
    \end{align*}
    that is
    \[
        \bbE \tilde{M}_n^2 = (\alpha^2 + (n+2)\beta^2)\bw \bw^\top + (\beta^2 + (n+1)\theta^2) I_n.
    \]
    Thus
    \[
        \|\bbE \tilde{M}_n^2\|_\op = \alpha^2 + (n+3)\beta^2 +  (n+1)\theta^2 = O(r + n), 
    \]
    that is $\sigma(\tilde{M}_n) = O(\sqrt{n} + \sqrt{r})$.
    Now for $\bu, \bv \in \bbS^{n-1}$, let us calculate $\bbE ((\bu^\top \tilde{M}_n \bv)^2)$. Since this depends only on the covariances between the entries of $\tilde{M}_n$, we can replace $\tilde{M}_n$ by $X_n$ as defined in Lemma~3.2 in the main text. So
    \begin{align*}
        &\bbE ((\bu^\top \tilde{M}_n \bv)^2) \\
        &= \bbE ((\bu^\top X_n \bv)^2)\\
        &= \bbE((\bu^\top (\alpha U \bw \bw^\top + \beta(\bw \bV^\top + \bV \bw^\top) + \theta Z_n)\bv)^2)\\
        &= \bbE ((\alpha U (\bu^\top \bw)(\bw^\top \bv) + \beta (\bu^\top \bw)(\bV^\top \bv) + \beta(\bu^\top \bV)(\bw^\top \bv) + \theta \bu^\top Z_n \bv)^2)\\
        &\le 4 (\bbE(\alpha U (\bu^\top \bw)(\bw^\top \bv))^2 + \bbE(\beta (\bu^\top \bw)(\bV^\top \bv))^2 \\
        &\qquad\qquad\qquad\qquad+ \bbE(\beta(\bu^\top \bV)(\bw^\top \bv))^2 + \bbE(\theta \bu^\top Z_n \bv)^2)\\
        &= 4 \alpha^2 (\bu^\top \bw)^2(\bw^\top \bv)^2 + \beta^2 (\bu^\top \bw)^2 + \beta^2(\bw^\top \bv)^2 + \theta^2 \bbE(\bu^\top Z_n \bv)^2.
    \end{align*}
    Now
    \begin{align*}
        \bbE(\bu^\top Z_n \bv)^2 &= \bbE\bigg(\bigg(\sum_{i,j} Z_{ij}u_i v_j\bigg)^2\bigg) \\
        &= \sum_{i,j,k,l} \Cov(Z_{ij}, Z_{kl}) u_i v_j u_k v_l\\
        &= \sum_{i \ne j} (u_i^2 v_j^2 + u_i u_j v_i v_j) + \sum_i 2 u_i^2 v_i^2\\
        &= 1+(\bu^\top \bv)^2.
    \end{align*}
    Thus 
    \[
        \sigma_*(\tilde{M}_n) = O(\sqrt{r}).
    \]
    Now, recall equation~\eqref{eq:C_S_frobenius} of Lemma~\ref{lem:C_S_bounds}. This yields
    \begin{align*}
        \|\cT'_S C_S\|_\op 
        &\le 2 K_n \sqrt{(r-1)!} \|\bW\|_\infty^{r-2}.
    \end{align*}
     and hence (iii) follows. 
    This finishes the proof of the lemma. 
\end{proof}

The proof of Lemma~\ref{lem:local_law_step_2} uses Lemma~\ref{lem:local_law_step_1} whose proof in turn relies on Lemma~\ref{lem:iso_local_semicircle_law}, the simplified version of the isotropic local semicircle law (which follows from the isotropic local semicirle law in \cite{knowles2013isotropic}.)

\begin{lemma}\label{lem:local_law_step_1}
    Suppose $\bx$ and $ \by$ are possibly random vectors in $\bbR^n$ and are independent of $Z'_n$. If $\|\bx\| = O_\prec(1)$ and $\|\by\| = O_\prec(1)$, then 
    \begin{align}\label{eq:local_law_step_1}
        \la \bx, \by \ra_{M'_n} 
        &= s_\sc(z) \la \bx, \by \ra - \frac{\beta' s_\sc(z)}{\sqrt n} \la \bx, \bw\ra \la \bV, \by\ra_{M'_n} \\ \nonumber
        &\qquad\qquad - \frac{\beta' s_\sc(z)}{\sqrt n} \la \bx, \bV\ra \la \bw, \by\ra_{M'_n} + O_\prec\bigg( \frac{1}{\sqrt n}\bigg).
    \end{align}
\end{lemma} 

\begin{proof}[Proof of Lemma~\ref{lem:local_law_step_1}]
    Notice that if $\bx,\, \by \in \bbR^n$, then from \eqref{eq:iso_local_semicircle}, we get that
    \begin{equation} \label{eq:iso_local_semicircle_Z}
        \langle \bx, \by \rangle_{Z'_n} =s_\sc(z) \la \bx, \by \ra + O_\prec\bigg(\frac{\|\bx \| \|\by \|}{\sqrt n}\bigg).
    \end{equation}
    Further, one can let $\bx,\, \by$ to be random and independent of $Z'_n$ in \eqref{eq:iso_local_semicircle_Z}. Another crucial observation is that
    \begin{align}\label{eq:ST_bound}
        &\la \bx, \by\ra_{Z'_n} \leq \frac{\|\bx \| \|\by\|}{\Im(z)},
        & \la \bx, \by\ra_{M'_n} \leq \frac{\|\bx \| \|\by\|}{\Im(z)}.
    \end{align}
    Now, the resolvents of $M'_n$ and $Z'_n$ are related as follows:
    \begin{align*}
        &(M'_n -zI)^{-1} - (Z'_n -zI)^{-1} \\
        &\qquad\qquad = - (Z'_n -zI)^{-1} \bigg(\frac{\alpha'U}{\sqrt n} \bw \bw^\top + \frac{\beta'}{\sqrt n}\bw \bV^\top + \frac{\beta'}{\sqrt n}\bV \bw^\top \bigg) (M'_n -zI)^{-1}.
    \end{align*}
    Upon left multiplying by $\bx^\top$ and right multiplying by $\by$, we get
    \begin{align*}
        &\la \bx, \by \ra_{M'_n} \\
        &= \la\bx, \by\ra_{Z'_n} - \frac{\alpha'U}{\sqrt n} \la\bx, \bw \ra_{Z'_n} \la \bw, \by\ra_{M'_n} - \frac{\beta'}{\sqrt n} \la\bx, \bw\ra_{Z'_n} \la\bV, \by\ra_{M'_n} - \frac{\beta'}{\sqrt n}\la \bx, \bV \ra_{Z'_n} \la \bw, \by \ra_{M'_n}  \\
        &= \bigg\{s_\sc(z) \la \bx, \by \ra  + O_\prec\bigg(\frac{\|\bx\|\|\by\|}{\sqrt n}\bigg)\bigg\}
        - \bigg\{ \frac{\alpha'  s_\sc(z) U}{ \sqrt n} \la \bx, \bw\ra \la \bw, \by \ra_{M'_n} + O_\prec\bigg( \frac{\|\bx\| \|\by\|}{n}\bigg)\bigg\}\\ 
        & \qquad\qquad - \bigg\{\frac{\beta' s_\sc(z)}{\sqrt n} \la \bx, \bw\ra \la \bV, \by\ra_{M'_n} + O_\prec\bigg( \frac{\|\bx\|\|\by\|\|\bV\|}{n}\bigg) \bigg\} \\
        & \qquad\qquad\qquad\qquad - \bigg\{\frac{\beta' s_\sc(z)}{\sqrt n} \la \bx, \bV\ra \la \bw, \by\ra_{M'_n} + O_\prec\bigg( \frac{\|\bx\|\|\by\|\|\bV\|}{n}\bigg)\bigg\}.
    \end{align*}
    In the previous line we have repeatedly used \eqref{eq:iso_local_semicircle_Z} and \eqref{eq:ST_bound}. From standard concentration results of chi-squared random variables (see, e.g., Proposition 2.10 of \cite{wainwright2019high}), we can conclude that
    \[
        \|\bV\| = O_\prec(\sqrt n).
    \]
    Further notice that 
    \[
        \bigg|\frac{\alpha'  s_\sc(z)}{ \sqrt n} \la \bx, \bw\ra \la \bw, \by \ra_{M'_n}\bigg| \leq \frac{\alpha'\|\bx\|\|\by\||U|}{\sqrt n \Im(z)^2}.
    \]
    The proof finishes by recalling the facts that $\|\bx\| = O_\prec(1)$, $\|\by\| = O_\prec(1)$ and $U = O_\prec(1)$.
\end{proof}

\begin{proof}[Proof of Lemma~\ref{lem:local_law_step_2}]
    In order to prove \eqref{eq:local_law_step_2_eq_1}, in \eqref{eq:local_law_step_1} we take $\bx=\bw$, $\by = \bw$ and $\bx= \frac{1}{\sqrt n}\bV$, $\by = \bw$ to get
    \begin{align*}
        \la \bw, \bw\ra_{M'_n} &= s_\sc(z) - \frac{\beta' s_\sc (z)}{\sqrt n} \la \bV, \bw\ra_{M'_n} - \frac{\beta's_\sc(z)}{\sqrt n}\la \bw, \bV\ra \la \bw, \bw\ra_{M'_n} +O_\prec\bigg(\frac{1}{\sqrt n}\bigg), \\
        \frac{1}{\sqrt n} \la \bV, \bw \ra_{M'_n} &= \frac{1}{\sqrt n}s_\sc(z) \la \bV, \bw\ra - \frac{\beta's_\sc(z)}{n} \la\bV, \bw\ra \la\bV, \bw\ra_{M'_n} \\
        &\qquad\qquad\qquad\qquad - \frac{\beta' s_\sc(z)}{n} \|\bV\|^2 \la \bw, \bw \ra_{M'_n} + O_\prec\bigg(\frac{1}{\sqrt n}\bigg).
    \end{align*}
    Since $\la \bV, \bw\ra \sim \cN(0,1)$, $\la \bV, \bw \ra = O_\prec(1)$. From the concentration results for the chi-squared distribution, $\frac{\|\bV\|^2}{n} = 1 + O_\prec\big(\frac{1}{\sqrt n}\big)$. So, we can simplify the above two equations as
    \begin{align*}
        \la \bw, \bw \ra_{M'_n} &= s_\sc(z) - \frac{\beta' s_\sc(z)}{\sqrt n} \la \bV, \bw \ra_{M'_n} + O_\prec\bigg(\frac{1}{\sqrt n}\bigg), \\
        \frac{1}{\sqrt n}\la \bV, \bw \ra_{M'_n} &= -\beta' s_\sc(z)\la \bw, \bw\ra_{M'_n} + O_\prec\bigg(\frac{1}{\sqrt n}\bigg).
    \end{align*}
    Solving for $\la \bw, \bw \ra_{M'_n}$ and $\la \bV, \bw \ra_{M'_n}$, we get that
    \begin{align*}
        \la \bw, \bw \ra_{M'_n} &= \frac{s_\sc(z)}{1- \beta'^2 s^2_\sc(z)} +O_\prec\bigg(\frac{1}{\sqrt n}\bigg). 
    \end{align*}
    Next, to prove \eqref{eq:local_law_step_2_eq_2}, put $\bx = \bx$, $\by = \bw$ to get
    \begin{align*}
        \la \bx, \bw \ra_{M'_n} &= s_\sc(z) \la \bx, \bw \ra - \frac{\beta' s_\sc(z)}{\sqrt n} \la \bx, \bw\ra \la \bV, \bw\ra_{M'_n} \\
        &\qquad\qquad- \frac{\beta' s_\sc(z)}{\sqrt n} \la \bx, \bV\ra \la \bw, \bw\ra_{M'_n} + O_\prec\bigg( \frac{1}{\sqrt n}\bigg)
    \end{align*}
    and notice that $\la \bx, \bw \ra = 0$, $\la \bx, \bV \ra \sim \cN(0,1)$ and $\la \bw, \bw \ra_{M'_n} = O_\prec(1)$.\\
    Lastly, \eqref{eq:local_law_step_2_eq_3} follows by putting $\bx =\bx$, $\by = \bx$ in \eqref{eq:local_law_step_1} and using \eqref{eq:local_law_step_2_eq_2}. This completes the proof of lemma~\ref{lem:local_law_step_2}
\end{proof}

\begin{proof}[Proof of Lemma~\ref{lem:Stieltjes_inversion}]
    We shall use the inversion formula for Stieltjes transform to get $\mu$. The inversion formula says that $\mu$ has density at $x \in \bbR$ if $\lim_{z \to x} \frac{1}{\pi} \Im(s(z))$ exists and in that case it is given by the latter. Notice that $s$ has two poles at $s_\sc(z) = \pm 1/\beta'$, that is at $z = \pm (\beta' + 1/\beta')$. Fix $x \in \bbR \backslash \{\pm (\beta' + 1/\beta')\}$. Then the density of $\mu$ at $x$ is given by
    \[
        f(x) = \lim_{z \to x}\frac{1}{\pi} \Im(s(z)) = \frac{1}{\pi} \Im (s(x)) 
    \]
    provided $s$ is continuous at $x$. Recall that $s_\sc(z) = \frac{-z +\sqrt{z^2-4}}{2}$. Clearly, $s_\sc$ is continuous at $x$ if $x \neq \pm 2$. Also, $\Im(s(x)) = 0$ if $x > 2$. If $x <2$, then 
    \begin{align*}
        \frac{1}{\pi}\Im(s(x)) &= \frac{1}{\pi}\Im\bigg( \frac{\rho^2s_\sc(x)}{1- \beta'^2 s^2_\sc(x)} + (1-\rho^2)s_\sc(z)\bigg)\\
        &= \frac{\rho^2}{\pi}\Im \bigg(\frac{s_\sc(x) (1 + \beta'^2 +\beta'^2 x \bar s_\sc(x) )}{|1 + \beta'^2 +\beta'^2 x {s_\sc(x)} |^2}\bigg) +(1-\rho^2) f_\sc(x)\\
        &= \frac{\rho^2(1+\beta'^2)}{(1+\beta'^2 - \frac{\beta'^2 x^2}{2})^2 + \frac{\beta'^4 x^2(4-x^2)}{4} } f_\sc(x) + +(1-\rho^2) f_\sc(x)\\
        &= \frac{\rho^2(1+\beta'^2)}{(1+\beta'^2)^2 - \beta'^2 x^2}. f_\sc(x) +(1-\rho^2) f_\sc(x)\\ 
        &= \rho^2 \frac{d\nu(x)}{dx} + (1-\rho^2)f_\sc(x)
    \end{align*}
    The mass at $\beta' +1/\beta'$ is given by
    \begin{align*}
        -\lim_{z\to (\beta' +1/\beta')} &(z- (\beta' +1/\beta'))s(z) \\
        &= -\lim_{z\to (\beta' +1/\beta')}\frac{ z - (\beta' +1/\beta')}{s_\sc(z) - s_\sc(\beta' + 1/\beta')} \lim_{z\to (\beta' +1/\beta')} (s_\sc(z) + 1/\beta')s(z)\\
        &= \frac{\rho^2}{2\beta'^2 s'_\sc(\beta' + 1/\beta')}
    \end{align*}
    Now, differentiating the equation $s_\sc(z) + 1/ s_\sc(z) = -z$ with respect to $z$, we get
    \[
        s'_\sc(z) -\frac{s'_\sc(z)}{s^2(z)} = -1,
    \]
    that is $s'_\sc(z) = \frac{s^2_\sc(z)}{1 - s^2_\sc(z)}$. So, 
    \[
        s'_\sc(\beta' + 1/\beta') = \frac{1}{\beta'^2 -1}.
    \]
    Hence, the mass at $(\beta' + 1/\beta')$ is given by $\frac{\rho^2}{2} (1-1/\beta'^2)$. The same calculation goes through for $-(\beta' +1/\beta')$.
\end{proof}

\begin{proof}[Proof of Lemma~\ref{lem:concen_lambda_1_P_n}]
    Note that from the proof of Lemma~\ref{lem:P_n_fluctuation},
    \begin{align*}
        \sqrt{\frac{n}{r-1}}&\left(\lambda_1(P_n') - \sqrt{r-2} \right) \\
        &= \frac{\alpha}{2}U + \beta \bw^\top \bV + \frac{\frac{\alpha^2}{4\sqrt n} U^2 + \frac{\alpha\beta}{\sqrt n} U \bw^\top \bV +\beta^2\sqrt n \left(\frac{||\bV||^2}{n} -1\right)}{\sqrt{\frac{\alpha^2}{4 n} U^2 + \frac{\alpha\beta}{ n} U \bw^\top \bV +\beta^2\frac{||\bV||^2}{n}} + \beta}.
    \end{align*}
    Since $U \sim \cN(0,1)$, $\bw ^\top \bV \sim \cN(0,1)$, $\|\bV\|^2 \sim \chi^2_n$, from their concentration inequalities, one can, in particular, conclude that 
    \begin{align*}
        U &= O_\prec (1),\\
        \bw ^\top \bV &= O_\prec (1),\\
        \sqrt{n} \left(\frac{\|\bV\|^2}{n} - 1\right) &= O_\prec(1).
    \end{align*}
    Hence
    \begin{align*}
        \sqrt{\frac{n}{r-1}}\left(\lambda_1(P_n') - \sqrt{r-2} \right) &= O_\prec(1),
    \end{align*}
    or, in other words,
    \begin{equation*}
        \lambda_1(P_n') =  \sqrt{r-2} + O_\prec\left(\frac{1}{\sqrt{n}}\right).
    \end{equation*}
    Similarly, one proves \eqref{eq:concen_P_n_lambda_n}.
\end{proof}

\begin{proof}[Proof of Lemma~\ref{lem:cov-structure-mixed-4}]
    Unlike Lemma~\ref{lem:cov-structure-pure}, the proof involves direct brute-force computations. We only prove \eqref{eq:var-1-mixed} and \eqref{eq:cov-5-mixed} here. Equations \eqref{eq:var-2-mixed}, \eqref{eq:cov-1-mixed}, \eqref{eq:cov-2-mixed}, \eqref{eq:cov-3-mixed} and \eqref{eq:cov-4-mixed} can proved similarly. Note that
    \begin{align*}
        \Var&(M_{ii})\\
        &= \Var\bigg(\sum_{i_3,\, i_4}\cG_{iii_3i_4}u_{i_3}v_{i_4}\bigg)\\
        &= \sum_{i_3, i_4, i'_3, i'_4}\bbE [\cG_{iii_3i_4} \cG_{iii'_3i'_4}]u_{i_3}v_{i_4}u_{i'_3}v_{i'_4}\\
        &= \sum_{i_3\ne i_4\ne i} \frac{4}{4!/(2!1!1!)} (u^2_{i_3} v^2_{i_4} + u_{i_3}u_{i_4}v_{i_3}v_{i_4}) + \sum_{i_3=i_4\ne i} \frac{4}{4!/(2!2!)} u^2_{i_3} v^2_{i_3}\\
        &\qquad\qquad+ \sum_{i_4 \ne i} \frac{4}{4!/(3!1!)}(u_iv_{i_4}u_iv_{i_4} + u_iv_{i_4}u_{i_4}v_i) \\
        &\qquad\qquad\qquad\qquad+ \sum_{i_3 \ne i} \frac{4}{4!/(3!1!)}(u_{i_3}v_iu_{i_3}v_i + u_{i_3}v_iu_iv_{i_3}) + \frac{4}{4!/4!}u^2_i v^2_i\\
        &= \frac{1}{3}\sum_{i_3 \ne i} u^2_{i_3}(1-v^2_{i_3} - v^2_i) + \frac{1}{3}\sum_{i_3 \ne i}u_{i_3}v_{i_3}(\bu^\top\bv - u_{i_3}v_{i_3}-u_iv_i) \\
        &\qquad\qquad+ \frac{2}{3}\sum_{i_3\ne i} u^2_{i_3} v^2_{i_3} + u^2_i(1-v^2_i) \\
        &\qquad\qquad\qquad\qquad+ v^2_i(1-u^2_i) + 2u_iv_i(\bu^\top \bv - u_iv_i) + 4u^2_i v^2_i\\
        &= \frac{1}{3}(1 + (\bu^\top \bv)^2) + \frac{4}{3} (\bu^\top \bv) u_i v_i + \frac{2}{3} (u_i^2 + v_i^2) + \frac{2}{3} u_i^2 v_i^2.
    \end{align*}
    This proves \eqref{eq:var-1-mixed}. Now we prove \eqref{eq:cov-5-mixed}.
    \begin{align*}
        \Cov &(M_{ii}, M_{il})\\
        &= \Cov\bigg(\sum_{i_3, i_4} \cG_{iii_3 i_4}u_{i_3}v_{i_4}, \sum_{i'_3, i'_4}\cG_{ili'_3 i'_4}u_{i'_3}v_{i'_4}\bigg)\\
        &= \sum_{i_3, i_4, i'_3, i'_4} \Cov(\cG_{iii_3i_4}, \cG_{ili'_3i'_4})u_{i_3}v_{i_4}u_{i'_3}v_{i'_4}\\
        &= \sum_{i_4}\Var(\cG_{iili_4})u_lu_iv^2_{i_4} + \sum_{i_4} \Var(\cG_{iili_4})u_lv_{i_4}u_{i_4}v_i \\ 
        &\qquad+ \sum_{i_3}\Var(\cG_{iii_3l})u_{i_3}v_lu_iv_{i_3}
        + \sum_{i_3} \Var(\cG_{iii_3l}) u_{i_3}v_lu_{i_3}v_i\\ 
        &\qquad- [\Var(\cG_{iill})u_lu_iv^2_l + \Var(\cG_{iill})u^2_lv_iv_l + \Var(\cG_{iiil})u^2_iv_iv_l + \Var(\cG_{iiil})u_iu_lv^2_i]\\
        &= \bigg[\frac{4}{4!/(2!1!1!)}u_iu_l\sum_{i_4 \ne i, l} v^2_{i_4} + \frac{4}{4!/(3!1!)}u_iu_lv^2_i + \frac{4}{4!/(2!2!)}u_iu_lv^2_i \bigg]\\
        &\qquad+ \bigg[\frac{4}{4!/(2!1!1!)}u_lv_i\sum_{i_4 \ne i, l}u_{i_4}v_{i_4}  + \frac{4}{4!/(3!1!)}u_iu_lv^2_i + \frac{4}{4!/(2!2!)}u^2_lv_iv_l \bigg]\\
        &\qquad+ \bigg[\frac{4}{4!/(2!1!1!)}u_iv_l\sum_{i_3 \ne i, l} u_{i_3}v_{i_3}  + \frac{4}{4!/(3!1!)}u^2_iv_iv_l + \frac{4}{4!/(2!2!)}u_iu_lv^2_l \bigg]\\
        &\qquad+ \bigg[\frac{4}{4!/(2!1!1!)}v_iv_l\sum_{i_3 \ne i, l}u_{i_3}u_{i_3}  + \frac{4}{4!/(3!1!)}u^2_iv_iv_l + \frac{4}{4!/(2!2!)}u^2_lv_iv_l \bigg]\\
        &\qquad- \bigg[\frac{4}{4!/(2!2!)}u_lu_iv^2_l + \frac{4}{4!/(2!2!)}u^2_lv_iv_l + \frac{4}{4!/(3!1!)}u^2_iv_iv_l + \frac{4}{4!/(3!1!)}u_iu_lv^2_i\bigg]\\
        &=\bigg[\frac{1}{3}u_iu_l(1-v^2_i-v^2_l) + u_iu_lv^2_i + \frac{2}{3}u_iu_lv^2_i \bigg] \\
        &\qquad+ \bigg[\frac{1}{3}u_lv_i(\bu^\top\bv - u_iv_i-u_lv_l)  + u_iu_lv^2_i + \frac{2}{3}u^2_lv_iv_l \bigg]\\
        &\qquad+ \bigg[\frac{1}{3}u_iv_l(\bu^\top\bv - u_iv_i-u_lv_l)  + u^2_iv_iv_l + \frac{2}{3}u_iu_lv^2_l \bigg] \\
        &\qquad+ \bigg[\frac{1}{3}v_iv_l(1-u^2_i-u^2_l)  + u^2_iv_iv_l + \frac{2}{3}u^2_lv_iv_l \bigg]\\
        &\qquad- \bigg[\frac{2}{3}u_lu_iv^2_l + \frac{2}{3}u^2_lv_iv_l + u^2_iv_iv_l + u_iu_lv^2_i\bigg]\\
        &= \frac{1}{3}(u_iu_l + v_iv_l) + \frac{1}{3}(\bu^\top\bv) (u_iv_l + u_lv_i) + \frac{1}{3}u_iv_i(u_iv_l + u_lv_i).
    \end{align*}
     This completes the proof.
\end{proof}

\begin{proof}[Proof of Lemma~\ref{lem:cov-representation-mixed-4}]
    One can directly verify that the covariance structure of $\cG\cdot\bu\otimes \bv$, which is available from Lemma~\ref{lem:cov-structure-mixed-4}, matches that of the matrix in the right-hand side of \eqref{eq:mixed_low_rank}.
\end{proof}

\begin{proof}[Proof of Lemma~\ref{lem:cov-mat-eval}]
    Let
    \[
        X_n = \alpha U \bu \bu^\top + \beta (\bu\bV^\top + \bV\bu^\top) + \theta Z_n.
    \]
    Then
    \[
        \vec(X_n) = \alpha U \vec(\bu\bu^\top) + \beta \vec(\bu\bV^\top + \bV\bu^\top) + \theta \vec(Z_n).
    \]
    By the Kronecker-product identity
    \[
        \vec(ABC) = (C^\top \otimes A) \vec(B),
    \]
    and the fact that for a vector $\bx$, $\vec(\bx) = \vec(\bx^\top) = \bx$,
    we have
    \begin{align*}
        \vec(X_n) = \alpha U (I \otimes \bu) \bu  + \beta (I \otimes \bu + \bu \otimes I) \bV + \theta \vec(Z_n).
    \end{align*}
    Therefore
    \begin{align*}
        &\Cov(\vec(X_n)) \\
        &\quad= \alpha^2 (I\otimes \bu) \bu \bu^\top (I \otimes \bu^\top) + \beta^2 (I \otimes \bu + \bu \otimes I) (I \otimes \bu + \bu \otimes I)^\top + \theta^2 \Cov(\vec(Z_n)) \\
        &\quad= \alpha^2 \bu \bu^\top \otimes \bu \bu^\top + \beta^2 (I \otimes \bu + \bu \otimes I) (I \otimes \bu + \bu \otimes I)^\top + \theta^2 \Cov(\vec(Z_n)).
    \end{align*}
    Now
    \[
        \Sigma_{\bu, \bu} = \Cov(\vech(X_n)) = L \Cov(\vec(X_n)) L^\top,
    \]
    where we recall that $L$ is the row-orthogonal elimination matrix. An immediate consequence of this is that
    \[
        \Sigma_{\bu, \bu} = L \Cov(\vec(X_n)) L^\top \succcurlyeq \theta^2 L \Cov(\vec(Z_n)) L^\top,
    \]
    the latter being a diagonal matrix with entries either $2$ or $1$. Thus
    \[
        \lambda_{\min}(\Sigma_{\bu, \bu}) \ge \theta^2 = \frac{1}{3}.
    \]
    Also,
    \[
        \lambda_{\max}(\Sigma_{\bu, \bu}) \le \alpha^2 + \beta^2 \lambda_{\max}((I \otimes \bu + \bu \otimes I) (I \otimes \bu + \bu \otimes I)^\top) + 2 \theta^2.
    \]
    Here we are using the fact that for a PSD matrix $A$,
    \[
        \lambda_{\max}(LAL^\top) = \|LAL^\top\|_{\op} \le \|A\|_{\op} \|L\|_{\op}^2 = \lambda_{\max}(A) \lambda_{\max}(LL^T) = \lambda_{\max}(A).  
    \]
    We claim that
    \[
        \lambda_{\max}((I \otimes \bu + \bu \otimes I) (I \otimes \bu + \bu \otimes I)^\top) = 4.
    \]
    In fact this matrix has $(n - 1)$ eigenvalues equal to $2$ and $1$ eigenvalue equal to $4$. This is because the non-zero eigenvalues of this matrix equal those of
    \[
        (I \otimes \bu + \bu \otimes I)^\top (I \otimes \bu + \bu \otimes I) = 2 (I + \bu\bu^\top). 
    \]
    Thus
    \[
        \lambda_{\max}(\Sigma_{\bu, \bu}) \le \alpha^2 + 4\beta^2  + 2 \theta^2 = 4.
    \]
    This completes the proof.
\end{proof}

To prove Lemma~\ref{lem:frob_norm}, it will be more convenient to work with a slightly different form of the covariance matrix.
\begin{definition}[Covariance matrix of Random Matrices]
    Given two random matrices $A$ and $B$, their covariance matrix is defined as 
    \[
        \Cov(A,B) := \bbE[(A - \bbE A)\otimes(B - \bbE B)^\top].
    \]
\end{definition}
In the same spirit, we define $\Var(A) := \Cov(A,A)$. Notice that this coincides with the standard definition when $A$ and $B$ are random vectors. Let
\begin{align*}
    &\tilde{\Sigma}_{\bu, \bv} := \Var(\cG\cdot \bu \otimes \bv),  &\tilde \Sigma_{\bu, \bu} := \Var(\cG\cdot \bu \otimes \bu).
\end{align*}

Recall that
\begin{align}\label{eq:mixed_low_rank}
    \cG \cdot \bu \otimes \bv &\overset{d}{=} \frac{1}{\sqrt{6}} U (\bu \bv^\top + \bv \bu^\top) \\ \nonumber
    &\qquad\qquad + \frac{1}{\sqrt{6}} (\bV_1 \bu^\top + \bu \bV_1^\top + \bV_2 \bv^\top + \bv \bV_2^\top) + \frac{1}{\sqrt{6}} \sqrt{1 + \langle \bu, \bv \rangle^2} Z_n.
\end{align}

\begin{proof}[Proof of Lemma~\ref{lem:frob_norm}]
    Clearly, $\| \Sigma_{\bu, \bv} - \Sigma_{\bu, \bu}\|_{\mathrm{F}} \leq \frac{1}{2}\|\tilde \Sigma_{\bu, \bv} - \tilde\Sigma_{\bu, \bu}\|_{\mathrm{F}}$. 
   We shall show that $\|\tilde \Sigma_{\bu, \bv} - \tilde\Sigma_{\bu, \bu}\|_{\mathrm{F}} \le 5 n \|\bu - \bv\|$. From \eqref{eq:mixed_low_rank}, using the independence of $\bu$, $\bv$ and $Z_n$ we get that
    \begin{align*}\label{eq:Sigma_uv}
        6\Sigma_{\bu, \bv} =& (\bu \bv^\top + \bv \bu^\top)\otimes(\bu \bv^\top + \bv \bu^\top)\\
        &+ \bbE[(\bu \bV_1^\top + \bV_1 \bu^\top + \bv \bV_2^\top + \bV_2 \bv^\top)\otimes(\bu \bV_1^\top + \bV_1 \bu^\top + \bv \bV_2^\top + \bV_2 \bv^\top)]\\
        &+ (1+\rho^2) \bbE[Z_n \otimes Z_n].
    \end{align*}
    Call the three terms in RHS $6\Sigma^{(1)}_{\bu, \bv}$, $6\Sigma^{(2)}_{\bu, \bv}$ and $6\Sigma^{(3)}_{\bu, \bv}$, respectively. Now we consider 
    \[
        6(\Sigma_{\bu, \bv} - \Sigma_{\bu,\bu}) = 6\sum_{i=1}^3 (\Sigma^{(i)}_{\bu, \bv} - \Sigma^{(i)}_{\bu, \bu}).
    \]
    Now, we look at these three differences separately.
    
    \textbf{The first term.} Notice that
    \begin{align*}
        (\bu &\bv^\top + \bv \bu^\top)\otimes(\bu \bv^\top + \bv \bu^\top) \\
        &= (\bu \otimes \bv^\top + \bv \otimes\bu^\top)\otimes(\bu \otimes \bv^\top + \bv \otimes \bu^\top)\\
        &= \bu\otimes \bv^\top \otimes \bu \otimes \bv^\top +\bu\otimes \bv^\top \otimes \bv \otimes \bu^\top + \bv \otimes \bu^\top \otimes \bu \otimes \bv^\top + \bv \otimes \bu^\top \otimes \bv \otimes \bu^\top.
    \end{align*}
    So,
    \begin{align*}
        6(\Sigma^{(1)}_{\bu, \bv} - \Sigma^{(1)}_{\bu, \bu}) =& (\bu\otimes \bv^\top \otimes \bu \otimes \bv^\top - \bu\otimes \bu^\top \otimes \bu \otimes \bu^\top)\\
        &+(\bu\otimes \bv^\top \otimes \bv \otimes \bu^\top - - \bu\otimes \bu^\top \otimes \bu \otimes \bu^\top) \\
        &+ (\bv \otimes \bu^\top \otimes \bu \otimes \bv^\top - \bu\otimes \bu^\top \otimes \bu \otimes \bu^\top) \\
        &+ (\bv \otimes \bu^\top \otimes \bv \otimes \bu^\top - \bu\otimes \bu^\top \otimes \bu \otimes \bu^\top).
    \end{align*}
    We look at the first term. 
    \begin{align*}
        \|\bu \otimes \bv^\top \otimes \bu \otimes \bv^\top &- \bu \otimes \bu^\top \otimes \bu \otimes \bu^\top\|_{\mathrm{F}}\\
        &\le \|\bu \otimes (\bv-\bu)^\top \otimes \bu \otimes \bv^\top\|_{\mathrm{F}} + \|\bu \otimes \bu^\top \otimes \bu \otimes (\bv-\bu)^\top\|_{\mathrm{F}}\\
        &= \|\bu\| \|\bv-\bu\| \|\bu\| \|\bv\| + \|\bu\|\| \bu\| \|\bu\| \|\bv-\bu\|\\
        &= 2\|\bv - \bu\|.
    \end{align*}
    Similarly, we can bound the other three terms. Hence,
    \[
        \|\Sigma^{(1)}_{\bu, \bv} - \Sigma^{(1)}_{\bu, \bu}\|_{\mathrm{F}} \le \frac{4}{3}\|\bv - \bu\|. 
    \]
    \textbf{The second term:} Notice that
    \begin{align*}
        \bbE[\,\,\cdot \otimes \bV_1 \otimes \cdot \otimes \bV_1] &= \sum_{i=1}^n \,\, \cdot \otimes \be_i^\top \otimes \cdot \otimes \be_i^\top,\\
        \bbE[\,\,\cdot \otimes \bV_1 \otimes \cdot \otimes \bV_2] &= \rho\sum_{i=1}^n \,\, \cdot \otimes \be_i^\top \otimes \cdot \otimes \be_i^\top.
    \end{align*}
    With this in mind,
    \begin{align*}
        &6(\Sigma^{(2)}_{\bu, \bv} - \Sigma^{(2)}_{\bu, \bu}) \\
        &=\bigg[\sum_{i=1}^n (\bv \otimes \be_i^\top \otimes \bv \otimes \be_i^\top - \bu \otimes \be_i^\top \otimes \bu \otimes \be_i^\top ) \\
        &\quad\qquad\qquad + \sum_{i=1}^n (\be_i \otimes \bv^\top \otimes \be_i \otimes \bv^\top - \be_i \otimes \bu^\top \otimes \be_i \otimes \bu^\top)\\
        &\qquad\qquad\qquad+ \sum_{i=1}^n (\bv \otimes \be_i^\top \otimes \be_i \otimes \bv^\top - \bu \otimes \be_i^\top \otimes \be_i \otimes \bu^\top ) \\
        &\qquad\qquad\qquad+ \sum_{i=1}^n (\be_i \otimes \bv^\top \otimes \bv \otimes \be_i^\top - \be_i \otimes \bu^\top \otimes \bu \otimes \be_i^\top)\bigg]\\
        &+ \bigg[\sum_{i=1}^n (\rho \bu \otimes \be_i^\top \otimes \bv \otimes \be_i^\top-  \bu \otimes \be_i^\top \otimes \bu \otimes \be_i^\top ) \\
        &\qquad\qquad\qquad+ \sum_{i=1}^n (\rho \bv \otimes \be_i^\top \otimes \bu \otimes \be_i^\top - \bu \otimes \be_i^\top \otimes \bu \otimes \be_i^\top) \\
        &\qquad\qquad\qquad+ \sum_{i=1}^n (\rho \bu \otimes \be_i^\top \otimes \be_i \otimes \bv^\top - \bu \otimes \be_i^\top \otimes \be_i \otimes \bu^\top) \\
        &\qquad\qquad\qquad+ \sum_{i=1}^n (\rho \be_i\otimes \bv^\top\otimes \bu \otimes \be_i^\top- \be_i\otimes \bu^\top\otimes \bu \otimes \be_i^\top) \\
        &\qquad\qquad\qquad+ \sum_{i=1}^n (\rho \be_i \otimes \bu^\top\otimes \bv \otimes \be_i^\top - \be_i \otimes \bu^\top\otimes \bu \otimes \be_i^\top) \\
        &\qquad\qquad\qquad+ \sum_{i=1}^n (\rho \bv \otimes \be_i^\top\otimes \be_i \otimes \bu^\top- \bu \otimes \be_i^\top\otimes \be_i \otimes \bu^\top) \\
        &\qquad\qquad\qquad+ \sum_{i=1}^n (\rho \be_i\otimes \bu^\top\otimes \be_i \otimes \bv^\top - \be_i\otimes \bu^\top\otimes \be_i \otimes \bu^\top) \\
        &\qquad\qquad\qquad+ \sum_{i=1}^n (\rho \be_i \otimes \bv^\top\otimes \be_i \otimes \bu^\top - \be_i \otimes \bu^\top\otimes \be_i \otimes \bu^\top)\bigg]. 
    \end{align*}
    Notice that the first term can be bounded as follows
    \begin{align*}
        \|\bv \otimes \be_i^\top \otimes \bv \otimes \be_i^\top - \bu \otimes \be_i^\top \otimes \bu \otimes \be_i^\top\|_{\mathrm{F}} \leq 2\|\bv -\bu\|. 
    \end{align*}
    and the fifth term can be bounded as follows
    \begin{align*}
        \|\rho \bu \otimes \be_i^\top \otimes \bv &\otimes \be_i^\top -  \bu \otimes \be_i^\top \otimes \bu \otimes \be_i^\top\|_{\mathrm{F}}\\
        &\le \rho\|\bu \otimes \be_i^\top \otimes (\bv - \bu) \otimes \be_i^\top\|_{\mathrm{F}} + (1-\rho) \|\bu \otimes \be_i^\top \otimes \bu \otimes \be_i^\top\|_{\mathrm{F}}\\
        &= \rho \|\bv -\bu\| + \frac{1}{2} \|\bv -\bu\|^2\\
        &\le 2\|\bv -\bu\|.
    \end{align*}
    Similar calculations for the other terms yield that 
    \[
        \|\Sigma^{(2)}_{\bu, \bv} - \Sigma^{(2)}_{\bu, \bu}\|_{\mathrm{F}} \leq 4n\|\bv -\bu\|.
    \]
    \textbf{The third term:} Notice that 
    \[
        \bbE [Z_n \otimes Z_n] = I_{n^2} + \sum_{i=1}^n \be_i \otimes \be^\top_i \otimes \be_i \otimes \be_i^\top + \sum_{1 \leq i \neq j \leq n} \be_i \otimes \be^\top_j \otimes \be_j \otimes \be^\top_i.
    \]
    So, $\|\bbE [Z_n \otimes Z_n]\|_{\mathrm{F}}^2 = 2n + 2n(n-1) = 2n^2$. Notice that $2 - 2\rho = \|\bu\|^2 + \|\bv\|^2 - 2 \bu^\top \bv = \|\bu - \bv\|^2$. Then,
    \begin{align*}
        (\Sigma^{(3)}_{\bu, \bv} - \Sigma^{(3)}_{\bu, \bu}) &=\frac{1}{6}\|2\bbE[Z_n \otimes Z_n] - (1+\rho^2)\bbE[Z_n \otimes Z_n]\|_{\mathrm{F}}\\
        &= \frac{1}{6}(1-\rho^2) \|\bbE [Z_n \otimes Z_n]\|_{\mathrm{F}}\\
        &= \frac{1}{6}(1+\rho)(1-\rho)\sqrt 2 n \\
        &\leq \frac{\sqrt 2}{6} n\|\bv - \bu\|^2\\
        &\leq \frac{\sqrt 2}{3} n\|\bv - \bu\|.
    \end{align*}
    Hence, combining the previous estimates we get that
    \[
        \|\tilde{\Sigma}_{\bu, \bv} - \tilde{\Sigma}_{\bu, \bu}\|_{\mathrm{F}} \leq 5n \|\bv -\bu\|.
    \]
    This completes the proof.
\end{proof}

\begin{proof}[Proof of Lemma~\ref{lem:interpolation}]
    We will use an interpolation argument. Set $\bw = \bv - \bu$ and for $\alpha \in [0, 1]$, let
    \[
        \bx(\alpha) = \bu + \alpha (\bv - \bu) = \bu + \alpha \bw. 
    \]
    Let
    \begin{align*}
        M_n(\alpha) := \frac{1}{\sqrt{6}} U (\bu \bx(\alpha)^\top + \bx(\alpha)\bu^\top) &+ \frac{1}{\sqrt{6}} [\bV_1 \bu^\top + \bu \bV_1^\top + \bV_2 \bx(\alpha)^\top + \bx(\alpha) \bV_2^\top] \\
        &+ \frac{1}{\sqrt{6}} \sqrt{1 + (\bu^\top \bx(\alpha)^2)} Z_n.
    \end{align*}
    Then
    \[
        \cG \cdot \bu \otimes \bu \overset{d}{=} M_n(0) \quad \text{and} \quad \cG \cdot \bu \otimes \bv \overset{d}{=} M_n(1).
    \]
    Let
    \[
        f(\alpha) = \Cov(\bs^\top M_n(\alpha) \bs, \bt^\top M_n(\alpha) \bt).
    \]
    Then
    \[
        f(1) - f(0) = \int_0^1 f'(\alpha) \, d\alpha.
    \]
    Now 
    \[
        f'(\alpha) = \Cov(\bs^\top \frac{d}{d\alpha}M_n(\alpha) \bs, \bt^\top M_n(\alpha) \bt) + \Cov(\bs^\top M_n(\alpha) \bs, \bt^\top M'(\alpha) \bt).
    \]
    Notice that
    \begin{align*}
        \frac{d}{d\alpha}M_n(\alpha) &= \frac{1}{\sqrt{6}} U (\bu \bw^\top + \bw \bu^\top) + \frac{1}{\sqrt{6}}[\bV_2 \bw^\top + \bw \bV_2^\top] \\
        &\qquad\qquad\qquad\qquad\qquad+ \frac{1}{\sqrt{6}} \frac{\bu^\top \bx(\alpha)}{\sqrt{1 + (\bu^\top \bx(\alpha))^2}} (\bu^\top \bw) Z_n.
    \end{align*}
    Therefore
    \begin{align*}
        \Cov(\bs^\top \frac{d}{d\alpha}M_n(\alpha) \bs, \bt^\top M_n(\alpha) \bt) &= \frac{2}{3} (\bs^\top \bu) (\bs^\top \bw) (\bt^\top \bu)(\bt^\top \bx(\alpha)) \\
        &\qquad+ \frac{2}{3}[(\bs^\top\bw)(\bt^\top \bu) \rho (\bs^\top \bt) + (\bs^\top \bw) (\bt^\top \bx(\alpha) (\bs^\top \bt)] \\
        &\qquad\qquad+ \frac{1}{3} (\bu^\top \bw) (\bs^\top \bt)^2 (\bu^\top \bx(\alpha)).
    \end{align*}
    It follows that
    \[
        |\Cov(\bs^\top \frac{d}{d\alpha}M_n(\alpha) \bs, \bt^\top M_n(\alpha) \bt)| \le C \|\bw\|_2,
    \]
    for some universal constant $C > 0$. Similarly,
    \[
        |\Cov(\bs^\top M_n(\alpha) \bs, \bt^\top \frac{d}{d\alpha}M_n(\alpha) \bt)| \le C \|\bw\|_2.
    \]
    Therefore
    \[
        |f(1) - f(0)| \le 2 C \|\bu - \bv\|_2.
    \]
    This completes the proof.
\end{proof}

\begin{proof}[Proof of Corollary~\ref{cor:mixed_limit_largest_evalue}]
    Since $\|\bu - \bv\|_{2} = o(1)$, we have
    \[
        \limsup_{n \to \infty} |\bbE \lambda_1(n^{-1/2} \cG \cdot \bu \otimes \bv) - \bbE\lambda_1(n^{-1/2} \cG \cdot \bu \otimes \bu)| \le C_1 \varepsilon.
    \]
    Letting $\varepsilon \to 0$, we get
    \begin{equation}\label{eq:some_limit_1}
        \lim_{n \to \infty} |\bbE \lambda_1(n^{-1/2} \cG \cdot \bu \otimes \bv) - \bbE\lambda_1(n^{-1/2} \cG \cdot \bu \otimes \bu)| = 0.
    \end{equation}
    Now, as a consequence of Lemma~\ref{lem:concentration},
    \begin{equation}\label{eq:some_limit_2}
        \lambda_1(n^{-1/2} \cG \cdot \bu \otimes \bv) - \bbE \lambda_1(n^{-1/2} \cG \cdot \bu \otimes \bv) \convas 0.
    \end{equation}
    Taking $\bv = \bu$ in the previous equation yields
    \[
        \lambda_1(n^{-1/2} \cG \cdot \bu \otimes \bu) - \bbE \lambda_1(n^{-1/2} \cG \cdot \bu \otimes \bu) \convas 0.
    \]
    Since $\lambda_1(n^{-1/2} \cG \cdot \bu \otimes \bu) \convp \varpi_4$, we have 
    \begin{equation}\label{eq:some_limit_3}
    \bbE \lambda_1(n^{-1/2} \cG \cdot \bu \otimes \bu) \to \varpi_4.
    \end{equation}
    
    Combining \eqref{eq:some_limit_1}, \eqref{eq:some_limit_2} and \eqref{eq:some_limit_3} we finally get that
    \[
        \lambda_1(n^{-1/2} \cG \cdot \bu \otimes \bv) \convas \varpi_4.
    \]
    This completes the proof.
\end{proof}
\end{document}